\documentclass[12pt]{amsart}
\usepackage{graphicx}
\usepackage{amsmath}
\usepackage{amssymb}
\usepackage{amsthm}
\usepackage{tikz}
\usepackage{tikz-cd}
\usepackage{comment}
\usepackage[utf8]{inputenc}
\usepackage{comment}
\usepackage{hyperref}
\usepackage[normalem]{ulem}

\makeatletter
\@namedef{subjclassname@2020}{%
  \textup{2020} Mathematics Subject Classification}
\makeatother

\usepackage{cite}

\newtheorem{theorem}{Theorem}[section]
\newtheorem{lemma}[theorem]{Lemma}

\newtheorem{corollary}[theorem]{Corollary}
\newtheorem{proposition}[theorem]{Proposition}
\newtheorem{observation}[theorem]{Observation}

\theoremstyle{definition}
\newtheorem{definition}[theorem]{Definition}
\newtheorem{question}[theorem]{Question}
\newtheorem{example}[theorem]{Example}
\newtheorem{remark}[theorem]{Remark}
\newtheorem{problem}[theorem]{Problem}
\newtheorem{notation}[theorem]{Notation}

\makeatletter
\newcommand{\customlabel}[2]{%
\protected@write \@auxout {}{\string \newlabel {#1}{{#2}{}}}}
\makeatother

\newcommand{\iLim}{\varprojlim}
\newcommand{\Fl}{Fra\"{\i}ss\'e limit }
\newcommand{\F}{Fra\"{\i}ss\'e  }

\newlength{\wdth}

\DeclareMathOperator{\card}{card}
\DeclareMathOperator{\diam}{diam}
\DeclareMathOperator{\bd}{bd}
\DeclareMathOperator{\ord}{ord}
\DeclareMathOperator{\len}{len}

\DeclareMathOperator{\id}{id}

\def\g{\mathcal{G}}
\def\gg{\mathbb{G}}
\newcommand{\fra}{Fra\"\i ss\'e}

\newcommand{\blu}[1]{{\color{blue}{#1}}}

\renewcommand{\blu}[1]{#1}

\begin{document}

\title[Projective Fra\"{\i}ss\'{e} limits, graphs, and confluent maps]{The Projective Fra\"{\i}ss\'{e} limit of the family of all connected finite graphs with confluent epimorphisms}

\author[W. J. Charatonik]{W\l odzimierz J. Charatonik}

\address{W. J. Charatonik, Department of Mathematics and Statistics\\
         Missouri University of Science and Technology\\
         400 W 12th St\\
         Rolla MO 65409-0020}

\author[A. Kwiatkowska]{Aleksandra Kwiatkowska}
\address{A. Kwiatkowska, Institut f\"{u}r Mathematische Logik und Grundlagenforschung, Universit\"{a}t  M\"{u}nster,  
Einsteinstrasse 62,
48149  M\"{u}nster,
Germany {\bf{and}} 
Instytut Matematyczny, Uniwersytet Wroc{\l}awski,  pl. Grunwaldzki 2/4, 50-384 Wroc{\l}aw, Poland, }
\email{kwiatkoa@uni-muenster.de}

\author[R. P. Roe]{Robert P. Roe}
\address{R. P. Roe, Department of Mathematics and Statistics\\
         Missouri University of Science and Technology\\
         400 W 12th St\\
         Rolla MO 65409-0020}
\curraddr{Robert P. Roe, 1407 Athene Dr, Lafayette CO 80026 USA}
\email{rroe@mst.edu}
\thanks{Funded by the Deutsche Forschungsgemeinschaft (DFG, German Research Foundation) under Germany’s Excellence Strategy EXC 2044–390685587, Mathematics Münster: Dynamics–Geometry–Structure and by CRC 1442 Geometry: Deformations and Rigidity.}

\date{\today}


\begin{abstract}
We investigate the projective Fra\"{\i}ss\'e family of finite connected graphs with confluent epimorphisms and the continuum obtained as the topological realization of its projective Fra\"{\i}ss\'e limit.
This continuum was unknown before. We prove that it is indecomposable, but not hereditarily indecomposable, one-dimensional, Kelley, pointwise self-homeomorphic, but not homogeneous. It is hereditarily unicoherent and 
each point is the top of the Cantor fan.
 Moreover, the universal solenoid, the universal pseudo-solenoid, and the pseudo-arc may be embedded in it.

\end{abstract}

\maketitle

\section{Introduction}

In 2006, Irwin and Solecki \cite{Pseudo} initiated the study of projective Fra\"{\i}ss\'{e} limits creating a bridge between logic and continuum theory.  
In their article they constructed the pseudo-arc as the topological realization of the projective Fra\"{\i}ss\'{e} limit of the projective Fra\"{\i}ss\'{e} family  consisting of all finite linear (combinatorial) graphs and all epimorphisms, i.e., surjective mappings between the finite graphs that preserve the edge relation.

The study of projective Fra\"{\i}ss\'{e} limits has in recent years caught the attention of many researchers. They used properties of the mappings between finite structures, in particular finite graphs, to construct and investigate various compact one-dimensional metric spaces, such as the Lelek fan, generalized Wa\.zewski dendrites or Knaster continua, see for example \cite{B-C}, \cite{B-K},  \cite{wazewski-fraisse}, \cite{SI}, \cite{LW}.
Compact spaces were also studied in a framework  of an approximate Fra\"{\i}ss\'{e} theory, see for example
\cite{Bar-Kub} or \cite{Kub-Kwi}, where in particular, the Poulsen simplex or  pseudo-solenoids were realized as Fra\"{\i}ss\'{e} limits.

Finite graphs are discrete as topological spaces. Therefore they are topologically totally disconnected. Nevertheless, they can be connected as graphs.
In \cite{Menger}, Panagiotopoulos and Solecki introduced appropriate definitions for connectedness and for maps between topological graphs preserving connectedness,
which take into account both the topology and the edge relation. They showed that the family of finite connected graphs with monotone epimorphisms is a projective Fra\"{\i}ss\'{e} family whose topological realization is the Menger curve.
In \cite{WJC-RPR-Fraisse}, Charatonik, Kwiatkowska,  Roe, and Yang considered finite trees with epimorphisms, that are monotone, confluent, order-preserving, retractions, light, etc. As topological realizations of the various Fra\"{\i}ss\'{e} limits they obtained known continua, such as the Cantor fan and the generalized Wa\.zewski dendrite $D_3$, as well as previously unknown continua.

In the present
article, we study topological properties of the
 continuum obtained as the topological realization of the  projective Fra\"{\i}ss\'{e} limit of the family of finite
connected graphs with confluent epimorphisms.
 Confluent maps between continua were introduced by J.~J.~Charatonik \cite{JJC-Confluent} as a generalization of monotone maps and have been very well studied. Confluent maps between finite connected graphs are discretizations of such maps. 
We summarize our results in the following theorems.

\begin{theorem}\label{main-list}
The family ${\mathcal G}$ of finite connected graphs with confluent epimorphisms is a projective \F family. Its projective \F limit 
$\mathbb G$ has a transitive edge relation and its topological realization $|\mathbb G|$ is a continuum that  has the following properties:

\begin{enumerate}
    \item it is a Kelley continuum;
    \item it is one-dimensional;
     \item it is an indecomposable continuum;
    \item all arc components are dense; 
   \item each point is the top of the Cantor fan;
    \item it is pointwise self-homeomorphic;
    \item it is hereditarily unicoherent, in particular, the circle $S^1$ does not embed in it;
    \item the pseudo-arc, the universal pseudo-solenoid, and the universal solenoid embed in it.
  
\end{enumerate}
\end{theorem}

In Section \ref{nine}, we will show that topological realizations of projective Fra\"{\i}ss\'{e} limits of many projective Fra\"{\i}ss\'{e} families of cycles with confluent epimorphisms exist and are solenoids, see Theorem \ref{graphtosol}.

Since the \F limit of finite graphs with monotone bonding maps is the Menger universal curve, see \cite{Menger}, which is homogeneous, the question arises if the same is true with confluent maps. Understanding solenoids that can be embedded in the \F limit leads to a major and somewhat surprising result of the article, that the continuum $|\mathbb G|$ is not homogeneous, which we accomplish in Section \ref{nine}. 
Rogers \cite{Rogers} proved that
each homogeneous, hereditarily indecomposable curve (i.e. continuum of dimension 1) is
tree-like. Note that we cannot apply his theorem since $|\mathbb G|$
is not hereditarily indecomposable. 
To accomplish the non-homogeneity, we show that some points of  $|\mathbb G|$ belong to a solenoid and others do not. In particular, we show:

\begin{theorem}\label{mainabc}
Let ${\mathcal G}$, $\mathbb G$, and $|\mathbb G|$ be as in Theorem \ref{main-list}. Then the following hold.
\begin{enumerate}
\item Let $S_{univ}$ be the universal solenoid. The set 
$$\{x\in |\mathbb G|\colon \text{ there is an embedding } \tau\colon S_{univ}\to |\mathbb G| \text{ such that } x\in\tau(S_{univ})\} $$
is dense in $|\mathbb G|$.

\item The set 
$$\{y\in |\mathbb G|\colon \text{for every solenoid } S \text{ and every embedding } \tau\colon S\to |\mathbb G| \text{ we have } y\notin\tau(S)\} $$
is dense in $|\mathbb G|$.

\item The universal solenoid is the only solenoid which embeds in~$|\mathbb G|$.

\end{enumerate}

\end{theorem}

To prove Theorem \ref{mainabc}(2) in Section \ref{nine} we analyze topological subgraphs of $\mathbb{G}$ whose topological realizations are solenoids. These turn out to be not just inverse limits of finite cycles with confluent bonding maps. In fact, we have to deal with more general topological graphs,  which we call almost graph-solenoids.

Theorem \ref{mainabc} immediately yields the following.
\begin{corollary}
The continuum $|\mathbb G|$ is not homogeneous.

\end{corollary}

Speaking with researchers in continuum theory it appears that there is no previously known continuum having these properties. We do not know if the properties listed here characterize $|\mathbb G|$.

\section{Definitions and background results}

A {\it graph} is an ordered pair $A=( V(A), E(A))$, where $E(A)\subseteq V(A)^2$ is a reflexive and symmetric relation on $V(A)$. The elements of $V(A)$ are called {\it vertices} of graph $A$ and elements of $E(A)$ are called {\it edges}. 
For the sake of simplicity of notation we will assume that  $\langle x,y\rangle$ represents both the edge $\langle x,y\rangle$ and the edge $\langle y,x\rangle$. If $A$ and $B$ are graphs then $A\setminus B$ denotes the graph with $V(A\setminus B) = V(A)\setminus V(B)$ and $\langle x,y\rangle \in E(A\setminus B)$ if and only if $x,y \in V(A\setminus B)$ and $\langle x,y \rangle \in E(A)$. 
We say that a graph $A$ is a \textit{subgraph} of a graph $B$, and denote it by $A\subseteq B$, if  $V(A) \subseteq V(B)$ and 
$\langle x,y\rangle \in E(A)$ if and only if $\langle x,y\rangle \in E(B)$, for $x,y\in V(A)$.

A {\it topological graph} ${\bf K}$ is a graph $( V({\bf K}),E({\bf K}))$, whose domain $V({\bf K})$ is a 0-dimensional, compact, second-countable (thus has a metric) space and $E({\bf K})$ is a closed, reflexive and symmetric subset of $ V({\bf K})^2$. A topological graph is an example of a topological $\mathcal{L}$-structure, in the sense of Irwin-Solecki
\cite{Pseudo}. From now on we will often write ${\bf A}$ when we mean $V({\bf A})$, where ${\bf A}$ is a topological graph.

Given two topological graphs ${\bf A}$ and ${\bf B}$, a continuous function $f\colon {\bf A}\to {\bf B}$ is a {\it homomorphism} if it maps edges to edges, i.e.  $\langle {\bf a},{\bf b}\rangle \in E({\bf A})$ implies $\langle f({\bf a}),f({\bf b})\rangle \in E({\bf B})$.  
A homomorphism $f$ is an {\it epimorphism} if it is moreover surjective
on both vertices and edges. An {\it isomorphism} is an injective epimorphism. A homomorphism $f\colon {\bf A} \to {\bf B}$ is an {\it embedding} if it is an isomorphism onto a subgraph of ${\bf B}$.

The theory of projective Fra\"{\i}ss\'e limits was developed in \cite{Pseudo} and further refined in \cite{Menger}. We literally recall their definitions here.

\begin{definition}\label{definition-Fraisse}
Let $\mathcal{F}$ be a family of nonempty finite graphs with a fixed family of epimorphisms between
the graphs in $\mathcal{F}$.  We say that $\mathcal{F}$ is a {\it projective Fra\"{\i}ss\'e family } if
\begin{enumerate}
\item $\mathcal{F}$ is countable up to isomorphism, that is, any sub-collection of pairwise
non-isomorphic graphs of $\mathcal{F}$ is countable;
\item epimorphisms are closed under composition and each identity map is an epimorphism;
\item for $B,C \in  \mathcal{F}$ there exist $D\in \mathcal{F}$ and epimorphisms $f \colon D \to B$ and $g \colon D \to C$ in $\mathcal{F}$;
and
\item for $A,B,C\in \mathcal{F}$ and for every two epimorphisms $f \colon B \to A$ and $g \colon C \to A$ in $\mathcal{F}$, there exist $D \in \mathcal{F}$ and epimorphisms
$f_0 \colon D \to B$ and $g_0 \colon D\to C$ in $\mathcal{F}$ such that $f \circ f_0 = g \circ g_0$, i.e. the diagram (D1) commutes.
\end{enumerate}

\begin{equation}\tag{D1}
\begin{tikzcd}
&B\arrow{ld}[swap]{f}\\
A&&D\arrow[lu,swap,dotted,"f_0"] \arrow[ld,dotted,"g_0"]\\
&C\arrow[lu,"g"]
\end{tikzcd}
\end{equation}

The last property is known as the {\it projective amalgamation property}.
\end{definition}
Whenever we consider a projective Fra\"{\i}ss\'e family we identify isomorphic graphs.

\begin{definition}
Given $\varepsilon>0$, a map $f\colon G \to H$, where $G$ is a metric space with metric $d$, is said to be an {\em $\varepsilon$-map} if for each $x \in H$ the diameter of $f^{-1}(x)$ is less than  $ \varepsilon$. An $\varepsilon$-map which is an epimorphism is called an {\em $\varepsilon$-epimorphism}.

\end{definition}

\begin{notation}
Given sequences $F_n$, of finite graphs, and epimorphisms (homomorphisms, surjections) $\alpha_n\colon F_{n+1} \to F_n$ we denote the inverse sequence by $\{F_n,\alpha_n\}$. For $m > n$ we let $\alpha_n^m= \alpha_n\circ \dots \circ \alpha_{m-1}$ and note that $\alpha_n^{n+1} = \alpha_n$. Given an inverse sequence $\{F_n,\alpha_n\}$, the associated inverse limit space is the subspace of the product space $\Pi F_i$ determined by $\{(x_1,x_2,\ldots) \colon x_i \in F_i \text{ and } x_i = \alpha_i(x_{i+1})\}$ and is denoted as ${\bf F}=\iLim\{F_n,\alpha_n\}$.   Further, we denote the canonical projection from the inverse limit space ${\bf F}$ to the $n$th factor space $F_n$ by $\alpha_n^{\infty}$.
\end{notation}

The family $\mathcal F$ of finite graphs and epimorphisms is enlarged to a family $\mathcal F^\omega$ which includes all topological graphs obtained as inverse limits of graphs in $\mathcal F$ with bonding maps from the family of epimorphisms.  If ${\bf G}=\iLim\{G_n,\alpha_n\} \in \mathcal F^\omega$ and ${{\bf a}}=(a_n)$ and ${\bf{b}}=(b_n)$ are elements of ${\bf G}$ then $\langle {{\bf a}},{\bf b}\rangle$ is an edge in ${\bf G}$ if and only if for each $n$, $\langle a_n, b_n\rangle$ is an edge in $G_n$. An epimorphism $h$ between a topological graph ${\bf G}=\iLim\{G_n,\alpha_n\}$ in $\mathcal F^\omega$ and  $A\in \mathcal F$ is in the family $\mathcal F^\omega$ if and only if there is an $m$ and an epimorphism  $h'\colon G_m \to A$, $h'\in \mathcal F$, such that $h= h'\circ \alpha_m^\infty$. Note that this implies that each $\alpha^\infty _n$ is in $\mathcal F^\omega$. Finally, if ${\bf K}$ and ${\bf L}$ are in $\mathcal F^\omega$ an epimorphism $h\colon {\bf L} \to {\bf K}$ is in the family $\mathcal F^\omega$ if and only if for any  $A \in \mathcal F$ and any epimorphism $g\colon {\bf K} \to A$ in $\mathcal F^\omega$, $g\circ h \in \mathcal F^\omega$.

\begin{lemma}\label{ladder}
Let $\{F_n,\alpha_n\}$ be an inverse sequence of finite graphs, where each $\alpha_n$ is an epimorphism. Let ${\bf F}$ denote its inverse limit.
\begin{enumerate}
\item If $A$ is a finite graph and $f\colon {\bf F}\to A$ is an epimorphism, then there is $n$ and an epimorphism  $\gamma\colon F_n\to A$ such that $f=\gamma\circ\alpha^\infty_n$.
\item Suppose that $\{G_m,\beta_m\}$ is another inverse sequence of finite graphs, where each $\beta_m$ is an epimorphism, such that ${\bf F}$ is its inverse limit.
Then there are subsequences $(n_i)$ and $(m_j)$ and there are epimorphisms $g_i\colon F_{n_i}\to G_{m_i}$ and $h_i\colon G_{m_{i+1}}\to F_{n_i}$ such that
$\alpha^{n_{i+1}}_{n_i}=h_ig_{i+1}$ and $\beta^{m_{i+1}}_{m_i}=g_ih_i$.
\end{enumerate}

\end{lemma}

\begin{proof}

For (1), take any $n$ such that the partition $\{(\alpha^\infty_n)^{-1}(x)\colon x\in F_n\}$ refines the partition $\{f^{-1}(x)\colon x\in A\}$.
We let for $z\in F_n$, $\gamma(z)$ to be $f(t)$ for any $t\in (\alpha^\infty_n)^{-1}(z)$. Since $f$ and $\alpha^\infty_n$ are epimorphism, by Lemma 2.1 in \cite{Pseudo}, $\gamma$ is an epimorphism.
For (2), if we already constructed $(g_i)$ and $(h_i)$ for $i<i_0$, to get $g_{i_0}$ apply part (1) to $f_1=\beta^\infty_{m_{i_0}}$ and get $n_{i_0}$ and $g_{i_0}$ such that
$f_1=g_{i_0}\circ\alpha^\infty_{n_{i_0}}$. Next, we apply part (1) to $f_2=\alpha^\infty_{n_{i_0}}$ and get $m_{{i_0+1}}$ and $h_{i_0}$ such that 
$f_2=h_{i_0}\circ\beta^\infty_{m_{i_0+1}}$. Those $g_i$ and $h_i$ are as required.
\end{proof}

Note that we may replace epimorphism with surjection in the previous lemma. In the proof of Theorem \ref{notuniversal}, we will use the following statement.
\begin{lemma}\label{laddersur}
Let $\{F_n,\alpha_n\}$ be an inverse sequence of finite graphs, where each $\alpha_n$ is a surjection. Let ${\bf F}$ denote its inverse limit.
If $A$ is a finite graph and $f\colon {\bf F}\to A$ is surjection, then there is $n$ and a surjection  $\gamma\colon F_n\to A$ such that $f=\gamma\circ\alpha^\infty_n$.
\end{lemma}

\blu{Lemma \ref{ladder} motivates the following definition. Let $\mathcal F$ be a family of finite graphs with a possibly restricted family of epimorphisms. We say that $\mathcal{F}$ is {\it consistent} if for any topological graph ${\bf F}$ in $\mathcal F^\omega$, which can be represented by isomorphic inverse limits $\iLim \{F_n,\alpha_n\}$ and $\iLim \{G_n,\beta_n\} $, where $\alpha_n, \beta_n\in\mathcal F$, there are subsequences $(n_i)$ and $(m_j)$ and  epimorphisms $g_i\colon F_{n_i}\to G_{m_i}$, $h_i\colon G_{m_{i+1}}\to F_{n_i}$ with $h_i, g_i\in\mathcal F$ such that
$\alpha^{n_{i+1}}_{n_i}=h_ig_{i+1}$, $\beta^{m_{i+1}}_{m_i}=g_ih_i$,
$\beta^\infty_{m_{i}}=g_{i}\circ\alpha^\infty_{n_{i}}$,
and $\alpha^\infty_{n_{i}}=h_{i}\circ\beta^\infty_{m_{i+1}}$. 
 
The first two of these equalities show that $(h_i)$ induces an isomorphism $h_\infty \colon \iLim \{G_n,\beta_n\} \to \iLim \{F_n,\alpha_n\}$ (equivalently, $(g_i)$ induces isomorphism $g_\infty \colon \iLim \{F_n,\alpha_n\} \to \iLim \{G_n,\beta_n\})$ and the last two equalities show that $h_\infty$ (equivalently $g_\infty$) is the identity.}

\blu{The consistency of $\mathcal{F}$ guarantees that the definition of an epimorphism in $\mathcal{F}^\omega$ does not depend on the inverse limit representation of a structure. Indeed, suppose that $\iLim \{F_n,\alpha_n\}$ and $\iLim \{G_n,\beta_n\} $ are isomorphic, where $\alpha_n, \beta_n\in\mathcal F$, and let
 $f\colon \iLim\{F_n,\alpha_n\}\to A$ be an epimorphism in $\mathcal{F}^\omega$.
  Let $m$ and
$f'\colon F_m \to A$ with $f'\in \mathcal F$, be such that $f= f'\circ \alpha_m^\infty$. 
Take any $n_i\geq m$  and consider $h_i\colon G_{m_{i+1}}\to F_{n_i}$. Then $f= f'\circ \alpha^{n_i}_m\circ h_i\circ \beta^\infty_{m_{i+1}}$ and $f'\circ \alpha^{n_i}_m\circ h_i\in\mathcal F$.}

The following, from \cite{Menger} is an adaptation of the basic result that first appeared in \cite{Pseudo}, which allows us to work with subfamilies of epimorphisms.
We emphasize that in \cite{Pseudo} the authors considered more general structures known as topological $\mathcal L$-structures. Topological graphs, as defined here, are examples of topological $\mathcal L$-structures.

\begin{theorem}\label{limit}
 Let $\mathcal {F}$ be a projective Fra\"{\i}ss\'e family with a fixed collection of epimorphisms between graphs of $\mathcal{F}$. 
 Then there exists a unique topological graph  $\mathbb{F}$ in $\mathcal{F}^\omega$ such that
\begin{enumerate}
    \item for each $A\in \mathcal{F}$, there exists an epimorphism in $\mathcal{F}^\omega$
    from $\mathbb{F}$ onto~$A$;
    \item for $A,B \in \mathcal{F}$ and epimorphisms $f\colon \mathbb{F} \to A$ and $g\colon B\to A$
   in $\mathcal{F}^\omega$
    there exists an epimorphism $h\colon \mathbb{F}\to B$
  in $\mathcal{F}^\omega$
    such that $f=g\circ h$.
    \item\label{refinement}
For every $\varepsilon>0$  there is a graph $G\in \mathcal F$ and an $\varepsilon$-epimorphism $f\colon \mathbb{F} \to G$     in $\mathcal{F}^\omega$.

\end{enumerate}
\end{theorem}
The topological graph $\mathbb F$ from Theorem \ref{limit} is called the {\it projective Fra\"{\i}ss\'e limit} of $\mathcal F$.

\begin{definition}\label{fund-seq-def}

Given a projective \F family $\mathcal F$, an inverse sequence $\{F_n,\alpha_n\}$ where $F_n\in {\mathcal F}$ and $\alpha_n\colon F_{n+1} \to F_n$  are epimorphisms in $\mathcal{F}$, we say that  $\{F_n,\alpha_n\}$ is a {\it Fra\"{\i}ss\'e sequence} for $\mathcal F$ if the following two conditions hold.

\begin{enumerate}
    \item For any $G\in {\mathcal F}$ there is an $n$ and an epimorphism in $\mathcal{F}$ from $F_n$ onto $G$;
    \item For any $n$, any pair $G, H \in {\mathcal F}$, and any epimorphisms $g\colon H \to G$ and $f\colon F_n \to G$ in $\mathcal{F}$ there exists $m >n$ and an epimorphism $h\colon F_m \to H$ in $\mathcal{F}$ such that $g\circ h = f\circ \alpha_n^m$  i.e the diagram (D2) commutes.
\end{enumerate}

\begin{equation}\tag{D2}
\begin{tikzcd}
F_n\arrow{d}[swap]{f}&F_m\arrow[d,dotted,"h"]\arrow{l}[swap]{\alpha^m_n} \\
G&\arrow[l,"g"]H
\end{tikzcd}
\end{equation}

\end{definition}

The following follows from \cite[Theorem 2.4]{Pseudo}.

\begin{theorem}\label{fund-sequence}
Given a projective \F family $\mathcal F$, there exists a Fra\"{\i}ss\'e sequence $\{F_i,g_i\}$ for $\mathcal F$, and the inverse limit of any Fra\"{\i}ss\'e sequence for $\mathcal F$ is isomorphic to the projective \F limit $\mathbb F$.
\end{theorem}

The following result is Proposition~2.3 of \cite{B-C}.
\begin{proposition}\label{B-CThm}
Let $\mathcal F$ be a projective \F family. Let $\{A_n,\alpha_n\}$ be an inverse sequence in $\mathcal F$. Assume that for each $A \in \mathcal F$, $n \in \mathbb N$, and epimorphism $f\colon A \to A_n$ in $\mathcal{F}$, there exists $m \ge n$ and an epimorphism $g\colon A_m \to A$ in $\mathcal{F}$ such that $fg=\alpha^m_n$. Then $\{A_n,\alpha_n\}$ is a Fra\"{\i}ss\'e sequence for 
$\mathcal F$. The converse holds as well.
\end{proposition}

\begin{definition}\label{definition-order}
A finite graph  $T$ is a {\it tree} if  for every two distinct vertices $a,b \in T$ there is a unique finite sequence
$v_0=a,v_1, \dots , v_n=b$ of vertices in $T$ such that for every $i\in \{0,1,\dots , n-1\}$ we have $\langle v_i,v_{i+1}\rangle\in E(T)$
and  $v_i\ne v_{j}$ if $i\neq j$.  
Let $n$ be a natural number, a vertex $p\in T$ has {\it order } $n$ ($\ord(p)=n$) if there are exactly $n$ non-degenerate edges in $T$ that contain $p$. If $\ord(p)=1$ then $p$ is an {\it end-vertex}, if $\ord(p)=2$ then $p$ is an {\it ordinary vertex}, and if $\ord(p)\ge 3$ then $p$ a {\it ramification vertex}.
\end{definition}

\begin{definition}
Given a topological graph ${\bf K}$, if the edge relation $E({\bf K})$ is also a transitive relation, i.e. it is an equivalence relation, then ${\bf K}$ is known as a {\it prespace}. The quotient space ${\bf K}/E({\bf K})$ is compact metrizable and is called the {\it topological realization} of ${\bf K}$ which we will denote by $|{\bf K}|$. Throughout we will use $\pi$ to denote the quotient map $\pi\colon  {\bf K} \to |{\bf K}|$.
\end{definition}

\begin{lemma}\label{quotient-emb}
    Suppose that ${\bf K}$ and $  {\bf L}$ are topological graphs that have transitive edge relations and suppose that $f\colon   {\bf K}\to   {\bf L}$ is an embedding of graphs. Let $\pi_{\bf K}\colon  {\bf K}\to | {\bf K}|$ and $\pi_{\bf L}\colon {\bf L}\to |{\bf L}|$ be the quotient maps. Then the map $f^*\colon |{\bf K}|\to |{\bf L}|$ satisfying $f^*\circ\pi_{\bf K}=\pi_{\bf L}\circ f$ is an embedding.
\end{lemma}
\begin{proof}
    Let $U$ be open in $|{\bf L}|$. Then $\pi_{\bf L}^{-1}(U)$ is open by the definition of the quotient topology. By continuity of $f$, 
    $f^{-1}(\pi_{\bf L}^{-1}(U))$ is open in ${\bf L}$.
    However, $f^{-1}(\pi_{\bf L}^{-1}(U))=\pi_{\bf K}^{-1}((f^*)^{-1}(U))$. Therefore again by the definition of the quotient topology, $(f^*)^{-1}(U)$ is \blu{open in $|{\bf K}|$}.
\end{proof}

The following is Theorem 2.17 of \cite{WJC-RPR-Fraisse}. It  allows us to show, in some cases,  transitivity of the edge relation, and thus the existence of the topological realization.  
\begin{theorem}\label{transitive-edges}
Suppose that $\mathcal F$ is a projective \F family of graphs and for every $G\in \mathcal F$, for all pairwise different $a,b,c\in G$ such that 
$\langle a,b\rangle\in E(G)$ and $\langle b,c\rangle\in E(G)$ there is a graph $H$ and an epimorphism 
$f^H_G\colon H\to G$ in $\mathcal{F}$ such that for all vertices $p,q,r\in H$ such that $f^H_G(p)=a$, $f^H_G(q)=b$, and
$f^H_G(r)=c$ we have $\langle p,q\rangle\notin E(H)$ or $\langle q,r\rangle\notin E(H)$. Then the edge relation of the projective \Fl $\mathbb F$ of $\mathcal F$ is transitive.
Moreover, for any ${\bf a}\in \mathbb F$ there is at most one  ${\bf b}\in \mathbb F$, ${\bf b}\neq {\bf a}$, such that $\langle {\bf a},{\bf b}\rangle\in E(\mathbb F)$. 
\end{theorem}

Below we give definitions for connectedness properties for topological graphs analogous to similarly named properties for continua.  The first definition below originally appeared in \cite{Pseudo} and the remainder were introduced in \cite{WJC-RPR-Fraisse}.

\begin{definition}
  Given a topological graph ${\bf G}$, a subgraph ${\bf S}$ of ${\bf G}$ is {\it disconnected} if there are two nonempty disjoint closed subgraphs ${\bf P}$ and ${\bf Q}$ of ${\bf S}$ such that ${\bf P}\cup {\bf Q}={\bf S}$ and if
  ${\bf a}\in {\bf P}$ and ${\bf b}\in {\bf Q}$, then $\langle {\bf a},{\bf b}\rangle\notin E({\bf G})$. \blu{In this case we say the graphs ${\bf P}$ and ${\bf Q}$ form a {\it separation} of ${\bf S}$ and denote it by ${\bf S} = {\bf P}|{\bf Q}$.} A  subgraph ${\bf S}$ of ${\bf G}$ is {\it connected} if it is not disconnected.
 \end{definition}

\begin{observation}\label{quotient connected}
    \blu{A topological graph ${\bf G}$
is connected if and only if its topological realization $|{\bf G}|$
is connected. Moreover, if $ K\subseteq |{\bf G}|$ is connected, then $\pi^{-1}(K)$ is connected. If ${\bf Q}\subseteq {\bf G}$ is connected, then $\pi({\bf Q})$ is connected.}
\end{observation}

\begin{definition} 
Given a topological graph ${\bf G}$, a subgraph ${\bf S}$ of ${\bf G}$, and a vertex ${\bf a}\in {\bf S}$, the {\it component of ${\bf S}$ containing ${\bf a}$} is the
largest connected  subgraph ${\bf C}$ of ${\bf S}$ that contains ${\bf a}$; in other words ${\bf C}=\bigcup\{{\bf P}\subseteq {\bf S}\colon {\bf a}\in {\bf P}\ \text{ and }{\bf P} \text{ is connected} \} $. Note that if ${\bf S}$ is a topological graph then every component of ${\bf S}$ is closed.
\end{definition}

\begin{definition}
 Let $F$ be a graph, $A$ and $B$ be subgraphs of $ F$ and $a,b\in F$.
 We say that $A$ is {\it adjacent} to $B$ if $A \cap B = \emptyset$ and there are $x\in A$ and $y \in B$ such that 
 $\langle  x,  y \rangle \in E( F)$. 
 We say that $ a$ is {\it adjacent} to $ A$ if $\{ a\}$ is adjacent to $ A$,  and similarly, $ a$ is {\it adjacent} to $ b$ if $\{ a\}$ is adjacent to $\{ b\}$.
\end{definition}

\blu{The following definition of an arc in the setting of topological graphs corresponds to the topological characterization of an arc being a continuum that contains exactly two non-cut points, see Nadler~\cite[Theorem 6.17]{Nadler-intro}, though we additionally allow a single vertex to be an arc.}

\begin{definition}\label{arc-def}
 \blu{ We say that a topological graph ${\bf G}$ is an {\it arc} if $ {\bf G}$ is a single vertex or two vertices joined by an edge, or else $\bf G$ is connected and there are exactly two vertices ${\bf a},{\bf b}\in {\bf G}$ such that for every
  ${\bf x}\in {\bf G}\setminus \{{\bf a},{\bf b}\}$ we have  ${\bf G}\setminus \{{\bf x}\}={\bf C}\cup {\bf D}$ where ${\bf C},{\bf D}$ are nonempty and closed in ${\bf G}\setminus \{{\bf x}\}$, ${\bf a}\in {\bf C}$, ${\bf b}\in {\bf D}$, and  $\langle {\bf t}, {\bf s}\rangle\notin E({\bf G})$ whenever ${\bf t}\in {\bf C}$ and ${\bf s}\in {\bf D}$.
  } The vertices ${\bf a}$ and ${\bf b}$ are called {\it end-vertices} of the arc and we say that $\bf G$  {\it joins} ${\bf a}$ and ${\bf b}$. 
\end{definition}

\begin{definition}
  A topological graph ${\bf G}$ is called {\it arcwise connected} if for every two vertices ${\bf a},{\bf b}\in {\bf G}$ there is a subgraph of ${\bf G}$ that is an arc and contains ${\bf a}$ and ${\bf b}$. If ${\bf p}  \in {\bf G}$ then the {\em arc-component of ${\bf G}$ at ${\bf p}$} is $\bigcup\{{\bf A} \colon {\bf A} \text{ is an arc in } {\bf G} \text { and } {\bf p} \in {\bf A}\} $.
  \end{definition}

\begin{definition}
A {\em path} in a finite graph $A$ is a sequence $p_0, p_1, p_2,\ldots, p_n$  of vertices in $A$ such that $n>0$ and for every $0\leq i<n$ it holds 
$\langle p_i, p_{i+1}\rangle \in E(A)$. 
Note that a path is exactly the image by a homomorphism of a finite arc.
\end{definition}

\begin{definition}\label{her-uni}

 A topological graph ${\bf G}$ is {\it hereditarily unicoherent} if for every two nonempty closed connected topological graphs, ${\bf P}$ and ${\bf Q}$, with $V({\bf P}) \subseteq V({\bf G})$, $V({\bf Q}) \subseteq V({\bf G})$, $E({\bf P}) \subseteq E({\bf G})$, and $E({\bf Q}) \subseteq E({\bf G})$, the intersection ${\bf P}\cap {\bf Q}=(V({\bf P})\cap V({\bf Q}), E({\bf P})\cap E({\bf Q}))$ is connected. 

 \end{definition}

Note that if ${\bf P}$ and ${\bf Q}$  in Definition \ref{her-uni} were taken to be subgraphs of ${\bf G}$ then some graphs, for example, complete graphs, would be hereditarily unicoherent, which seems counter-intuitive. 
Nevertheless, in Section~\ref{sec-uni}, it will suffice to work with the following weaker version of hereditary unicoherence.

\begin{definition}
    
 A topological graph ${\bf G}$ is {\it subgraph hereditarily unicoherent} if for every 
two nonempty closed connected subgraphs ${\bf P}$ and ${\bf Q}$ of $\mathbb F$  the intersection ${\bf P}\cap {\bf Q}$ is connected. 

 \end{definition}
  Complete graphs are not hereditarily unicoherent but they are subgraph hereditarily unicoherent.

\begin{definition}\label{dendroid-def}
  A hereditarily unicoherent and arcwise connected topological graph is called a {\it dendroid}.
\end{definition}

\begin{definition}\label{order-def} 

By a {\it rooted graph} we mean a topological graph ${\bf G}$ with a distinguished vertex $r({\bf G})$, called the {\it root}. 
On a rooted graph
we define an order $\le$ by ${\bf x}\le {\bf y}$ if every arc containing $r({\bf G})$ and ${\bf y}$ contains~${\bf x}$. 
\end{definition}

\begin{definition}\label{def:fan}
A rooted dendroid ${\bf T}$ is a {\it fan} if for every ${\bf s},{\bf t}\in {\bf T}$ which are incomparable in the order from Definition \ref{order-def} , if ${\bf p}\not \in \{{\bf s},{\bf t}\}$ is such that 
${\bf p}\leq {\bf s}$ and ${\bf p}\leq {\bf t}$, then ${\bf p}$ is the root of ${\bf T}$. 
\end{definition}

\begin{definition}
A rooted topological graph ${\bf X}$ with the root $r({\bf X})$ is called a {\it smooth dendroid} if ${\bf X}$ is a dendroid 
and the order $\le$, defined in \ref{order-def}, 
 is closed. A smooth dendroid which is a fan is called a {\it smooth fan}.
\end{definition}

\section{Confluent epimorphisms}

Let us start with recalling the following definition introduced in \cite{WJC-RPR-Fraisse}.

\begin{definition}
Given two topological graphs ${\bf G}$ and ${\bf H}$, an epimorphism $f\colon  {\bf G}\to {\bf H}$  is called {\it confluent} if for every closed connected subset ${\bf Q}$ of ${\bf H}$ and every component ${\bf C}$ of $f^{-1}({\bf Q})$ we have $f({\bf C})={\bf Q}$. Equivalently, if for every closed connected subset ${\bf Q}$ of ${\bf H}$ and every vertex ${\bf a}\in {\bf G}$ such that $f({\bf a})\in {\bf Q}$, there is a connected set ${\bf C}$ of ${\bf G}$ such that ${\bf a}\in {\bf C}$ and $f({\bf C})={\bf Q}$.
It is called {\it monotone} if preimages of closed connected sets are connected.
Clearly, every monotone epimorphism is confluent.
\end{definition}

\begin{observation}
If $f\colon  {\bf X} \to {\bf Y}$ and $g\colon  {\bf Y} \to {\bf Z}$ are confluent epimorphisms between topological graphs, then $g\circ f\colon {\bf X} \to {\bf Z}$ is a confluent epimorphism.
\end{observation}

Let us recall some results shown in \cite{WJC-RPR-Fraisse}

\begin{theorem}\label{confluent-edges}
 Given two finite graphs $G$ and $H$, the following conditions are equivalent for an epimorphism $f\colon G\to H$:
\begin{enumerate}
  \item $f$ is  confluent;
  \item for every edge $P\in  E(H)$ and for every vertex $a\in G$ such that $f(a)\in P$, there is an edge $E\in E(G)$ and a connected set $R\subseteq G$ such that  $E\cap R\ne \emptyset$, $f(E)=P$, $a\in R$, and $f(R)=\{f(a)\}$;
  \item  for every edge $P\in  E(H)$ and every component $C$ of $f^{-1}(P)$ there is an edge $E$ in $C$ such that $f(E)=P$.
\end{enumerate}
\end{theorem}

The following two propositions will be needed later.
\begin{proposition}[\protect{\cite[Proposition 4.5]{WJC-RPR-Fraisse}}]\label{inverselim}
Let $\{F_n,\alpha_n\}$ be an inverse sequence of finite graphs, where each $\alpha_n$ is a confluent epimorphism. Let ${\bf F}$ denote its inverse limit. Then for each $n$, the map $\alpha^\infty_n\colon {\bf F}\to F_n$ is a confluent epimorphism.
\end{proposition}

\begin{proposition}[\protect{\cite[Proposition 4.4]{WJC-RPR-Fraisse}}]\label{threemaps}
 Consider the epimorphisms $f\colon {\bf F}\to {\bf G}$ and $g\colon {\bf G}\to {\bf H}$ between topological graphs. If the composition 
$g\circ f\colon {\bf F}\to {\bf H}$ is confluent, then $g$ is confluent.

\end{proposition}

\begin{corollary}\label{uniqueness}
Let $\mathcal G$ be the family of nonempty
finite connected graphs with confluent epimorphisms.
Then  \blu{$\mathcal{G}$} is consistent.

Moreover, epimorphisms in $\mathcal{G}^\omega$ are exactly confluent epimorphisms.
\end{corollary}
\begin{proof}
Let ${\bf F}=\iLim\{F_n,\alpha_n\}=\iLim \{G_n, \beta_n\}$.
It suffices to check that maps $g_i$ and $h_i$ constructed in Lemma \ref{ladder} are in $\mathcal G$. 
By our assumption, if $\alpha^{n_{i+1}}_{n_i}=h_ig_{i+1}$,
then since $\alpha^{n_{i+1}}_{n_i}$ is in $\mathcal G$, we get by Proposition \ref{threemaps} that $h_i$ is in $\mathcal G$. Similarly, if 
$\beta^{m_{i+1}}_{m_i}=g_ih_i$, then since $\beta^{m_{i+1}}_{m_i}$ is in  $\mathcal G$, we have that $g_i$ is in $\mathcal G$.
This shows that \blu{$\mathcal{G}$} is consistent.

The moreover part follows from Proposition \ref{inverselim}.
\end{proof}

The following proposition is an analog to \cite[I, p. 213]{JJC-Confluent}. 

\begin{proposition}[\protect{\cite[Proposition 4.11]{WJC-RPR-Fraisse}}]\label{confluent-restrictions}
Let $f\colon {\bf F}\to {\bf G}$ be a confluent epimorphism between topological graphs, let ${\bf Q}$ be a closed connected subgraph   of  ${\bf G}$, and let ${\bf C}$
be a component of $f^{-1}({\bf Q})$. Then the restriction 
$f|_{\bf C}$ is confluent. 
\end{proposition}

It is not the case, as the following example shows, that if $f\colon G \to H$ is a confluent epimorphism and $K$ is a subgraph of $G$ then $f|_K\colon K \to H$ must be confluent.

\begin{example}
Let $K$ consist of the vertices $a_1,c_1,a_2,d_1,a_3,b_1$ and edges $\langle a_1,c_1\rangle, \langle c_1,a_2\rangle, \langle a_2,d_1\rangle, \langle d_1,a_3\rangle,\langle a_3,b_1\rangle,\langle b_1,a_1\rangle$. Let $G$ be the graph $K$ with the additional vertices $b_2,c_2,d_2$ and edges $\langle b_2,a_2\rangle,$ $\langle b_2,b_1\rangle,$ $\langle c_2,a_3\rangle,$ $\langle c_2,c_1\rangle,$ $\langle d_2,a_1\rangle,$ $\langle d_2,d_1\rangle$. Finally, let $H$ be the triod with vertices $A,B,C,D$ and edges $\langle A,B\rangle, \langle A,C\rangle,\langle A,D\rangle$. Define $f\colon G \to H$ by $f(a_i)=A$, $f(b_i)=B$, $f(c_i)=C$, $f(d_i)=D$.  One can use Theorem~\ref{confluent-edges} to check that $f\colon G \to H$ is confluent but $(f|_K)^{-1}(\langle A,B\rangle)$ contains the component $\{a_2\}$ so condition (3) of Theorem \ref{confluent-edges} is not satisfied by $f|_K\colon K \to H$.
\end{example}

The construction below was introduced in \cite{WJC-RPR-Fraisse}.
 
\begin{definition}
For given graphs $A,B,C$ and epimorphisms $f\colon B\to A$, $g\colon C\to A$ by {\it standard amalgamation procedure} we mean the graph $D$ defined by $V(D)=\{( b,c)\in {V}(B)\times {V}(C)\colon f(b)=g(c)\}$; $E(D)=\{\langle ( b,c), (b',c') \rangle\colon \langle b,b'\rangle\in E(B)  \text{ and } \langle c,c'\rangle\in E(C)\}$; and
$f_0(( b,c))=b$, $g_0(( b,c))=c$. Note that $f_0$ and $g_0$ are epimorphisms and $f\circ f_0=g\circ g_0$.
\end{definition}

\begin{theorem} 
The family of nonempty finite connected graphs with confluent epimorphisms is a projective \F family.
\end{theorem}

\begin{proof}
We need to see that Definition \ref{definition-Fraisse} holds for this family. Conditions (1) and (2) are easy to check.  

Condition (4) was proved in \cite{WJC-RPR-Fraisse}. We repeat that proof for reader's convenience. We show that a component of the graph obtained by the standard amalgamation procedure is as required. Let $f\colon B\to A$ and $g\colon C\to A$ be confluent epimorphisms. Let $f_0\colon D_0\to B$ and $g_0\colon D_0\to C$ be as in the standard amalgamation procedure.

We show that $g_0$ is confluent.
Let $\langle c,c'\rangle \in E(C)$ and $(b,c)\in D_0$. By Theorem \ref{confluent-edges} we need to find an edge 
$\langle d_1,d_2\rangle \in E(D_0)$ that is mapped by $g_0$ onto  $\langle c,c'\rangle$ and a connected set $R$ such that 
$(b,c), d_1\in R$, and $g_0(R)=\{c\}$. 
In case $g(c)=g(c')$, we have $(b,c')\in D_0$, therefore we can simply take $d_1=(b,c)$, $R=\{(b,c)\}$, and $d_2=(b,c')$. Suppose now that $g(c)\neq g(c')$. 
Since $f$ is confluent, there are vertices $b_1,b_2\in B$ and a connected set $R_B$ such that $\langle b_1,b_2\rangle\in E(B)$, $b,b_1\in R_B$,  $f(\langle b_1,b_2\rangle)=\langle g(c),g(c')\rangle$, and $f(R_B)=\{f(b)\}$. It is enough to put $d_1=( b_1,c)$, $d_2=( b_2,c')$, and $R=R_B\times \{c\}$. 
Similarly, $f_0$ is confluent. 
Let $D$ be a component of $D_0$. Then $D$, $f_0|_D$ and $g_0|_D$, by Proposition \ref{confluent-restrictions}, are as required.

Finally, condition~(3) follows from 
condition~(4) by noting that given finite connected graphs $B$ and $C$ if we let $A$ be the graph consisting of a single vertex and $f\colon B \to A$, $g\colon C\to A$ be constant maps, which are confluent, then condition~(4) gives the existence of a connected graph $D$ and the required confluent epimorphisms from $D$ onto $B$ and $C$.

\end{proof}

\begin{theorem}
If $\mathcal G$ is the family of nonempty finite connected graphs with confluent epimorphisms and $\mathbb G$ is the projective \F limit of $\mathcal G$ then the edge relation $E(\mathbb G)$ is transitive.
\end{theorem}

\begin{proof}
To see that $\mathbb{G}$ has a transitive edge relation note that given $G$, a finite connected graph containing edges $\langle a,b\rangle$ and $\langle b,c \rangle$, where $a,b,c$ are pairwise different, we may let $H$ be the graph with $V(H)=V(G)\cup \{s\}$ and $E(H)=E(G)\cup \{\langle a,s\rangle, \langle s,b\rangle\}\setminus \{\langle a,b\rangle \}$ and  $f^H_G\colon H \to G$ be the confluent epimorphism defined by $f^H_G(s)=b$ and $f^H_G=id$ otherwise.  Then $H$ satisfies conditions of Theorem \ref{transitive-edges}, so $\mathbb G$ has a transitive edge relation.

\end{proof}

\begin{notation}
Throughout the remainder of the article, we denote the family of nonempty finite connected graphs with confluent epimorphisms by $\mathcal G$ and the projective \F limit of the family $\mathcal G$ by $\mathbb G$. We denote the topological realization of $\mathbb G$ by $|\mathbb G|$.
\end{notation}

\section{Kelley property and dimension}

The following definition adapts the idea of a Kelley continuum to topological graphs.  It was first given in \cite[Definition 6.3]{WJC-RPR-Fraisse}.

\begin{definition}
A topological  graph ${\bf X}$ is called {\it Kelley} if ${\bf X}$ is connected and for every closed and connected set ${\bf K}\subseteq {\bf X}$,
every vertex ${\bf p}\in {\bf K}$, and every sequence ${\bf p}_n\to {\bf p}$ of vertices in ${\bf X}$ there are closed and connected sets ${\bf K}_n$ such that ${\bf p}_n\in {\bf K}_n$ and
$\lim {\bf K}_n={\bf K}$.

\end{definition}

\begin{proposition}\label{Kelley}

Let $ \mathbb H$ be a Kelley topological graph, which has a transitive set of edges. Then 
$|{\mathbb H}|$ is a Kelley continuum.
\end{proposition}

\begin{proof}

To see that $|{\mathbb H}|$ is Kelley, let $K$ be a subcontinuum of $|{\mathbb H}|$, $p \in K$, and $p_1,p_2,\ldots$ be a sequence in $|\mathbb H|$ that converges to $p$. Let $\pi\colon \mathbb H\to |{\mathbb H}|$ be the quotient map, ${\bf p}' \in \pi^{-1}(p)$ and ${\bf p}_i'\in \pi^{-1}(p_i)$ be such that ${\bf p}_i'$ converges to ${\bf p}'$. Since $\mathbb H$ is Kelley \blu{and $\pi^{-1}(K)$ is connected, see Observation~\ref{quotient connected}.} there exist closed connected graphs ${\bf K}_i' \subseteq {\mathbb H}$ such that ${\bf p}_i' \in {\bf K}_i'$ and ${\bf K}_i'$ converges to $\pi^{-1}(K)$. Then $\pi({\bf K}_i)$ are subcontinua in $|{\mathbb H}|$ that contain $\pi({\bf p}'_i)=p_i$ and converge to $K$, which follows from the continuity of $\pi$.
\end{proof}

\begin{corollary}
The continuum $|\mathbb G|$ is a Kelley continuum.
\end{corollary}
\begin{proof}
    By \cite[Theorem 6.5]{WJC-RPR-Fraisse}
 $\mathbb G$ is  Kelley.
Therefore the conclusion follows from Proposition \ref{Kelley}.
\end{proof}

\begin{theorem}
Suppose  ${\bf G}$ is a topological graph with a transitive edge relation and such that $\card (\pi^{-1}(x))\le 2 $, where $\pi$ is the quotient map. Then $\dim |{\bf G}|\le 1$.
\end{theorem}
\begin{proof}
It is known, see e.g. \cite[Theorem 22.2]{Nadler-Dimension},  that for a continuous surjective map $f\colon X\to Y$ between compact sets, if $\card(\bd f^{-1}(y)))\le m$ for all $y\in Y$, then 
$\dim(Y)\le\dim(X)+m-1$. In our case $\dim ({\bf G})=0$  and $\card(\bd (\pi^{-1}(x)))=\card (\pi^{-1}(x))\le 2$ for all $x\in |{\bf G}|$, by Theorem \ref{transitive-edges}, so the conclusion follows. 
\end{proof}

\begin{corollary}
The continuum $|\mathbb G|$ is one-dimensional.
\end{corollary}

\section{indecomposable topological graphs}
\begin{definition}
A connected topological graph ${\bf G}$ is called {\it decomposable} if there are two closed connected subgraphs ${\bf A}$ and ${\bf B}$
such that ${\bf G}={\bf A}\cup {\bf B}$, ${\bf A}\ne {\bf G}$, and ${\bf B}\ne {\bf G}$. A topological graph ${\bf G}$ is called {\it indecomposable} if it is not decomposable.
\end{definition}

\begin{theorem}\label{two-pass}
Suppose $\mathcal F$ is a projective \F family of connected graphs  and for every graph $F\in \mathcal F$ there is a graph
$G\in \mathcal F$ and an  epimorphism $f^G_F\colon  G\to F$ such that for every two connected graphs $A,B\subseteq G$ such that $G=A\cup B$ we have $f^G_F(A)=F$ or $f^G_F(B)=F$. Then the projective \F limit $\mathbb F$ of
$\mathcal F$ is an indecomposable graph.
\end{theorem}
\begin{proof}

Suppose the projective \F limit $\mathbb F$ of
$\mathcal F$ is a decomposable graph, i.e. there are two closed connected subgraphs ${\bf A}$ and ${\bf B}$
such that $\mathbb F={\bf A}\cup {\bf B}$, ${\bf A}\ne \mathbb F$, and ${\bf B}\ne \mathbb F$. Consider vertices ${\bf a}\in {\bf A}\setminus {\bf B}$
and ${\bf b}\in {\bf B}\setminus {\bf A}$ and let $\varepsilon >0$ be such that $\varepsilon < d({\bf x},{\bf a})$ for every ${\bf x}\in {\bf B}$ and 
$\varepsilon < d({\bf x},{\bf b})$ for every ${\bf x}\in {\bf A}$, where $d$ is a fixed metric on $ \mathbb F$. Let $F\in \mathcal F$ be a graph and let $f_F\colon \mathbb F \to F$ be an 
{$\varepsilon$-epimorphism}. Then $f_F({\bf A})\cup f_F({\bf B})=F$, $f_F({\bf A})\ne F$, and $f_F({\bf B})\ne F$. 
Let the graph $G$ and the epimorphism $f^G_F$ be as in assumption of the theorem, and let $f_G\colon \mathbb F\to G$ be an  epimorphism such that 
$f_F=f_F^G\circ f_G$. Then $f^G_F(f_G({\bf A}))=F$ or $f^G_F(f_G({\bf B}))=F$, contrary to the established properties of $f_F({\bf A})$ and $f_F({\bf B})$. 
\end{proof}
\begin{theorem}
The family $\mathcal{G}$ satisfies assumptions of Theorem \ref{two-pass}, so its projective \F limit $\mathbb G$ is indecomposable.

\end{theorem}
\begin{proof}
Given graph $F$, let $\{E_1,E_2,\dots, E_n\}$ be the set of all nondegenerate edges and let $E_i=\langle a_i, b_i\rangle$ for each $i\in\{1,2,\dots,n\}$. 

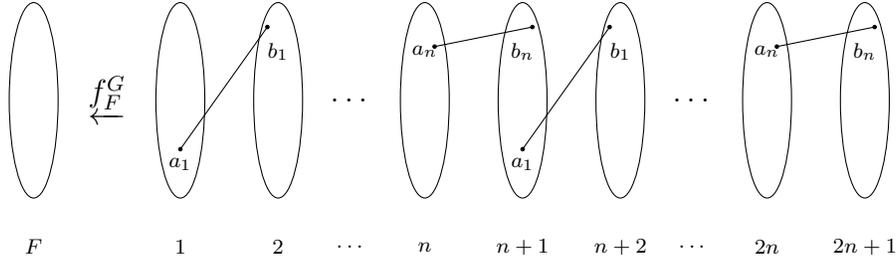
\begin{center}
\begin{tikzpicture}[scale=0.65]

\draw (2,0) ellipse (0.5 and 2);
\draw (4,0) ellipse (0.5 and 2);

\draw (7,0) ellipse (0.5 and 2);
\draw (9,0) ellipse (0.5 and 2);
\draw (11,0) ellipse (0.5 and 2);
\draw (14,0) ellipse (0.5 and 2);
\draw (16,0) ellipse (0.5 and 2);

\draw (-1,0) ellipse (0.5 and 2);

\draw (2,-1) -- (3.78, 1.5);
\filldraw[black] (2,-1) circle (1pt);
\filldraw[black] (3.78, 1.5) circle (1pt);
 \node at (2,-1.3) {\tiny $a_1$};
\node at (4,1) {\tiny $b_1$};

\node at (5.5,0) {$\dots$};

\draw (9,-1) -- (10.78, 1.5);

  \filldraw[black] (9,-1) circle (1pt);
    \filldraw[black] (10.78, 1.5) circle (1pt);
 \node at (9,-1.3) {\tiny $a_1$};
\node at (11,1) {\tiny $b_1$};
\node at (12.5,0) {$\dots$};

\draw (7.2,1.1) -- (9.2, 1.5);
\filldraw[black] (7.2,1.1) circle (1pt);
\filldraw[black] (9.2, 1.5) circle (1pt);
 \node at (7,1) {\tiny $a_n$};
  \node at (9,1) {\tiny $b_n$};
  
  \draw (14.2,1.1) -- (16.2, 1.5);
  
  \filldraw[black] (14.2,1.1) circle (1pt);
\filldraw[black] (16.2, 1.5) circle (1pt);
 \node at (14,1) {\tiny $a_n$};
  \node at (16,1) {\tiny $b_n$};
\node at (2,-3) {\tiny 1};  
 \node at (4,-3) {\tiny 2};  
 \node at (7,-3) {\tiny $n$};  
\node at (9,-3) {\tiny $n+1$};  
\node at (11,-3)  {\tiny $n+2$};  
\node at (14,-3)  {\tiny $2n$};
\node at (16,-3)  {\tiny $2n+1$};  

\node at (12.5,-3) {\tiny $\dots$};
\node at (5.5,-3) {\tiny $\dots$};
\node at (-1,-3) {\tiny $F$};
\node at (0.5,0) {$\underleftarrow{f^G_F}$};

\end{tikzpicture}
\end{center}

Define the graph $G$ by the following conditions (see the picture above, each ellipse represents a copy of the graph $F$):
\begin{enumerate}
    \item $V(G)=V(F)\times \{1,2,\dots, 2n+1\}$;
   \item 
   $E(G)=E(F)\times \{1,2,\dots, 2n+1\}\cup \\ 
   \bigcup_{i=1}^n \{ \langle (a_{i}, i), (b_{i},i+1)\rangle, \langle (a_{i},n+i),  (b_{i},n+i+1)\rangle \}$.
\end{enumerate}
Define $f^G_F$ to be the projection onto the first coordinate and note that $f^G_F$ is confluent. Now suppose that $A$ and $B$ are connected subgraphs of $G$ such that $A\cup B=G$.  Note that if $A$ intersects $F\times \{i\}$ and $F\times \{i+j\}$ for some $j\ge 1$ then, since $A$ is connected, it must contain the
edges between $F\times \{i+k\}$ and $F\times \{i+k+1\}$,
for each $k=0,1,\ldots, j-1$.

One of $A$ or $B$, say $A$, must intersect at least $n+1$ of the copies of $F$.  It follows that $f_F^G(A)=F$. Hence, by Theorem \ref{two-pass}, $\mathbb G$ is indecomposable.
\end{proof}

The following observation follows from the fact that a topological graph ${\bf G}$
is connected if and only if its topological realization $|{\bf G}|$
is connected.

\begin{observation}
Suppose ${\bf G}$ is an indecomposable topological graph with a transitive edge relation; then the topological realization $|{\bf G}|$  is an indecomposable continuum.
\end{observation}

\begin{corollary}
The continuum $|\mathbb G|$ is indecomposable.
\end{corollary}

If $X$ is a continuum, we say that $p,q\in X$ belong to the same composant if there is a proper subcontinuum of $X$ which contains both $p$ and $q$.
If moreover $X$ is non-trivial and indecomposable, then $X$ has uncountably many composants and they are mutually disjoint (see 11.15 and 11.17 in \cite{Nadler-intro}). In particular, we obtain the following corollary.

\begin{corollary}
    $|\mathbb G|$ is not arcwise connected.
\end{corollary}

\section{Arc components, pseudo-arcs, and Cantor fans}

 We show in this section that arc components are dense in both $\mathbb G$ and $|\mathbb{G}|$. Furthermore, we prove that the pseudo-arc, the universal pseudo-solenoid, and the Cantor fan, can be embedded in $|\mathbb{G}|$
and that each point of $|\mathbb G|$ is the top of the Cantor fan.

We start with the following lemmas.

\begin{lemma}  \label{arc-proj-to-arc}  
Suppose ${\bf G}$ is a topological graph which is an arc with a transitive edge relation, then the topological realization $|{\bf G}|$  is an arc or a point.
\end{lemma}

\begin{proof}

    \blu{Suppose $|{\bf G}|$ is not a point. Then ${\bf G}$ has more than two vertices. Let ${\bf a}$ and ${\bf b}$ be end-vertices of ${\bf G}$ and $x\in |{\bf G}|\setminus \{\pi({\bf a}), \pi({\bf b})\}$, where $\pi\colon{\bf G} \to |{\bf G}|$ is the quotient map. Let ${\bf t}\in\pi^{-1}(x)$, then ${\bf G} \setminus \{{\bf t}\}$ is disconnected, and write ${\bf G} \setminus \{{\bf t}\} = {\bf H}|{\bf K}$. Take ${\bf H}'={\bf H}\cap ({\bf G}\setminus \pi^{-1}(x))$ and ${\bf K}'={\bf K}\cap ({\bf G}\setminus \pi^{-1}(x))$. Then $\pi({\bf H}') \cup \pi({\bf K}')= |{\bf G}|\setminus\{x\}$ and $\pi({\bf H}') \cap \pi({\bf K}') = \overline{\pi({\bf H}')} \cap \pi({\bf K}') = \pi({\bf H}') \cap \overline{\pi({\bf K}')}=\emptyset$, so $|{\bf G}|\setminus \{x\}$ is disconnected. 

    Now suppose $x \in \{\pi({\bf a}), \pi({\bf b})\}$. We need to show that $|{\bf G}| \setminus \{x\}$ is connected.  
    Denote ${\bf C}=\pi^{-1}(\pi({\bf a}))$ and note that ${\bf C}$ is connected. An argument essentially identical to \cite[Proposition 6.3]{Nadler-intro} shows that if ${\bf G}\setminus {\bf C} = {\bf P}|{\bf Q}$, then ${\bf P} \cup {\bf C}$ and ${\bf Q} \cup {\bf C}$ are connected. If ${\bf b} \in {\bf C}$, i.e. $\pi({\bf a})=\pi({\bf b})$, then $\langle {\bf a}, {\bf b}\rangle \in E({\bf G})$, contradicting Definition~\ref{arc-def}.
   Therefore ${\bf b} \in {\bf P} \cup {\bf Q}$,  say ${\bf b} \in {\bf Q}$. Take any ${\bf p}\in {\bf P}$, we know that ${\bf G}\setminus\{{\bf p}\}$ is disconnected, so ${\bf G}\setminus\{{\bf p}\}={\bf T}|{\bf S}$, and say ${\bf a}\in {\bf T}$ and ${\bf b}\in {\bf S}$. Then ${\bf a}\in {\bf T}\cap ({\bf C}\cup {\bf Q})$, ${\bf b}\in {\bf S}\cap ({\bf C}\cup {\bf Q})$, and  $({\bf T}\cap ({\bf C}\cup {\bf Q})) \cup ({\bf S}\cap ({\bf C}\cup {\bf Q}))={\bf C}\cup {\bf Q}$, contradicting that ${\bf C}\cup {\bf Q}$ is connected.

    Thus $|{\bf G}| \setminus \{ x\}$ is disconnected if and only if $x \not \in\{\pi({\bf a}),\pi({\bf b})\}$. So, by the topological characterization of an arc, $|{\bf G}|$ is an arc.
    }
\end{proof}

The lemma below is stated in \cite[Lemma 3.8]{WJC-RPR-Fraisse} and it is proved in \cite[Lemma 3.15]{wazewski-fraisse}. We will use it implicitly in Theorems \ref{arc-dense} and \ref{Cantor-fan}.
\begin{lemma}\label{inv-limit-of-arcs}
    The inverse limit of finite arcs with monotone epimorphisms is an arc.
\end{lemma}

To show the density of arc components, we first need a lemma on lifting arcs.
\begin{lemma}\label{arcs-in-graphs}
Let $G$ and $H$ be  finite graphs and let 
$f\colon G\to H$ be a confluent epimorphism.  Let $A\subseteq H$ be an arc with an end-vertex $a$, and let $b\in G$ be a vertex such that $f(b)=a$. Then
there is 
an arc $B\subseteq G$ with one of the end-vertices equal to $b$ such that $f|_B\colon B\to A$ is
a monotone epimorphism. 
\end{lemma}

\begin{proof}
Let $A=\{a_1=a,a_2, \dots, a_n\}$ be an arc with $a=a_1$
and $\langle a_i,a_{i+1} \rangle\in E(H)$, $i=1,\ldots, n-1$.
Let $C_1$ be the component of $f^{-1}(\{ a_1=a,a_2\})$ that contains the vertex $b$ 
and note that $C_1$ contains an arc $B_1$ which has $b$ as an end-vertex, whose image is the set 
$\{a_1,a_2\}$,  such that  $f|_{B_1}$ is monotone,
and $|f^{-1}(a_2)\cap B_1|=1$.
In a similar way we may construct an arc $B_2$ in the component $C_2$ of $f^{-1}(\{ a_2,a_3\})$ such that $B_2$ contains the vertex from $f^{-1}(a_2)\cap B_1$ as an end-vertex,
 $f|_{B_2}$ is monotone,   $f({B_2})=\{a_2,a_3\}$, and 
$|f^{-1}(a_3)\cap B_2|=1$.
Continuing the same way, we construct an arc $B=B_1\cup B_2\cup \dots\cup B_{n-1}$ that has the required properties. 
\end{proof}

\begin{theorem}\label{arc-dense}
Each arc component of $\mathbb G$ is dense in $\mathbb G$.
\end{theorem}

\begin{proof}
Let ${\bf a},{\bf b}\in \mathbb G$ and let $\varepsilon >0$ be given. Write $\mathbb G$ as an inverse limit of a Fra\"\i ss\'e sequence $\{G_n,g_n\}$ for $\mathcal G$. Let $N$ be such that for $n>N$,
$g^\infty_n$ is an $\varepsilon$-map.  Write ${\bf a}=(a_0,a_1,a_2,\ldots)$ with $a_i\in G_i$, and let $b_{N}=g^\infty_{N}({\bf b})$. Let $A_N$ be an arc in $G_N$ with end-vertices $a_{N}$ and $b_{N}$. Apply Lemma \ref{arcs-in-graphs} to $g_N$, $a_N$, $a_{N+1}$ and the arc $A_N$ to obtain an arc $A_{N+1}$ in $G_{N+1}$.  Next,
apply Lemma \ref{arcs-in-graphs} to $g_{N+1}$, $a_{N+1}$, $a_{N+2}$ and the arc $A_{N+1}$ to obtain an arc $A_{N+2}$ in $G_{N+2}$.  Continue in the same way.
Finally, for each $n\geq N$ we have an arc $A_n$ with an end-vertex equal to $a_n$, such that $g_n|_{A_{n+1}}$ is monotone and onto $A_n$. Then the arc obtained as the inverse limit of $\{A_n,g_n|_{A_{n+1}}\}$ has ${\bf a}$ as one of the end-vertices and the second end-vertex is within $\varepsilon$ from ${\bf b}$.
\end{proof}

The theorem above together with Lemma \ref{arc-proj-to-arc} give the following corollary.
\begin{corollary}
The continuum $|\mathbb G|$ has all arc-components dense.
\end{corollary}

Our next goal is to show that the pseudo-arc and the universal pseudo-solenoid can be embedded in $|\mathbb G|$. We begin with a definition and an embedding theorem.

We say that a family $\mathcal P$ of finite connected graphs and a subfamily of epimorphisms between them has the {\em lifting property} if the following holds:

For all finite graphs $G$, $H$ and  
 a confluent epimorphism $f\colon G\to H$, if  
 $A\subseteq H$ is in $\mathcal P$, then
there is 
  $B\subseteq G$ in $\mathcal P$ such that $f|_B\colon B\to A$ is in $\mathcal P$.

\begin{example}

\begin{enumerate}
\item The family of arcs with monotone epimorhisms has the lifting property. (Lemma \ref{arcs-in-graphs})

\item The family of cycles with confluent epimorhisms 
    has the lifting property. (Lemma \ref{raz})
    \end{enumerate}
\end{example}  

\begin{theorem}\label{lemma lift}
\blu{Let $\mathcal{P}$ be a family of finite connected graphs with epimorphisms, which has the lifting property.}
Let $\{B_n, \beta_n\}$  be an inverse sequence \blu{ with $B_n\in \mathcal{P}$ and $\beta_n$ being epimorphisms, for all $n$,} with the following property. For every  $D\in\mathcal{P}$, $k>0$, and    $g\colon D\to B_k$ in 
$\mathcal{P}$,
there is $l>k$ and epimorphism (not necessarily in $\mathcal{P}$)  $f\colon B_l\to D$  with $g\circ f=\beta^l_k$. Then the inverse limit of $\{B_n, \beta_n\}$ can be embedded in $\mathbb{G}$.
\end{theorem}

\begin{proof}
Let $\{F_n,\alpha_n\}$ be a Fra\"{\i}ss\'e sequence for $\mathcal G$ and  $\{B_n, \beta_n\}$ be the inverse sequence in $\mathcal{P}$.
We construct $(n_k)$ and $D_k\subseteq F_{n_k}$ such that the inverse limit of $\{D_n, \alpha_n|_{D_{n+1}}\}$ is isomorphic to the inverse limit of $\{B_n, \beta_n\}$.

We may assume that $F_1=B_1$. Let $D_1=D_2=F_1$ and take $f_1\colon B_1\to D_1$ and $g_1\colon D_2\to B_1$ to be identity maps.
Take $n_1=m_1=1=n_2$. 

Suppose that we have $n_1<\ldots< n_{k+1}$, $m_1<\ldots< m_k$, and we have epimorphisms $f_i\colon B_{m_i}\to D_i$ (not necessarily confluent) and  epimorphisms
$g_i\colon  D_{i+1}\to B_{m_i}$ from  $\mathcal{P}$, for $i=1,2,\ldots, k$. We also have  $D_i\subseteq F_{n_i}$, where $D_i$ is in $\mathcal{P}$, such that 
$g_{i-1}\circ f_i = \beta^{m_{i}}_{m_{i-1}}$ and $f_i\circ g_i=\alpha^{n_{i+1}}_{n_i}|_{D_{i+1}}$, for
$i=1,\ldots, k$.

We now obtain $n_{k+1}<n_{k+2}, m_k<m_{k+1}, f_{k+1}, g_{k+1}$, and $D_{k+2}$.
Using the property of $\{B_n, \beta_n\}$ we find $m_{k+1}$ and $f_{k+1}$ such that 
$f_{k+1}\colon B_{m_{k+1}}\to D_{k+1}\subseteq F_{n_{k+1}}$ and
$g_{k}\circ f_{k+1} = \beta^{m_{k+1}}_{m_{k}}$. 

We claim that we can extend the epimorphism $f_{k+1}\colon B_{m_{k+1}}\to D_{k+1}$, which may be not confluent, 
to a confluent epimorphism $F\colon A\to F_{n_{k+1}}$, for some connected graph $A$. We let
$ V(A)$ to be the disjoint union of $V( B_{m_{k+1}})$ and $V(F_{n_{k+1}})$.
Let 
\begin{align*}
E(A)=&\{\langle x, y\rangle \colon \langle x, y\rangle \in E(F_{n_{k+1}}) \cup E(B_{m_{k+1}}) \text{ or } \\
&(x\in B_{m_{k+1}}, y\in F_{n_{k+1}}, \langle f_{k+1}(x),y\rangle\in E(F_{n_{k+1}}) )\}.
\end{align*}
The map $F$ defined by $F|_{B_{m_{k+1}}}=f_{k+1}$ and equal to the identity otherwise is a 
confluent epimorphism, see Theorem~\ref{confluent-edges}, extending $f_{k+1}$. 

Finally use that  $\{F_n,\alpha_n\}$ is the Fra\"{\i}ss\'e sequence to find $n_{k+2}$ and $G\colon F_{n_{k+2}}\to A$ such that
$F\circ G= \alpha^{n_{k+2}}_{n_{k+1}}$. Let $D_{k+2}\subseteq F_{n_{k+2}}$ be in $\mathcal{P}$ such that 
$G(D_{k+2})=B_{m_{k+1}}$ and $G|_{D_{k+2}}$ is in  $\mathcal{P}$, which exists since $\mathcal{P}$ has the lifting property.  We let $g_{k+1}=G|_{D_{k+2}}$.

Note that the inverse limit of $\{D_n, \alpha_n|_{D_{n+1}}\}$ is isomorphic to the inverse limit of $\{B_n, \beta_n\}$.
Since the inverse limit of $\{F_{n_i}, \alpha_{n_i}^{n_{i+1}}\}$ is isomorphic to $\mathbb{G}$, we get the required embedding of 
the inverse limit of $\{B_n, \beta_n\}$ in $\mathbb{G}$.

\end{proof}

The \textit{pseudo-arc} is the unique hereditarily indecomposable chainable continuum.
In \cite{Pseudo} it was shown that the topological realization of the projective Fra\"\i ss\'e limit of the family of finite linear graphs, which are called arcs in this article, with epimorphisms which were not necessarily confluent, is the pseudo-arc. Our next result shows, somewhat surprisingly, that $|\mathbb G|$ contains a pseudo-arc.

\begin{theorem}\label{parc-embed}
The pseudo-arc can be embedded in $|\mathbb{G}|$.
\end{theorem}

\begin{proof}
Apply Irwin-Solecki  \cite{Pseudo}, where it is shown that the pseudo-arc is the topological realization of the projective Fra\"{\i}ss\'e family of finite arcs.
Apply Lemma \ref{arcs-in-graphs} to any Fra\"{\i}ss\'e sequence $\{I_n,\beta_n\}$  for that Fra\"{\i}ss\'e family and use Theorem \ref{lemma lift} and then Lemma \ref{quotient-emb}.
\end{proof}

 A \textit{pseudo-solenoid} is a non-chainable circularly chainable hereditary indecomposable continuum. The \textit{universal pseudo-solenoid} is a pseudo-solenoid such that there is a continuous surjection from it onto any non-chainable circularly chainable continuum, see Rogers~\cite{RogersII}, in particular, onto any pseudo-solenoid.

\begin{theorem}\label{pseu-embed}
The universal pseudo-solenoid can be embedded in $|\mathbb{G}|$.
\end{theorem}

\begin{proof}
    Irwin \cite[Section 4]{PhD-Irwin} showed that the universal pseudo-solenoid can be obtained as the topological realization of the projective \fra\ limit of the projective \fra\ family of cycles with epimorphisms of positive degree. 
    Apply Lemma \ref{raz} to any Fra\"{\i}ss\'e sequence $\{C_n,\gamma_n\}$  for that Fra\"{\i}ss\'e family and use Theorem \ref{lemma lift} and then Lemma \ref{quotient-emb}.
    
\end{proof}

We next want to show the following result.

\begin{theorem}\label{Cantor-fan}
Each point of the continuum $|\mathbb G|$ is the top of the Cantor fan embedded in $|\mathbb G|$.
\end{theorem}

First, we recall a definition from continuum theory.

\begin{definition}
A continuum homeomorphic to the cone over the Cantor set is called the {\em Cantor fan}, that is, it is homeomorphic to the quotient space $C \times [0,1]/\sim$ where $C$ is the Cantor set and the points $(x,1)$ are identified.  We will call the point $(x,1)$ the {\em top} of the fan. Note that this top point is often called the vertex of the fan but we want to only use the term vertex for elements of graphs. The Cantor fan $X$ can be characterized as a smooth fan whose set of endpoints is  homeomorphic to the Cantor set  (see Theorem 1 in \cite{Cantor-fan}).
\end{definition}

\begin{definition}\label{doubling-def}
Given $G\in {\mathcal G}$ and a vertex $p\in G$ let $$\Delta(G,p)=G\times \{0,1\}/ \left(\{p\}\times \{0,1\}\right)$$ and $\delta\colon \Delta(G,p)\to G$ be just the projection $\delta(x,i)=x$. Observe that 
$\delta$ is a confluent epimorphism.

\end{definition}

\begin{proof} (Proof of Theorem \ref{Cantor-fan}).

 Let $\{F_i,\alpha_i\}$ be a Fra\"{\i}ss\'e sequence for $\mathcal G$. 
Choose ${\bf a}\in\mathbb G$ and write ${\bf a}=(a_i) $. Let $\pi\colon \mathbb G\to |\mathbb G|$ be the quotient map.
 We are going to construct the Cantor fan in $ |\mathbb G|$ with the top $\pi({\bf a})$. 
 
Let $K_0$ be a graph in $\mathcal G$ containing at least two vertices.
Using that $\{F_i,\alpha_i\}$ is a Fra\"{\i}ss\'e sequence, take $n_0$ and a confluent epimorphism $f_0\colon F_{n_0}\to K_0$. 
\blu{Suppose that we have already constructed $f_i\colon F_{n_i}\to K_i$, for $i\leq m$, and we have $\delta_i\colon K_{i+1}\to F_{n_i}$,
 for $i<m$, such that $ \delta_{i-1}\circ f_i=\alpha^{n_i}_{n_{i-1}} $ for $i\leq m$.
 Let $K_{m+1}=\Delta(F_{n_{m}},a_{n_{m}})$ and let  $\delta_{m}\colon K_{m+1}\to F_{n_m}$ be the projection map of Definition \ref{doubling-def}, that we are subscripting to help keep track of which graph it is operating on.  
 By Proposition  \ref{B-CThm},  there is $n_{m+1}$ and a confluent epimorphism $f_{m+1}\colon F_{n_{m+1}}\to K_{m+1}$ such that 
 $ \delta_m\circ f_{m+1}=\alpha^{n_{m+1}}_{n_{m}} $. }
 Let $\beta_i\colon K_{i+1} \to K_i$ be $\beta_i=f_i\circ \delta_{i}$. See Diagram (D3).

\begin{equation}\tag{D3}
\begin{tikzcd}
F_{n_i}\arrow[d,swap,"f_i"]&\arrow[l,swap,"\alpha^{n_{i+1}}_{n_i}"]
F_{n_{i+1}}\arrow[d,dotted,"f_{i+1}"]\\
K_i&K_{i+1}\arrow[l,"\beta_i"]
\end{tikzcd}
\end{equation}

All maps in Diagram (D4) are confluent epimorphisms and $\iLim\{K_i,\beta_i\}$ is isomorphic to $\mathbb G$.

\begin{equation}\tag{D4}
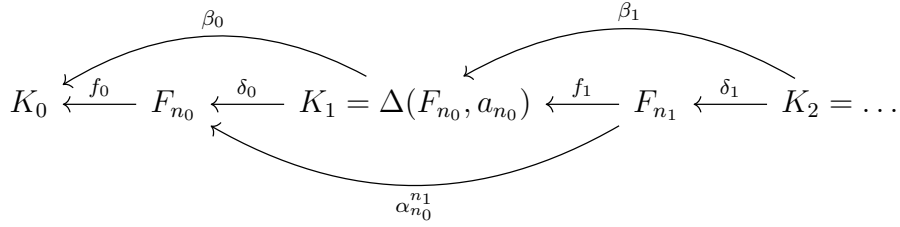
\begin{tikzcd}
K_0&\arrow{l}[swap]{f_0}F_{n_0}&K_1=\Delta(F_{n_0},a_{n_0}) \arrow[ll,bend right,swap,"\beta_0"]\arrow[l,swap,"\delta_{0}"]&F_{n_1}\arrow{l}[swap]{f_1}\arrow[ll,bend left,"\alpha^{n_1}_{n_0}"]&K_2= \arrow[ll,bend right,swap,"\beta_1"] \arrow[l,swap,"\delta_{1}"]\dots
\end{tikzcd}
\end{equation}

For any 0,1-sequence  $s$ we use $\len(s)$ to denote 
the length of the sequence $s$ and, for $n<\len(s)$,  $s|_n$ to denote the sequence consisting of the first $n$ entries of $s$.

Now we are going to construct arcs in the graphs $K_n$ whose inverse limits are arcs in $\mathbb G$.
Let $A$ with $|A|\geq 2$ be an arc in $K_0$ such that  $f_0(a_{n_0})$ is an end-vertex of $A$.
Let $A_\emptyset =A$; by  Lemma \ref{arcs-in-graphs} there is an arc 
$B_\emptyset$ in $F_{n_0}$ such that $f_0(B_\emptyset)=A_\emptyset$ and $a_{n_0}$ is an end-vertex of $B_\emptyset$.
By the construction of $K_1$ there are two arcs $A_0$ and $A_1$ in $K_1$
with $A_0\cap A_1=f_1(a_{n_1})$ and $\delta_{0}(A_0)=\delta_{0}(A_1)=B_\emptyset$; consequently 
$\beta_0(A_0)=\beta_0(A_1)=A_\emptyset$. By  Lemma \ref{arcs-in-graphs}, there are  arcs 
$B_k$, $k=0,1$,  in $F_{n_1}$ such that $f_1(B_k)=A_k$ and $a_{n_1}$ is an end-vertex of each $B_k$.

Using this procedure repeatedly, we may construct, for any finite 0,1-sequence $s$ an arc $A_s$ in $K_{\len(s)}$ such that the following conditions are 
satisfied:
\begin{enumerate}
    \item $A_\emptyset =A$;
    \item if $\len(s)=k+1$, then $A_{s|_k}=\beta_k(A_s)$ and $\beta_k|_{A_s}$ is monotone;
    \item if $\len(s)=\len(s')=k$ and $s\ne s'$, then $A_{s}\cap A_{s'}=f_{k}(a_{n_k})$. 
\end{enumerate}

For an infinite 0,1-sequence $s$, we define $s_k=s|_k$ and 
${\bf A}_s=\iLim \{A_{s_k}, \beta_k|_{A_{s_{k+1}}}\}$.  Then the union
$${\bf L}=\bigcup \{{\bf A}_s\colon s \text{ is an infinite 0,1-sequence}\}$$ 
is, by \cite[Proposition 5.12]{WJC-RPR-Fraisse}, a smooth dendroid in $\mathbb G$.  
Indeed,  $D_k=\bigcup\{A_s\colon \len(s)=k\}  $
is the  disjoint union of a finite number of arcs identified at the vertex $f_k(a_{n_k})$,  $\beta_k|_{ D_{k+1}}$ maps end-vertices to end-vertices, and 
${\bf L}=\iLim \{ D_k, \beta_k|_{ D_{k+1}}\}$.
The topological graph ${\bf L}$ with the root ${\bf a}$ is a fan 
(Definition \ref{def:fan}). 
The set of end-vertices of ${\bf L}$ is closed and has no isolated vertices.  

It follows that the set $\pi(\bf L)$ is a smooth fan in the topological sense, with the top $\pi({\bf a})$. The set of endpoints is closed and has no isolated points, therefore it is homeomorphic to the Cantor set.  Thus $\pi({\bf L})$ is homeomorphic to the Cantor fan. 
\end{proof}

\section{Pointwise self-homeomorphic graphs}

\begin{definition}
A topological space $X$ is called {\em pointwise self-homeomorphic at a point $x\in X$} if for any neighborhood $U$ of $x$ there is a set $V$, not necessarily a neighborhood, such that $x\in V \subseteq U$ and $V$ is homeomorphic to $X$.  It is called {\em pointwise self-homeomorphic} if it is pointwise self-homeomorphic at each of its points, see \cite{Self-homeo}. A topological graph is {\em pointwise self-isomorphic} if in the above we consider topological graphs and instead of requiring a homeomorphism we require a topological graph isomorphism.
\end{definition}

In this section, show that $\mathbb G$ is pointwise self-isomorphic and use this to show that $|\mathbb G|$ is pointwise self-homeomorphic.

\begin{lemma}\label{extension}
Suppose $U$ and $W$ are finite connected graphs and $f\colon W\to U$ is a confluent epimorphism, and $U\subseteq G$, where $G$ is again a finite connected graph.  Then there is a finite connected graph $H$ such that $W\subseteq H$ and a confluent epimorphism $f^*\colon H\to G$ satisfying $f^*|_W=f$ and $W=(f^*)^{-1}(U)$.
\end{lemma}

\begin{proof}
Take $V(H)$ to be the disjoint union of $V(W)$ and $V(G\setminus U)$ and let 
\begin{align*}
E(H)= &\{\langle x, y\rangle \colon \langle x, y\rangle \in E(W) \cup E(G\setminus U) \text{ or }\\
&( x\in W, y\in G\setminus U, \langle f(x),y\rangle\in E(G))\}.
\end{align*}

Then $H$ is connected.  
 Let $f^*$ be equal to $f$ on $W$ and to the identity on $G\setminus U$. For any edge in $G$ there are three possibilities, its end-vertices are both in $U$, both in $G\setminus U$, or one is in $U$ and one is in $G\setminus U$. In all three cases, Condition~(3) of Theorem~\ref{confluent-edges} is satisfied, so $f^*$ is confluent.
\end{proof}

\begin{theorem}\label{small-copies}
Let $h\colon \mathbb{G} \to G$ be a confluent epimorphism onto a finite connected graph and $x\in G$. Then any component of $h^{-1}(x)$ is isomorphic with $\mathbb{G}$.
\end{theorem}

\begin{proof}
Let $h\colon \mathbb{G} \to G$ and $x$ be as in the statement of the theorem. 
Let $\{F_n,\alpha_n\}$ be a Fra\"{\i}ss\'e sequence for $\mathcal G$. 
Let $k$ and a confluent epimorphism $h_0\colon F_k\to G$ be such that $h=h_0\circ \alpha^\infty_k$,    let ${\bf C}$ be a component of $h^{-1}(x)$, and write
 $C_i=\alpha^\infty_i({\bf C})$.
Note that for $i\geq k$, $C_i$ is a component of $(\alpha_k^i)^{-1} h_0^{-1}(x)$.
By Proposition \ref{confluent-restrictions},  $\bar{\alpha}_m^n=\alpha^n_m |_{C_n}$ are confluent whenever $m\leq n$. 
Using Proposition \ref{B-CThm}, we show that $\{C_n, \bar{\alpha}_n\}$ is a Fra\"{\i}ss\'e sequence for $\mathcal G$.

For that, pick $l$ and a confluent epimorphism $f\colon A\to C_l$, where $A\in\mathcal{G}$. Apply Lemma \ref{extension} to $U=C_l$, $W=A$, and $G=F_l$. Obtain $H\in\mathcal{G}$ and a confluent epimorphism $f^*\colon H\to G$, which extends $f$ and $A=(f^*)^{-1}(C_l)$.
Since $\{F_n,\alpha_n\}$ is a Fra\"{\i}ss\'e sequence for $\mathcal G$, we get $n$ and a confluent epimorphism $g\colon F_n\to H$ such that
$\alpha^n_l=(f^*)\circ g$. By confluence of $g$, since $A$ is connected, and using $A=(f^*)^{-1}(C_l)$, we have that 
$\bar{g}=g|_{C_n}$ is
onto $A$ and, by Proposition \ref{confluent-restrictions}, it is confluent. Clearly $\bar{\alpha}^n_l=f \bar{g}$.

\end{proof}

\begin{corollary}\label{pointwise}
The \F limit ${\mathbb G}$ is pointwise self-isomorphic.
\end{corollary}
\begin{proof}
Let ${\bf a}\in{\mathbb G}$ and $\varepsilon >0$. Then, by condition (3) of Theorem \ref{limit}, 
there is a graph $G\in \mathcal G$ and a 
confluent $\varepsilon$-epimorphism 
$f_G\colon \mathbb G\to G$. 
 By Theorem \ref{small-copies} the component of $f_G^{-1}(f_G({\bf a}))$ that contains ${\bf a}$ is a subgraph of ${\mathbb G}$
isomorphic to ${\mathbb G}$.

\end{proof}

\begin{observation}
If ${\bf G}$ is a pointwise self-isomorphic topological graph with a transitive edge relation, then its topological realization $|{\bf G}|$ is a pointwise self-homeomorphic compactum.
\end{observation}

\begin{corollary}
The continuum $|\mathbb G|$ is pointwise self-homeomorphic.
\end{corollary}

\begin{corollary}
Let $X$ be a continuum, which embeds in $|\mathbb G|$. Then for any nonempty open set $U\subseteq |\mathbb G|$, $X$ embeds in $U$.
\end{corollary}

\section{ Hereditarily  unicoherence} \label{sec-uni}
The goal of this section is to show that the graph $\mathbb G$ and, consequently, its topological realization are hereditarily unicoherent (see Definition \ref{her-uni}).
We start with a general observation.

\begin{lemma}\label{herunidown}
Let $\bf F$ be a hereditarily unicoherent topological graph having a transitive edge relation. Then $|\bf F|$ is hereditarily unicoherent. In fact, if $|\bf F|$ is not hereditarily unicoherent, then $\bf F$ is not subgraph  hereditarily unicoherent. 
\end{lemma}

\begin{proof}
Suppose $|\bf F|$ is not hereditarily unicoherent.  Then there exist subcontinua $H$ and $K$ in $|\bf F|$ such that $H \cap K$ is disconnected. So $H\cap K= C\cup D$ where $C \not = \emptyset \not = D$ and $\overline{C} \cap D= \emptyset =C \cap \overline{D}$. 

Note that the sets $\pi ^{-1} (H)$ and $\pi ^{-1} (K)$ are closed since the quotient map $\pi$ is continuous. The set $\pi ^{-1}(H)$ is connected, otherwise the image of the sets under $\pi$, that form the disconnection of $\pi ^{-1}(H)$, would disconnect $H$. Likewise, $\pi^{-1}(K)$ is connected. Finally, $\pi ^{-1}(H) \cap \pi ^{-1}(K) = \pi ^{-1}(C) \cup \pi ^{-1}(D)$ and, since $\pi$ maps the vertices of a non-degenerate edge to a single point, there is no edge between $\pi ^{-1}(C)$ and $\pi ^{-1}(D)$. Thus $\pi ^{-1}(C) \cup \pi ^{-1}(D)$ is not connected, contradicting the subgraph 
 hereditary unicoherence of  $\bf F$
\end{proof}

We start with a definition and then we prove three general lemmas on subgraph hereditary unicoherence of \F limits.

\begin{definition}
Given a closed connected topological graph ${\bf G}$ a quadruple $\langle {\bf H},{\bf K},{\bf A},{\bf B}\rangle$
is called a {\it cycle division in ${\bf G}$} if the following conditions are satisfied: 
\begin{enumerate}
    \item ${\bf H}$ and ${\bf K}$ are closed connected subgraphs of ${\bf G}$
    such that ${\bf H}\cap {\bf K}\neq\emptyset$;
    
    \item ${\bf A}$ and ${\bf B}$ are  disjoint clopen subgraphs of ${\bf G}$ 
        such that ${\bf A}\cap ({\bf H}\cap {\bf K})\neq\emptyset$ and
     ${\bf B}\cap ({\bf H}\cap {\bf K})\neq\emptyset$; 
     
    \item ${\bf H}\cap {\bf K}\subseteq {\bf A}\cup {\bf B}$;
    
\item 
${\bf H}\setminus ({\bf A}\cup {\bf B})$ and 
${\bf K}\setminus ({\bf A}\cup {\bf B})$ are disjoint;

    \item there is no edge between ${\bf A}$ and ${\bf B}$, that is, if ${\bf a}\in {\bf A}$ and ${\bf b} \in {\bf B}$ then $\langle {\bf a},{\bf b}\rangle \not \in E({\bf G}).$
\end{enumerate}
\end{definition}
Note that  \blu{(4)} in the definition of cycle division \blu{follows from (3) and further (5)} implies that ${\bf A}$ and ${\bf B}$ are disjoint. Connectedness of ${\bf H}$ and ${\bf K}$ implies that ${\bf H}\setminus ({\bf A}\cup {\bf B})$ and ${\bf K}\setminus ({\bf A}\cup {\bf B})$ are nonempty.

In the following three lemmas we do not require the epimorphisms to be confluent.
\begin{lemma} \label{cycle_div_down}
     Let  $\langle {\bf H}, {\bf K}, {\bf A}, {\bf B}\rangle$ be a cycle division in $\mathbb F$, where $\mathbb F$ is a connected topological graph. Then there is an $\varepsilon>0$ such that for any $\varepsilon$-epimorphism, $f\colon  \mathbb F \to F$, where $F$ is a finite connected graph, we have that
    $\langle f({\bf H}),f({\bf K}),f({\bf A}),f({\bf B})\rangle$ is a cycle division in $F$.
    
In particular,  if $g\colon  \mathbb F\to G$ and $h\colon G\to F$, where $G$ is a finite connected graph, are  epimorphisms such that $h\circ g=f$, then
 $\langle g({\bf H}),g({\bf K}),g({\bf A}),g({\bf B})\rangle$ is a cycle division in $G$.

\end{lemma}

\begin{proof}
We first show that there is an $\varepsilon_1>0$ such that for any $\varepsilon_1$-epimorphism, $f\colon  \mathbb F \to F$, where $F$ is a finite connected graph,  Condition (5) holds for $\langle f({\bf H}),f({\bf K}),f({\bf A}),f({\bf B})\rangle$. Towards this, assume that for all $\varepsilon >0$ and for all $\varepsilon$-epimorphisms $f_{\varepsilon}\colon \mathbb F \to F_{\varepsilon}$, where $F_{\varepsilon}$ is a finite connected graph, there are vertices $a_{\varepsilon} \in f_{\varepsilon}({\bf A})$ and $b_{\varepsilon} \in f_{\varepsilon}({\bf B})$  such that $\langle a_{\varepsilon}, b_{\varepsilon}\rangle \in E (F_{\varepsilon})$. Let $\varepsilon_n \to 0$ and $f_n\colon \mathbb F \to F_n$ be an $\varepsilon_n$-epimorphism such that $\langle a_n,b_n\rangle \in E(F_n)$, $a_n\in f_n({\bf A})$, and $b_n\in f_n({\bf B})$. Since $f_n$ is an epimorphism there exists ${\bf a}_n' \in f_n^{-1}(a_n)$ and ${\bf b}_n' \in f_n^{-1}(b_n)$ such that $\langle {\bf a}_n', {\bf b}_n'\rangle \in E(\mathbb F)$. By compactness of $\mathbb F$ we may assume that ${\bf a}_n' \to {\bf a}$ and ${\bf b}_n' \to {\bf b}$, for some ${\bf a},{\bf b}\in \mathbb F$. Since $E(\mathbb F)$ is closed, $\langle {\bf a}, {\bf b}\rangle \in E(\mathbb F)$. Because $f_n$ is an $\varepsilon$-map we have that $d({\bf a}_n', {\bf A}) < \varepsilon_n$ and $d({\bf b}_n', {\bf B}) < \varepsilon_n$. So ${\bf a} \in {\bf A}$ and ${\bf b} \in {\bf B}$. This contradicts Condition~(5) for $\langle {\bf H},{\bf K},{\bf A},{\bf B}\rangle$. Thus there is an $\varepsilon_1 > 0$ such that for any $\varepsilon_1$-epimorphism $f$ there is no edge between $f({\bf A})$ and $f({\bf B})$. \blu{In particular,  $f({\bf A})$ and $f({\bf B})$ are disjoint and $\varepsilon_1\leq d({\bf A}, {\bf B})$}. 

Let $\varepsilon >0$ be such that 
$$\varepsilon \le \min\{\varepsilon_1, d({\bf H}\setminus ({\bf A}\cup {\bf B}), {\bf K}\setminus ({\bf A}\cup {\bf B}))\}.$$

That $f({\bf H}) \cap f({\bf K}) \subseteq f({\bf A}) \cup f({\bf B})$ follow from the choice of $\varepsilon$ and Condition~(3) for $\langle {\bf H},{\bf K},{\bf A},{\bf B}\rangle$. Indeed, if  ${\bf x}\in {\bf H}$ and ${\bf y}\in {\bf K}$ are such that $f({\bf x})=f({\bf y})\notin f({\bf A})\cup f({\bf B})$, then ${\bf x}\in {\bf H}\setminus ({\bf A}\cup {\bf B})$, ${\bf y}\in {\bf K}\setminus ({\bf A}\cup {\bf B})$, and hence $d({\bf x}, {\bf y})>\varepsilon$, which is impossible as $f$ is an $\varepsilon$-map.  Condition (1) is clear, Condition (2) follows from $\varepsilon\leq d({\bf A},{\bf B})$, \blu{and Condition (4) follows from (3).}

\end{proof}

\begin{definition}
If ${\bf G}$ and ${\bf F}$ are topological graphs containing cycle divisions $C({\bf G})=\langle {\bf H}_G,{\bf K}_G,{\bf A}_G,{\bf B}_G\rangle$ and $C({\bf F})=\langle {\bf H}_F,{\bf K}_F,{\bf A}_F,{\bf B}_F\rangle$ respectively, and $f$ is a mapping from ${\bf G}$ to ${\bf F}$, then $f$ is said to {\it map $C({\bf G})$ onto $C({\bf F})$} if  $f({\bf H}_G)={\bf H}_F$, $f({\bf K}_G)={\bf K}_F$, $f({\bf A}_G)={\bf A}_F$, and $f({\bf B}_G)={\bf B}_F$.
\end{definition}

\begin{lemma}\label{exists c d}
    If $\mathbb F$ is a connected topological graph which is not subgraph hereditarily unicoherent then there is a cycle division $\langle {\bf H}, {\bf K}, {\bf A}, {\bf B}\rangle$ in~$\mathbb F$.
\end{lemma}

\begin{proof}
    If $\mathbb F$ is not subgraph hereditarily unicoherent there exist closed connected \blu{sub}graphs ${\bf H},{\bf K}$ \blu{of $\mathbb F$} such that ${\bf H} \cap {\bf K}$ is not connected.  Let ${\bf C}$ and ${\bf D}$ be closed disjoint, such that ${\bf C} \cup {\bf D} = {\bf H} \cap {\bf K}$, and there is no edge between a vertex in ${\bf C}$ and a vertex in ${\bf D}$.  We claim that there exists $\varepsilon>0$ such that there is no edge between the sets $N({\bf C},\varepsilon)$ and $N({\bf D},\varepsilon)$ where  $N({\bf M},\varepsilon)$ is the $\varepsilon$ neighborhood of the set ${\bf M}$. As in the proof of Lemma~\ref{cycle_div_down}, if we assume this is not the case, then for a sequence $\varepsilon_n \to 0$ there are sequences of vertices $({\bf c}_n)$ and $({\bf d}_n)$ with ${\bf c}_n \in N({\bf C},\varepsilon_n)$ and ${\bf d}_n \in N({\bf D},\varepsilon_n)$ and edges $\langle {\bf c}_n, {\bf d}_n\rangle \in E(\mathbb F)$ which converge to an edge $\langle {\bf c}, {\bf d}\rangle$ between vertices in ${\bf C}$ and ${\bf D}$, which gives a contradiction with the choice of ${\bf C}$ and ${\bf D}$. 
   Let $f\colon \mathbb F \to G$ be an $\varepsilon$-epimorphism onto a finite graph $G$. Then ${\bf A}=\bigcup\{f^{-1}(f({\bf c}))\colon {\bf c}\in {\bf C}\}$ and ${\bf B}=\bigcup\{f^{-1}(f({\bf d}))\colon {\bf d}\in {\bf D}\}$ are clopen sets such that ${\bf C} \subseteq {\bf A} \subseteq N({\bf C},\varepsilon)$ and ${\bf D} \subseteq {\bf B} \subseteq N({\bf D},\varepsilon)$. 
     Then $\langle {\bf H}, {\bf K}, {\bf A}, {\bf B} \rangle$ is a cycle division in $\mathbb F$. 
\end{proof}

\begin{lemma}\label{extension h u}
Suppose $\mathcal F$ is a projective \F family of finite connected graphs such that for each graph $F\in \mathcal F$ and for each cycle division $\langle H, K, A, B\rangle$ in $F$ there is a graph $G\in \mathcal F$ and an epimorphism $f^G_F\colon G \to F$ in  $\mathcal F$ such that no cycle division in $G$ is mapped by $f^G_F$ onto $\langle H,K,A,B\rangle$. Then the projective \F limit $\mathbb F$ of $\mathcal F$ is subgraph hereditarily unicoherent.
\end{lemma}

\begin{proof}
Suppose $\mathbb F$ is not subgraph hereditarily unicoherent.
Then by Lemma \ref{exists c d} there is a cycle division $\langle {\bf H},{\bf K},{\bf A},{\bf B}\rangle$ in $\mathbb F$. 
Let $f_F\colon \mathbb F \to F$ be an  epimorphism in  \blu{$\mathcal F^\omega$} such that
$\langle f_F({\bf H}),f_F({\bf K}),f_F({\bf A}),f_F({\bf B})\rangle$ is a cycle division in $F$, which exists by Lemma~\ref{cycle_div_down} and by (3) of Theorem \ref{limit}.
Take an epimorphism 
$f^G_F\colon G\to F$ satisfying the assumptions. Let $f_G\colon \mathbb F\to G$ be an epimorphism in  $\mathcal F$ such that $f^G_F\circ f_G=f_F$. 
Then, by Lemma~\ref{cycle_div_down},
$\langle f_G({\bf H}),f_G({\bf K}),f_G({\bf A}),f_G({\bf B})\rangle$ is a cycle division in $G$ that $f^G_F$ maps onto $\langle f_F({\bf H}),f_F({\bf K}),f_F({\bf A}),f_F({\bf B})\rangle$, contrary to the definition of $G$. 
\end{proof}

\begin{theorem}\label{heruni}
$\mathbb G$ is subgraph hereditarily unicoherent.
\end{theorem}

\begin{proof}
We verify the assumptions of Lemma \ref{extension h u}.
Let $F$ be a graph in $\mathcal G$ 
and $\langle H_F,K_F ,A_F, B_F\rangle$ be a cycle division in $F$. Denote $C_F=A_F\cap (H_F\cap K_F)$ and $D_F=B_F\cap (H_F\cap K_F)$.

Let $G$ be the graph where $$V(G)=\{(v,i)\colon v\in V(F) \text { and } i\in\{0,1\}\}.$$ Let 
$$E_{i,j}=\{\langle (x,i),(y,j)\rangle\colon x\in K_F\setminus C_F,  y\in C_F
\text{ and } \langle x,y\rangle \in E(F)\},$$ for $i,j\in\{0,1\}$. 

Let
$$E(G)= \left(\{\langle (x,i),(y,i)\rangle\colon i\in \{0,1\}, \langle x,y \rangle \in E(F)\} \cup E_{0,1}\cup E_{1,0} \right)
\setminus (E_{0,0}\cup E_{1,1}).
$$

 For simplicity, let us denote
$K_i=K_F\times\{i\}$, $C_i=C_F\times\{i\}$, $i=\{0,1\}$, viewed as  graphs induced from $G$, and we similarly define
$H_i$ and   $D_i$, and let $L_i=(K_i\setminus C_i)\cup H_i=K_i\cup H_i$.
The projection $\alpha\colon G \to F$ given by $\alpha((x,i)) = x$ is a confluent epimorphism, by Theorem~\ref{confluent-edges}.

We claim that graph $G$ is connected. To see this, note that the graphs $K'_0=C_0\cup (K_1\setminus C_1)$ and $K'_1=C_1\cup (K_0\setminus C_0)$ are each connected as $\alpha |_{K'_i} \to K_F$ is an isomorphism and $K_F$ is connected. Also, each of $H_0$ and $H_1$ is isomorphic to $H_F$, hence is connected. Since $K'_i\cap H_j\neq\emptyset$ for $i,j \in \{0,1\}$, we obtain that $L_0\cup L_1$ is connected. Finally, since $F$ is connected, for every $x\in G$ there is $a\in L_0\cup L_1$ and a path in $F\times\{0\}$ or in $F\times\{1\}$ (hence in $G$) joining $a$ and $x$.

Suppose there is a cycle division $\langle H_G, K_G, A_G, B_G\rangle$ in $G$ that is mapped by $\alpha$ onto $\langle H_F, K_F, A_F, B_F\rangle$. Since $H_G$ is connected and there is no edge in $G$ between $H_0$
and $H_1$ we have that $H_G= H_0$ or $H_G=H_1$. Without loss of generality, assume that $H_G=H_0$. 
Let $C_G=A_G\cap (C_0\cup C_1)$ and 
$D_G=B_G\cap (D_0\cup D_1)$.

We claim that $H_G\cap K_G\subseteq C_G\cup D_G$. 
Indeed, we have $\alpha(H_G\cap K_G)\subseteq \alpha(H_G)\cap \alpha(K_G)=H_F\cap K_F=C_F\cup D_F$. Therefore $H_G\cap K_G\subseteq \alpha^{-1}(C_F\cup D_F)=C_0\cup C_1\cup D_0\cup D_1$. Hence $H_G\cap K_G\subseteq A_G\cup B_G$ implies $H_G\cap K_G\subseteq C_G\cup D_G$. 
Furthermore  $K_G \subseteq K_F \times \{0,1\}=K'_0\cup K'_1$ and  there is no edge between $K'_0$ and $K'_1$.  

To finish the proof we show that $K_G\cap K'_0\neq\emptyset$ and $K_G\cap K'_1\neq\emptyset$, which would imply that $K_G$ is not connected, providing a contradiction.
Since $H_G\cap K_G\subseteq C_G\cup D_G$, $C_G\subseteq A_G$, and $A_G\cap D_G\subseteq A_G\cap B_G=\emptyset$,
we have that $C_G\cap H_G\cap K_G=A_G\cap H_G\cap K_G\neq\emptyset$. Furthermore, as $C_G\cap H_G\cap K_G\subseteq H_G\subseteq C_0$ and $C_G\cap H_G\cap K_G\neq\emptyset$, we get $C_0\cap K_G \neq\emptyset$, and so $ K'_0\cap K_G\neq\emptyset$. 
We similarly argue that $D_G\cap H_G\cap K_G\neq\emptyset$ and 
 $D_G\cap H_G\cap K_G\subseteq D_0$, and so $ D_0 \cap K_G \neq\emptyset$. As $D_0 \subseteq K_0\setminus C_0$, we get
$ (K_0\setminus C_0) \cap K_G \neq\emptyset$, and 
hence $ K'_1\cap K_G\neq\emptyset$.

\end{proof}

Theorem \ref{heruni} and Lemma \ref{herunidown} give the following corollary.

\begin{corollary}
$|\mathbb G|$ is hereditarily unicoherent.
\end{corollary}

\section{Non-homogeneity}\label{nine}

A continuum  $C$ is said to be {\em homogeneous} if for any $ x, y \in C$ there is a homeomorphism $f\colon C \to C$ such that $f( x)= y$. Likewise, a topological graph  ${\bf A}$ is said to be {\em homogeneous} if for any ${\bf x},{\bf y} \in {\bf A}$ there is an automorphism $f\colon {\bf A} \to {\bf A}$ such that $f({\bf x})={\bf y}$.

The main goal of this section is to prove the following theorem.
\begin{theorem}\label{mainhom}
The continuum $|\mathbb G|$ is not homogeneous.
\end{theorem} 

We will accomplish this by showing that there are points in $|\mathbb G|$ that
belong to a solenoid, as well as  there are points in $|\mathbb G|$ that do not
belong to a solenoid.

\bigskip \bigskip

\subsection{Solenoids and graph-solenoids}

The main result of this subsection is Theorem \ref{graphtosol}.  It shows that for many projective Fra\"{\i}ss\'{e} families of cycles with confluent epimorphisms  the topological realization of the projective Fra\"{\i}ss\'{e} limit
exists and is a solenoid.

We define a {\it solenoid} to be a continuum $\iLim\{X_i=S^1,z\mapsto z^{n_i}\}$, where $S^1$ is  the unit circle viewed in the complex plane and $n_i \geq 2$ for infinitely many $i$. 
Hagopian showed in \cite{Hagopian} that a solenoid can be characterized as a homogeneous continuum different from the circle, in which every proper non-degenerate subcontinuum is an arc. 

For $A\in\g$ we will say that $C\subseteq A$ is a {\em cycle} in $A$ if $|C|>2$ and we can enumerate the vertices of $C$ as $(c_0,c_1,\ldots, c_n=c_0)$ in a way that  $c_i\neq c_j$ whenever $0\leq i<j<n$  and  $\langle c_i,c_j\rangle \in E(A)$ if and only if $|j-i| \le 1$ and $0\leq i<j\leq n$.

 Note that if $f\colon D \to C$ is a confluent epimorphism between cycles then for every $a\in C$ and every component $A$ of $f^{-1}(a)$, if $x,y\in D$, $x\neq y$, are adjacent to $A$, then $f(x)\neq f(y)$.  Confluent epimorphisms between cycles will often be referred to as {\em wrapping maps}.
The {\em winding number} of a wrapping map $f$ is 
$n$ if for every (equivalently: some) $c\in C$, $f^{-1}(c)$ has exactly $n$ components. Note that the winding number is multiplicative, that is, the winding number of the composition of wrapping maps is the product of winding numbers of each of the maps in the composition.

We call the inverse limit of an inverse sequence of cycles $\{C_n, p_n\}$, where $p_n$ are confluent epimorphisms, a {\em graph-solenoid} ({\em graph-circle}) if for infinitely many $n$ the winding number of $p_n$ is greater than 1 (all winding numbers are equal to 1) and for every $x \in C_n$ every component of $p_n^{-1}(x)$ contains at least $2$ vertices.

\begin{lemma}\label{sol to graph sol}
Any solenoid $S$ is a topological realization of a graph-solenoid.
\end{lemma}
 
\begin{proof} 
Write  $S$ as $\iLim\{X_n, w_n\}$, where $X_n = S^1$ and $w_n\colon X_{n+1}\to X_n$ is given by $w_n(e^{it})=e^{i l_n t}$, for some $l_n\in\mathbb{N}\setminus\{ 0,1\}$.
Let $\mathcal{P}_1$ be the cover of $X_1$: 
$\{ [e^{0}, e^{i2\frac{\pi}{5}}], [e^{i2\frac{\pi}{5}}, e^{i 2\frac{2\pi}{5}}],\ldots,
[e^{i2\frac{4\pi}{5}}, e^{i 2\frac{5\pi}{5}}]\},$ and for every $n>1$ let  $\mathcal{P}_n$ be the cover of $X_n$: 
$$\{[e^{i \frac{m 2\pi}{L_n}}, e^{i \frac{(m+1)2\pi}{L_n}}]\colon m=0,1,\ldots, L_n-1\},$$
where $L_n=5\times 2^{n-1}\times l_1 l_2\ldots l_{n-1}$.
Let $S_n$ be the cycle with $V(S_n)=\mathcal{P}_n$ and $\langle A_n,B_n\rangle\in E(S_n)$ iff $A_n\cap B_n\neq\emptyset$, whenever $A_n,B_n\in V(S_n)$.
Then $w_n\colon X_{n+1}\to X_n$ induces a wrapping map $g_n\colon S_{n+1}\to S_n$ by $g_n(A_{n+1})=B_n$ iff $A_{n+1}\subseteq w_n^{-1}(B_n)$,
and consider the graph-solenoid $\mathbb{S}=\iLim\{S_n, g_n\}$. Then $|\mathbb{S}|=S$.
\end{proof}

We will later prove that conversely the topological realization of a graph-solenoid is a solenoid. This will follow from Theorem \ref{graphtosol2} and a special case of that is  Theorem \ref{graphtosol}.

Let $\mathcal{C}$ be the family of all cycles with all possible confluent epimorphisms.

\begin{lemma}\label{cycle confl}
The family $\mathcal{C}$ is a projective Fra\"{\i}ss\'{e} family.
\end{lemma}
\begin{proof}
Conditions (1) and (2) of Definition \ref{definition-Fraisse} are clear.  For condition~(3), given cycles $B$ and $C$ having $m$ and $n$ vertices, $m \ge n$, let $D$ be a cycle with $m$ vertices and $f,g$ be the obvious  epimorphisms. 

For condition~(4), given cycles $A$, $B$, and $C$ with $A=(a_0,a_1,\ldots ,a_k=a_0)$  and  confluent epimorphisms $f\colon B\to A$ and $g\colon C \to A$, let $N$ be the maximum of the cardinalities of the components of $f^{-1}(a_i)$ and of $g^{-1}(a_i)$, $0 \le i < k$, and let $W_f$ and $W_g$ be the winding numbers for $f$ and $g$. Let $D=(d_0,d_1,\ldots , d_M=d_0)$ be a cycle where $M=N\cdot W_f\cdot W_g\cdot k$. Define $f_0\colon D \to B$ and $g_0\colon D \to C$ as follows. For $0 \le i < k$ and $1 \le j \le W_f$, let $B_i^j$ be a component of $f^{-1}(a_i)$ such that $B_i^j$ and $B_{i+1}^j$ are adjacent, $B_{k-1}^j$ and $B_0^{j+1}$ are adjacent and $B_{k-1}^{W_f}$ and $B_0^1$ are adjacent. Similarly define $C_i^j$ as components of $g^{-1}(a_i)$. 

Since $|B_0^1|$ and $|C_0^1|$ have cardinality less than or equal to $N$ there are a monotone epimorphisms $f_0$ and $g_0$ from the arc $[d_0, d_{N-1}]$ onto $B_0^1$ and $C_0^1$ respectively such that $f_0(d_{N-1})$ and $g_0(d_{N-1})$ are adjacent to $B_1^1$ and $C_1^1$ respectively.  Next define $f_0\colon [d_N,d_{2N-1}] \to B_1^1$ and $g_0\colon [d_N,d_{2N-1}] \to C_1^1$ to be monotone epimorphisms such that $f_0(d_N)$ and $g_0(d_N)$ are adjacent to $B_0^1$ and $C_0^1$ respectively. Then $f_0(d_{2N-1})$ and $g_0(d_{2N-1})$ will be adjacent to $B_2^1$ and $C_2^1$ respectively. Continue this process to obtain the required confluent epimorphisms $f_0$ and $g_0$ from $D$ onto $B$ and $C$.
\end{proof}

We will use a few times the following factorization lemma, whose proof is similar to the proof of the amalgamation in Lemma \ref{cycle confl}.
\begin{lemma}\label{cycle confl2}
Let $C,D, T$ be cycles and let $f\colon D\to C$ and $g\colon T\to C$ be wrapping maps. Suppose that the winding number of $g$ is divisible by the winding number of $f$ and that for every $x\in C$ any component of $g^{-1}(x)$ has a cardinality not smaller than the cardinality of any component of $f^{-1}(x)$. Then there is a wrapping map $h\colon T\to D$ such that $g=f\circ h$.
    
\end{lemma}
\begin{proof}
     Let $C=(c_0,c_1,\ldots, c_k=c_0)$, let $W_g$ and $W_f$ be winding numbers of $g$ and of $f$, 
     and we define $h\colon T \to D$  as follows. For $0 \le i < k$ and $1 \le j \le W_g$, let $T_i^j$ be a component of $g^{-1}(c_i)$ such that $T_i^j$ and $T_{i+1}^j$ are adjacent, $T_{k-1}^j$ and $T_0^{j+1}$ are adjacent and $T_{k-1}^{W_g}$ and $T_0^1$ are adjacent. Similarly, for $0 \le i < k$ and $1 \le j' \le W_f$, let $D_i^{j'}$ be a component of $f^{-1}(c_i)$ such that $D_i^{j'}$ and $D_{i+1}^{j'}$ are adjacent, $D_{k-1}^{j'}$ and $D_0^{j'+1}$ are adjacent and $D_{k-1}^{W_f}$ and $D_0^1$ are adjacent.
     In particular, each $D^{j'}_i$ and $T^j_i$ is an arc.
     
     We let $h(T^j_i)=D^{j'}_i$, where $j'=j$ mod $ W_f$, be monotone, different end-vertices are mapped to different end vertices, and adjacent vertices from different $T^j_i$ are mapped to adjacent vertices in different $D^{j'}_i$.  
\end{proof}

\begin{lemma}\label{fundcycles}
Let $\{C_n, g_n\}$ be an inverse sequence of cycles with  confluent epimorphisms 
such that for all $n$ and $x\in C_n$,  every component of $g_{n+1}^{-1}(x)$ has at least two vertices.
Moreover, we assume that for every $k, N$ there is $l>k$ such that the winding number of $g^l_k$ is divisible by $N$. Then $\{C_n, g_n\}$ is a Fra\"{\i}ss\'e sequence for the family $\mathcal{C}$.
\end{lemma}

\begin{proof}
We check the condition in the statement of Proposition \ref{B-CThm}. Let $f\colon T\to C_m$, where $T$ is a cycle, be a confluent  epimorphism. Let $W_f$ be the winding number of $f$ and let $N$
be the maximum cardinality  of the components of $f^{-1}(x)$, $x\in C_m$. By hypothesis, we can find an $n$ such that $W_f$ divides the winding number of $g^n_m$ and for every $x\in C_m$ the minimum cardinality  of the components of $(g_m^n)^{-1}(x)$ is at least $N$. Then by Lemma \ref{cycle confl2} we can find a confluent  epimorphism $h\colon C_n\to T$ such that $fh=g^n_m$.

\end{proof}

Our goal now is to show that the topological realization of the
projective  Fra\"{\i}ss\'{e} limit of $\mathcal{C}$ is a graph-solenoid. We will achieve this in Theorem \ref{graphtosol}. In fact we will prove Theorem \ref{graphtosol} for many projective Fra\"{\i}ss\'{e}  family of cycles with confluent epimorphisms.


\begin{notation}\label{primes}
Let $P$  be a nonempty set of prime numbers and let $\mathcal{D}_P$ be the family of all cycles 
with confluent epimorphisms whose winding numbers are of the form $p_{1}^{n_1} p_{2}^{n_2} \ldots p_{k}^{n_k}$, where $n_1, n_2, \ldots, n_k\in \mathbb{N}$ and $p_1, p_2,\ldots, p_k\in P$. Then $\mathcal{D}_P$ is a projective Fra\"{\i}ss\'{e} family.

\end{notation}

\begin{lemma}
If $\mathcal{D}=\mathcal {D}_P$, then  \blu{$\mathcal{D}$} is consistent.
    
\end{lemma}

\begin{proof}
By the multiplicativity of winding numbers, for any choice of $g_i$ and $h_i$ as in Lemma \ref{ladder}, we have $g_i, h_i\in\mathcal{D}$, and hence
\blu{$\mathcal{D}$} is consistent.
    
\end{proof}

\begin{theorem}\label{graphtosol}
\blu{
Let $\mathcal{D}=\mathcal{D}_P$, where $P$ is a nonempty set of prime numbers.
Then $\mathcal{D}$ is a projective  Fra\"{\i}ss\'{e} family,  $\mathcal{D}$ is consistent, and 
the topological realization  $|\mathbb D |$  exists and is a solenoid. }
\end{theorem}

The proof that $\mathcal{D}_P$ is a projective  Fra\"{\i}ss\'{e} family is along the lines of the proof of Lemma \ref{cycle confl}. We have already justified that $\mathcal{D}_P$ is consistent and the topological realization $|\mathbb{D}_P|$ exists by
Theorem \ref{transitive-edges}. To show that $|\mathbb D_P |$ is a solenoid we first need a lemma.

\begin{lemma}\label{singlepoint}
\blu{Let  $\mathcal{D}=\mathcal{D}_P$, where $P$ is a nonempty set of prime numbers}.
Let ${\bf a}_0, {\bf a}_1\in \mathbb{D}$ be such that $\langle {\bf a}_0, {\bf a}_1 \rangle \in E(\mathbb{D})$. Then the topological graph $\mathbb{D}_0$ obtained from $\mathbb{D}$ by identifying ${\bf a}_0$ and ${\bf a}_1$ is isomorphic to  the projective Fra\"{\i}ss\'{e} limit $\mathbb{D}$.

\end{lemma}
\begin{proof}

We check (1), (2), and (3) of Theorem \ref{limit}.
Given $X \in \mathcal D$, let $X_0 \in \mathcal D$ be such that $|X_0| = |X|+1$. There exists a confluent epimorphism $f\colon \mathbb D \to X_0$.  If $f({\bf a}_0) \not = f({\bf a}_1)$ let $p\colon X_0 \to X$ be a confluent epimorphism such that $p(f({\bf a}_0))=p(f({\bf a}_1))$ and let $f_0\colon \mathbb D_0 \to X$ be defined by $f_0({\bf x}) = p(f({\bf a}_0))$ if ${\bf x} \in f^{-1}(f({\bf a}_0))\cup f^{-1}(f({\bf a}_1))$ and $f_0({\bf x})=p(f({\bf x}))$ otherwise. If $f({\bf a}_0)  = f({\bf a}_1)$ then $p$ can be taken to be any confluent epimorphism from $X_0$ to $X$ and $f_0$ is defined the same as previously. This shows (1).

Now we prove (2).
 Let $f_0\colon \mathbb{D}_0\to C$ and $g\colon D\to C$ be  confluent  epimorphisms, where $C,D$ are cycles in $\mathcal D$.  Let ${\bf a}$ be the vertex in $\mathbb{D}_0$  that is the identification of ${\bf a}_0$ and ${\bf a}_1$. 
Define $\rho\colon \mathbb{D} \to \mathbb{D}_0$ by $\rho({\bf a}_0) = \rho({\bf a}_1)={\bf a}$ and $\rho =\id$ otherwise. Clearly $\rho$ is a confluent epimorphism and let $f=f_0\circ \rho$. By property~(2) for $\mathbb D$ there is a confluent epimorphism $w\colon \mathbb{D} \to D$ such that $g\circ w = f$. If $w({\bf a}_0)= w({\bf a}_1)$ then $w'\colon \mathbb{D}_0 \to D$ defined as $w'({\bf a})= w({\bf a}_0)=w({\bf a}_1)$ and $w' = w$ otherwise satisfies $g\circ w'=f_0$.
In the case  $w({\bf a}_0) \not = w({\bf a}_1)$ take cycle $E$ such that $|E|=2|D|$ and $h\colon E\to D$ let be the wrapping map that has the winding number equal to 1 and preimage of each vertex is of cardinality 2. By property~(2) for $\mathbb{D}$ we get $p\colon\mathbb{D}\to E$ such that $f=(gh)\circ p$.  Let $E_0$ be equal to $E$ with $p({\bf a}_0)$ and $p({\bf a}_1)$ identified
and let $\rho_1\colon E\to E_0$ be the map that identifies $p({\bf a}_0)$ and $p({\bf a}_1)$.
Let $p_0\colon \mathbb{D}_0\to E_0$ be the quotient of $p$. Then we have $p_0\circ \rho=\rho_1\circ p$. 
Denote $t=\rho_1(p({\bf a}_0))$ and let
 $h_0\colon E_0\to D$ be equal to $h$ outside $t$, and let $h_0(t)=h(p({\bf a}_0))$.
Then the epimorphism $p_1=h_0\circ p_0\colon \mathbb{D}_0\to D$ is as required, that is, $g\circ p_1=f_0$. 

Finally, we show (3). Let $d$ be a metric on $\mathbb D$.
Given $\varepsilon>0$ there is $C \in \mathcal D$ and a confluent epimorphism $f \colon \mathbb{D} \to C$ such that for all $x \in C$ the $\diam_d(f^{-1}(x)) < \varepsilon/2$. If $f({\bf a}_0) \not = f({\bf a}_1)$ let $C_0$ be the cycle obtained from $C$ by identifying $f({\bf a}_0)$ and $f({\bf a}_1)$;  otherwise let $C_0=C$. Let $f_0\colon \mathbb{D}_0 \to C_0$ be the quotient of $f$. 
Let $d_0$ be the quotient metric on $\mathbb D_0$, that is, $d_0([{\bf a}_0],{\bf x})= \min\{d({\bf a}_0,{\bf x}),d({\bf a}_1,{\bf x})\}$ and $d_0({\bf x},{\bf y})= d({\bf x},{\bf y})$ for ${\bf x},{\bf y}\neq {\bf a}_0,{\bf a}_1$. Then $f_0$ is an $\varepsilon$-epimorphism with respect to $d_0$. 

\end{proof}

\begin{proof}[Proof of Theorem \ref{graphtosol}]

 By the result of Hagopian we have to show that the topological realization is homogeneous and that every proper non-degenerate subcontinuum is an arc.

Let  $\{C_n, g_n\}$ be a
Fra\"{\i}ss\'e sequence for $\mathcal{D}$.
Let $K\subseteq |\mathbb{D}|$ be a proper subcontinuum and consider $\pi_D^{-1}(K)$,  where $\pi_D\colon \mathbb D\to |\mathbb D|$ is the quotient map.
This is a proper topological subgraph of $\mathbb{D}$, which is connected.
Since any proper connected subgraph of a cycle is an arc, in fact
$\pi^{-1}_D(K)$ is the inverse limit of  $\{I_n, g_n|_{I_{n+1}}\}$, where $I_n\subseteq C_n$ is an arc from some point on. Note also that if $I_n$ is an arc, then $g_n|_{I_{n+1}}$ is monotone, and hence $\pi^{-1}_D(K)$ is an arc. Moreover, $\pi_D$ takes endpoints of $\pi_D^{-1}(K)$ to distinct points. Therefore $K=\pi_D (\pi^{-1}_D(K))$ is an arc.

Let $A,B\in |\mathbb{D}|$.
 Applying  Lemma \ref{singlepoint} twice, we can assume that $|\pi_D^{-1}(A)|=1$ and $|\pi_D^{-1}(B)|=1$. Let ${\bf a}\in\mathbb{D}$ be the only vertex such that $\pi_D({\bf a})=A$ and similarly let ${\bf b}\in \mathbb{D}$ be the only vertex such that $\pi_D({\bf b})=B$.
We show that there is an automorphism 
of $\mathbb{D}$ that takes ${\bf a}$ to ${\bf b}$. This automorphism will induce a homeomorphism of $|\mathbb{D}|$ that takes $A$ to $B$.
It suffices to show the following.

\noindent {\bf{Claim:}}  Let ${\bf c}=(c_n)\in \mathbb{D}$, i.e. $c_n\in C_n$,  be such that $|\pi_D^{-1}(\pi_D({\bf c}))|=1$. Let $m$, a cycle $T$, and a confluent epimorphism $f\colon T \to C_m$, be given.
 Take $t\in T$  with $f(t)=c_m$.   Then there is $n>m$ and $h\colon C_n\to T$ such that $h(c_n)=t$ and $fh=g^n_m$.

Given the Claim, we can find sequences 
 $(n_i)$ and $(m_j)$ and  confluent epimorphisms $f_i\colon C_{n_i}\to C_{m_i}$ and $h_i\colon C_{m_{i+1}}\to C_{n_i}$ such that
$g^{n_{i+1}}_{n_i}=h_i f_{i+1}$ and $g^{m_{i+1}}_{m_i}=f_ih_i$ 
and $f_i(a_{n_i})=b_{m_i}$
and $h_i(b_{m_{i+1}})=a_{n_i}$. These maps induce an automorphism of $\mathbb{D}$ that takes $a$ to $b$.

\begin{proof}[Proof of Claim]
Let $I_n$ for $n\geq m$ be the component of $(g_m^n)^{-1}(c_m)$ containing $c_n$.
Since $|\pi_D^{-1}(\pi_D({\bf c}))|=1$ there is $n$ such that $c_n$ is not an end-vertex of  $I_n$.
Therefore for any $k$ there is $n_k$ such that  the cardinality of each sub-arc of $I_{n_k}$ between $c_{n_k}$ and an end-vertex of $I_{n_k}$ is at least $k$.
Hence for large enough $n$ there is a monotone epimorphism from $I_n$ onto $J$ which takes $c_n$ to $t$,
where $J$  is the component of $f^{-1}(c_m)$ containing $t$.
Using this last property we find the required $h$ exactly as in the proof of
Lemma~\ref{cycle confl2}. 
\end{proof}
\end{proof}

\subsection {Points in $|\mathbb{G}|$ that belong to a universal  solenoid.  } In this subsection we show the following theorem.
\begin{theorem}\label{maininsol}
There is a dense set of points in $|\mathbb{G}|$ that belong to a universal solenoid. 
\end{theorem}  

We start with some definitions and several lemmas.
\smallskip

We say that ${\bf x} \in \mathbb G$ {\em belongs to a graph-solenoid } if there is a graph-solenoid  $\mathbb S$ and an embedding $h\colon {\mathbb S} \to {\mathbb G}$ such that ${\bf x} \in h({\mathbb S})$.
Similarly, $y\in |\mathbb G|$ {\em belongs to a solenoid} if there is a solenoid $S$ and an embedding $j\colon S\to |\mathbb G|$ with $y\in j(S)$. A solenoid $S$ is said to be  {\it universal} if for any solenoid $S_1$  there is a continuous mapping from $S$ onto $S_1$.

\begin{lemma}\label{raz}
Let $A,B\in\g$ and let $f\colon B\to A$ be a confluent epimorphism.  Let $C=(c_0,c_1,\ldots, c_n=c_0)$ be a cycle in $A$.
Then there is a cycle  $D$ in $B$ such that  $f(D)=C$, and $f|_D$ is a wrapping map.
\end{lemma}

\begin{proof}
Note first that by the confluence of $f$, for every $x\in B$ and $0\leq i<n$, if $f(x)=c_i$ then there is a path
$x=x_0, x_1,\ldots, x_k, z$ in $B$, for some $k$, such that for every $j$, $f(x_j)=c_i$ and  $f(z)=c_{i+1}$. Using this simple observation, and that $B$ is a finite set, we obtain that for any $y\in B$ with $f(y)=c_0$
there is a path $y=y_0, y_1,\ldots, y_l$ such that: 1) for $i<l$ the $y_i$'s are pairwise different; 2) there is $i_0<l$
such that $y_l=y_{i_0}$ and $f(\{y_{i_0},\ldots, y_l\})=C  $;
 3)  the sequence  $f(y_0), f(y_1),\ldots, f(y_l)$
is of the form $c_0,\ldots, c_{n-1}, c_0, \ldots, c_{n-1}, \ldots, c_0, \ldots, c_t$,  for some $t\leq n-1$,
where in each block $c_0,\ldots, c_{n-1}$ we have a sequence of $c_0$'s, then a sequence of $c_1$'s, etc., 
and the block ends  with 
a sequence of $c_{n-1}$'s, and similarly for the block $c_0, \ldots, c_t $.
Note that $f|_{y_{i_0},\ldots, y_l}$ is a confluent epimorphism.

 The graph induced on the vertices $y_{i_0},\ldots, y_l=y_{i_0}$ may not be a cycle
 as there may be edges between non-consecutive vertices.
 We claim that there exists  $D\subseteq \{y_{i_0},\ldots, y_l\}$ such that the graph induced by $D$ is a cycle and $f|_D$ is a wrapping map onto $C$. To see this, start with a graph $K$ such that 
 $V(K)=\{y_{i_0},\ldots, y_l=y_{i_0}\}$ and $\langle y_i,y_j\rangle \in E(K)$ iff $|i-j|\leq 1$ and $i_0\leq i<j\leq l$.
 Suppose there are vertices $y_i,y_j$ in $K$ with $|i-j|>1$ and $\langle y_i,y_j\rangle \in E(B)$. 
 Let $K^+$ be the graph such that $V(K^+)=V(K)$ and $E(K^+)=E(K)\cup \{\langle y_i,y_j\rangle\}$.
  Let $K_R$ and $K_L$ be cycles in $K^+$ such that \blu{ $V(K_R)\cup V(K_L)=V(K)$, $E(K_R)\cup E(K_L)=E(K)\cup\{  \langle y_i,y_j\rangle\}$,} and $E(K_R)\cap E(K_L)= \langle y_i,y_j\rangle$. We claim that $f|_{K_R}$ or $f|_{K_L}$ is confluent.

 Since $f$ is an epimorphism and $C$ is a cycle,  $f(y_i)$ and $f(y_j)$ must either map onto the same vertex in $C$ or onto adjacent vertices in $C$.  
 
 In the case that $f(y_i)=f(y_j)$ there are two subcases: (1) All of the vertices in one of $K_R$ or $K_L$ map to $f(y_i)$ and we may assume $f(K_R)=f(y_i)$. Then $f|_{K_L}$ is a wrapping map from $K_L$ onto $C$. (2)~Both $K_R$ and $K_L$ contain vertices that map to vertices different from $f(y_i)$. Then $f$ restricted to either of $K_R$ or $K_L$ is a wrapping map onto~$C$.  
 
 In the case that $f$ maps $y_i$ and $y_j$ onto adjacent vertices in $C$ there are again two subcases: (1) The image of one of $K_R$ or $K_L$ is a subset of $\{f(y_i),f(y_j)\}$ and we may assume $f(K_R)=\{f(y_i),f(y_j)\}$.
 Then $f|_{K_L}$ is a wrapping map from $K_L$ onto $C$. 
 (2)~The image of both $K_R$ and $K_L$ contains vertices distinct from $f(y_i)$ and $f(y_j)$. In this case  one of $f|_{K_R}$ or $f|_{K_L}$ will not be confluent but the other will be confluent. In every case we obtain a cycle with fewer vertices than $K$ that $f$ maps confluently onto $C$. Since the graphs are finite, repeated applications of this process will 
produce the desired cycle $D$. 

\end{proof} 

\begin{remark}
Note that it may not be the case that for every $b\in B$ with $f(b)=c_0$ there is a cycle $D\subseteq B$ 
with $b\in D$ such that $f(D)=C$.
Compare with Lemma \ref{arcs-in-graphs}. 
\end{remark}

\begin{lemma}\label{dwa}
For any $A\in\g$, $m>0$ and cycle $C$ in $A$ there is $B\in \g$, a cycle $D$ in $B$ and a confluent epimorphism $p\colon B\to A$ such that $p(D)=C$\blu{, $p|_D$ is confluent,} and the winding number of $p|_D$ equals $m$.
\end{lemma}
\begin{proof}
Consider $m$ disjoint copies of $A$: $A_1,\ldots, A_m$. Pick any $x,y\in C$ with $\langle x,y\rangle \in E(A)$. The copy of $x$ in $A_i$
we call $x_i$,  the copy of $y$ in $A_i$ we call $y_i$,
and the copy of $C$ in $A_i$ we call $C_i$. From the disjoint union $A_1\cup A_2\cup\ldots\cup A_m$
remove edges: $\langle x_1,y_1\rangle, \langle x_2,y_2\rangle,\ldots, \langle x_m,y_m\rangle$. Add edges: $\langle x_1, y_2\rangle, \langle x_2,y_3\rangle, \langle x_3,y_4\rangle,\ldots, \langle x_m,y_1\rangle$.
This is the required $B$. Note that $B$ is connected (as removing an edge in the cycle $C$ does not disconnect $A$) and that the canonical projection $p\colon B\to A$ is a confluent epimorphism. 
Let $D$ be the finite graph where $V(D)=V(C_1)\cup\ldots\cup V(C_m)$ and edges of $D$ being all of the edges in the union of the $C_i$'s except for the deletions and additions above. Since $p^{-1}(D)=C$ we have by Theorem~\ref{confluent-restrictions} that $p|_D$ is confluent. Clearly, the winding number of $p$ equals $m$.
\end{proof}

\begin{lemma}\label{trzy}
For any $A\in\g$ and $a\in A$ there is $B\in \g$, a confluent epimorphism $f\colon B \to A$, and $b\in B$ with $f(b)=a$ such that $b$ belongs to a cycle in $B$.
\end{lemma}
\begin{proof}

Let $\langle a,c\rangle$ be a non-degenerate edge in $A$. To define $B$ let $b\not \in A$, $V(B)=V(A)\cup \{b\}$ and $E(B)=E(A)\cup \{\langle b,a\rangle, \langle b,c\rangle\}$. The map that takes $b$ to $a$, $\langle b,a\rangle$ to $\langle a,a\rangle$ and $\langle b,c\rangle$ to $\langle a,c\rangle$ and is the identity otherwise is confluent and $(a,b,c,a)$ is a cycle in $B$.
\end{proof}

\begin{lemma}\label{double}
For any $A\in\g$ and cycle $C$ in $A$ there is $B\in \g$, a cycle $D$ in $B$ and a confluent epimorphism $f\colon B\to A$
such that $f(D)=C$, \blu{$f|_D$ is confluent,} the winding number of $f$ is one, and for every $x\in C$, $|f^{-1}(x)\cap D|=2$. 
\end{lemma}
\begin{proof}

Write $C=(c_0,c_1,\ldots, c_n=c_0)$, let $c'_0, c'_1, \ldots, c'_n$ be vertices not in $A$, and let $D=(c_0,c'_0,c_1,c'_1,\ldots,c_n,c'_n=c_0)$. First take $B_0$, the disjoint union of the graphs $A\setminus C$ and $D$ and let $\alpha\colon D\to C$ be the map that takes $c_i$ and $c'_i$ to $c_i$, for every $i$. We then obtain $B$ from $B_0$ by adding edges $\langle x, y\rangle$ 
 iff $x\in A\setminus C$, $y\in D$, and $\langle x, \alpha(y)\rangle \in E(A)$. Let $f$ be the identity on $A \setminus C$ and $f$ be $\alpha$ on $D$. Then $f$ is a confluent epimorphism and since $f^{-1}(C)=D$ we have by Theorem~\ref{confluent-restrictions} that $f|_D$ is confluent.  
\end{proof}

\begin{theorem}\label{belongs}
The set of vertices in $\gg$ that belong to a graph-solenoid, which is a projective Fra\"{\i}ss\'{e} limit of the family $\mathcal{C}$ of cycles with confluent epimorphisms, is dense. 
\end{theorem}
\begin{proof}
Fix $\{F_n,\alpha_n\}$ a Fra\"{\i}ss\'e sequence for the family $\mathcal G$.
Given any $m$ and $a \in F_m$  apply Lemma~\ref{trzy}, with $A=F_m$ to obtain $B \in {\mathcal G}$ and a confluent epimorphism  $f \colon B \to F_m$ where $B$ contains a cycle $C$ such that  $a \in f(C)$. From the properties of the Fra\"{\i}ss\'e sequence, there is $n > m$ and a confluent epimorphism  $g \colon F_n \to B$ such that $fg=\alpha_m^n$. By Lemma \ref{raz} there is  a cycle $E\subseteq F_n$ which satisfies $g(E)=C$  and $g|_E$ is a wrapping map. Let $a_1\in E$ be such that $fg(a_1)=a$. Write $m_1=n$ and $C_1=E$. 

Applying  Lemmas \ref{dwa}, \ref{double}  and \ref{raz}, construct a sequence $m_1<m_2<\ldots$ as well as $a_i\in F_{m_i}$ that belongs to a cycle $C_i$
in $F_{m_i}$,  $\alpha^{m_{i+1}}_{m_i}(C_{i+1})=C_i$ and the winding number of
$\alpha^{m_{i+1}}_{m_i}|_{C_{i+1}}$ is a multiple of $i$. Then the sequence
$(a_i)$ induces a vertex 
in $ \gg$ which belongs to a graph-solenoid.

By the construction and Lemma \ref{fundcycles}, the inverse sequence $\{C_i, \alpha^{m_{i+1}}_{m_i}|_{C_{i+1}}\}$ is a Fra\"{\i}ss\'e sequence for $\mathcal{C}$.

\end{proof}

\begin{proof}[Proof of Theorem \ref{maininsol}]
Let $\mathbb{C}$ denote the projective Fra\"{\i}ss\'e limit of $\mathcal{C}$. 
Theorem \ref{graphtosol} gives that $|\mathbb{C}|$ is a solenoid and
Theorem \ref{belongs} implies that the set of points in $|\mathbb{G}|$ that belong to an isomorphic copy of $|\mathbb{C}|$ is dense.

We claim that $|\mathbb{C}|$ is universal, which will finish the proof. Indeed, let $S$ be a solenoid and write  $\mathbb{C}$ 
as the inverse limit of a \fra \ sequence.
By Lemma \ref{sol to graph sol}, there is a graph-solenoid
$\mathbb{S}$ such that $|\mathbb{S}|=S$. It follows from Definition \ref{fund-seq-def} (2), that there is an epimorphism from $\mathbb{C}$ onto $\mathbb{S}$. This epimorphism induces a continuous map from $|\mathbb{C}|$ onto $|\mathbb{S}|$, by Lemma \ref{quotient-emb}.

\end{proof}

\subsection{Almost graph-solenoids -- definition and properties}

Solenoids may be obtained as topological realizations of topological graphs other than graph-solenoids, see
Theorems \ref{soltograph} and \ref{soltograph-main}.  
To prove Theorem~\ref{mainhom}, we also have to understand 
certain more general topological graphs whose topological realization is a solenoid.
In this subsection we define and study almost wrapping maps between cycles and correspondingly define almost graph-solenoids.

A surjective homomorphism $g\colon D\to C$, where $C,D$ are cycles with $C=(c_0, c_1,c_2,\ldots, c_k=c_0)$, $|C|\geq 4$, is an {\it almost wrapping map} if there is a confluent epimorphism $f\colon D\to C$ such that for every $y\in D$, if $f(y)=c_i$ then $g(y)\in \{c_i, c_{i+1}\}$.  
We call $f$ a {\it confluent witness} for $g$. Clearly, a confluent witness is not unique. Note that almost wrapping maps are epimorphisms. If for every $x \in C$ every component of $f^{-1}(x)$ contains at least $2$ vertices, then we call $f$ a {\it proper confluent witness} for $g$.  A {\it winding number} of $g$ is the winding number of $f$.
The definition of the winding number clearly does not depend on the choice of a confluent witness.
We call the inverse limit of an inverse sequence of cycles $\{C_n, p_n\}$, where $p_n$ are almost wrapping maps, an {\em almost graph-solenoid} ({\em almost graph-circle}) if for infinitely many $n$ the winding number of $p_n$ is greater than 1 (all winding numbers are equal to 1) and every $p_n$ has a proper confluent witness.

Whenever we consider almost wrapping maps between cycles, we will assume that these cycles are oriented and that almost wrapping maps are orientation preserving. If $C=(c_0, c_1,c_2,\ldots, c_k=c_0)$, then we let the {\it orientation} $<_C$ on $C$ to be the binary relation $<_C=\{(c_i, c_{i+1})\colon i=0,\ldots, k-1\}$, and we write $c_i<_C c_{i+1}$. 
An almost wrapping map $g\colon D\to C$, where $C,D$ are cycles, is {\it orientation preserving} if for some (equivalently any) confluent   witness $f\colon D\to C$ of $g$ we have: for every $x,y\in D$, if $x<_D y$, then $f(x)<_C f(y)$ or $f(x)=f(y)$.
For a cycle $C$ and $a,b\in C$, {\it an oriented arc}, which we simply denote by $[a,b]$,
is the proper subset of $C$, $\{a=x_0, x_1,\ldots, x_m=b\}$ such that for each $j=0,1,\ldots, m-1$, it holds $x_j<_C x_{j+1}$.  We allow that $a=b$ and in that case $[a,b]=\{a\}$. Note that unless $a=b$, then $[a,b]\neq [b,a]$. We also let $(a,b)=[a,b]\setminus \{a,b\}$, and similarly define $[a,b)$ and $(a,b]$.

The following definition for winding numbers for epimorphisms between cycles was introduced in \cite[Definition 4.2]{PhD-Irwin}. This definition extends our definition of winding numbers for almost wrapping maps. It is also proved that the winding numbers are multiplicative \cite[Lemma 4.3]{PhD-Irwin}, that is, if $f$, $g$, and $h$ are epimorphisms between cycles such that $f \circ g = h$ then the winding number of $h$ equals the product of the winding numbers of $f$ and $g$.

    For an epimorphism $f\colon D\to C$, where $D$ and $C$ are cycles with a fixed orientation, and any $a_1<_C a_2$ pick $n$ and $m$ as follows. Let $n$ be the number of $b_1<_D b_2$ with $f(b_1)=a_1$ and $f(b_2)=a_2$, and let  $m$ be the number of $b_1<_D b_2$ with $f(b_1)=a_2$ and $f(b_2)=a_1$. Then the {\it winding number} of $f$ is $n-m$. It does not depend on the choice of $a_1, a_2$.

In the next example, we present an almost graph-circle and an almost graph-solenoid, which are not isomorphic to a graph-circle or a graph-solenoid. It will be convenient to apply the above definition for winding number in 1) of the next example.

\begin{example}
Let $C$ be a cycle and let $k=|C|$. For a positive integer $l$ we let $lC$ be the cycle with $lk$ many vertices. 

1) We show that the inverse limit of cycles with bonding maps being almost wrapping maps may not be a graph-solenoid or a graph-circle. For a cycle $C=(c_0, c_1,c_2,\ldots, c_k=c_0)$ we write 
$3C=(d_0, d_1,d_2,\ldots, d_{3k}=d_0)$, and we let $f_C\colon 3C\to C$ 
to be the map given as follows. Let  $x=d_n\in 3C$  and let $n'$ be the integer part of $\frac{n}{3}$. If $n=0$ (mod 3), then let $f_C(x)=c_{n'}$, if $n=1$ (mod 3), then let $f_C(x)=c_{n'+1 \text{ (mod k)}}$, and if
$n=2$ (mod 3), then let $f_C(x)=c_{n'}$. Then $f_C$ is an almost wrapping map. 

Consider the inverse sequence $\{D_n=3^nD_0, f_n  \}$, where $f_n=f_{D_n}$ and $D_0$ is a cycle. We show that its inverse limit is not isomorphic to a graph-solenoid or a graph-circle.

Suppose towards the contradiction that there is a graph-solenoid or a graph-circle $ \varprojlim \{S_n, w_n\}$ isomorphic to  $ \varprojlim\{D_n, f_n\}$. Then, by Lemma~\ref{ladder}, there are sequences $(k_n)$ and $(l_n)$ and epimorphisms $h_n\colon S_{k_{n+1}}\to D_{l_n}$ and $g_n\colon  D_{l_n}\to S_{k_n}$ such that
$g_n\circ h_n=w^{k_{n+1}}_{k_n}$ and $h_{n-1}\circ g_n=f^{l_{n}}_{l_{n-1}}$.

Note that $f^n_m\colon D_n=3^{n-m}D_m\to D_m$ is given as follows. Write 
$D_m=(c_0, c_1,c_2,\ldots, c_k=c_0)$ and $D_n=(d_0, d_1,d_2,\ldots, d_{3^{n-m}k}=d_0)$.
Let  $x=d_t\in 3^{n-m}C$ and let $t'$ be the integer part of $\frac{t}{3^{n-m}}$. If $t$ is equal to an even number  (mod $3^{n-m}$), then let $f^n_m(x)=c_{t'}$, if $t$ is equal to an odd number (mod $3^{n-m}$), then let $f^n_m(x)=c_{t'+1 \text{ (mod k)}}$. 

Since $g_n\circ h_n=w^{k_{n+1}}_{k_n}$ and $w^{k_{n+1}}_{k_n}$ is confluent, we conclude, by Proposition~\ref{threemaps}, that
$g_n$ is also confluent. Since  $h_{n-1}\circ g_n=f^{l_{n}}_{l_{n-1}}$ and $f^{l_{n}}_{l_{n-1}}$ is not confluent $h_{n-1}$ is not confluent. Also,  because $f^{l_{n}}_{l_{n-1}}$ maps adjacent vertices to different vertices, $g_n$ maps adjacent vertices to different vertices. Since the winding number of  $f^{l_{n}}_{l_{n-1}}$ is 1, by the multiplicativity of winding numbers, so is the winding number of $g_n$. This implies that $g_n$ is an isomorphism. But then we have $w^{k_{n+1}}_{k_n}=h_n$. This is a contradiction since $h_n$'s are not confluent.

2) The almost graph-circle presented in 1) does not have a topological realization. Indeed, it is not hard to find $a_n, b_n, c_n\in D_n$ with $\langle a_n, b_n\rangle\in E(D_n)$ and $\langle b_n, c_n\rangle\in E(D_n)$ such that $f_n(a_{n+1})=a_n$, $f_n(b_{n+1})=b_n$, and $f_n(c_{n+1})=c_n$. For that reason, we will present a slightly more complicated example, which will have a topological realization. Additionally, we will obtain an almost graph-solenoid rather than just an almost graph-circle. We will take the previous example, but now components of preimages of vertices will be of cardinality equal to two (instead of one) and winding numbers will be equal to two (instead of one).

For a cycle $C=(c_0, c_1,c_2,\ldots, c_k=c_0)$ we write 
$12C=(e_0, e_1,e_2,\ldots, e_{12k}=e_0)$, and
let $t_C\colon 12C\to C$  
 to be the map given as follows.
Let  $x=e_n\in 12C$ and let $n'$ be the integer part of $\frac{n}{6}$. If $n=0,1$  (mod 6), then let $t_C(x)=c_{n' \text{ (mod k)}}$, if $n=2,3$ (mod 6), then let $t_C(x)=c_{n'+1 \text{ (mod k)}}$, and if
$n=4,5$ (mod 6), then let $t_C(x)=c_{n' \text{ (mod k)}}$. Then $t_C$ is an almost wrapping map.

Consider the inverse sequence $\{E_n=12^nE_0, t_n  \}$, where $t_n=t_{ E_n}$, and $E_0$ is a cycle.  We show that its inverse limit is not isomorphic to a graph-solenoid or a graph-circle.

Suppose towards the contradiction that there is a graph-solenoid or a graph-circle $ \varprojlim \{S_n, w_n\}$ isomorphic to  $ \varprojlim\{E_n, t_n\}$. Then as in 1) there are sequences $(k_n)$ and $(l_n)$ and there are epimorphisms $h_n\colon S_{k_{n+1}}\to E_{l_n}$ and $g_n\colon  E_{l_n}\to S_{k_n}$ such that
$g_n\circ h_n=w^{k_{n+1}}_{k_n}$ and $h_{n-1}\circ g_n=t^{l_{n}}_{l_{n-1}}$. Moreover, since $g_n\circ h_n=w^{k_{n+1}}_{k_n}$ and $w^{k_{n+1}}_{k_n}$ is confluent, we conclude that
$g_n$ is also confluent. Note that we are using that the system of cycles $S_n, E_n$ can be given orientations coherent with all
$w_n, t_n, g_n, h_n$.

We first claim that for every $n$ and every $x<_{E_{l_{n-1}}}y$ there are $y'<_{S_{k_n}} x'$ with $h_{n-1}(x')=x $ and $h_{n-1}(y')=y$.
Indeed, observe first that for every $m$, we have the following property of $t_m$:  for every $w<_{E_m} z$ there are $w'<_{E_{m+1}} z'$ and $z''<_{E_{m+1}} w''$
with $t_m(w'')=t_m(w')=w$ and $t_m(z'')=t_m(z')=z$. This follows directly from the definition of $t_m$. Therefore for every $m<n$ and $w<_{E_{m}} z$ there are $z''<_{E_{n}} w''$ with $t^{n}_{m}(w'')=w$ and $t^{n}_{m}(z'')=z$.
Now let $x<_{E_{l_{n-1}}}y$ and take $y''<_{E_{l_{n}}} x''$ with $t^{l_{n}}_{l_{n-1}}(x'')=x$ and $t^{l_{n}}_{l_{n-1}}(y'')=y$.
Since $g_n$ is confluent, we have $g_n(y'')<_{S_{k_n}} g_n(x'')$. As $h_{n-1}\circ g_n=t^{l_{n}}_{l_{n-1}}$, we obtain that
 $x'=g_n(x'')$ and $y'=g_n(y'')$ are as required.

Next, take any $n$ and $a<_{S_{k_n}}b$ and let $a'<_{E_{l_n}} b'$ be such that $g_n(a')=a$ and $g_n(b')=b$. Let now $b''<_{S_{k_{n+1}}}a''$ be such that $h_n(a'')=a'$ and $h_n(b'')=b'$. However, then we must have $w^{k_{n+1}}_{k_n}(a'')=a$ and $w^{k_{n+1}}_{k_n}(b'')=b$, as $g_n\circ h_n=w^{k_{n+1}}_{k_n}$. This gives a contradiction since $w^{k_{n+1}}_{k_n}$ is confluent.

\end{example}

The example above shows that an almost graph-solenoid may be not isomorphic to a graph-solenoid. To justify the necessity of studying almost graph-solenoids, we still have to show that topological realizations of almost graph-solenoids are in fact solenoids.
The following theorem generalizes Theorem \ref{graphtosol}.

\begin{theorem}\label{graphtosol2}
Let $\{C_n, f_n\}$ be an inverse sequence of cycles with almost wrapping maps such that for every $n$ and $x\in C_n$, every component of $f_n^{-1}(x)$ has at least two vertices, and let  $\mathbb{C}=\iLim\{ C_n, f_n\}$. Then the topological realization of $\mathbb{C}$ exists and is a solenoid or a circle.

\end{theorem}

\begin{proof}


By a result of Krupski \cite[Theorem 4.3]{Kr}, it suffices 
to show that $|\mathbb{C}|$ is circularly chainable,  it has Kelley property, and it has no local endpoints.
Note, a point $x$ is not a local endpoint of $|\mathbb C|$ if for every open set  $U\subseteq |\mathbb C|$ and $x\in U$ there are continua $X\subseteq U$ and $Y\subseteq U$ such that $x\in X\cap Y$ and neither $X\subseteq Y$ nor $Y\subseteq X$. Let 
$\pi\colon \mathbb{C}\to |\mathbb{C}|$ be the  quotient map.

\medskip

1) {\bf circularly chainable:} 
Take any open cover $\mathcal{U}$ of $|\mathbb{C}|$ and consider $\mathcal{V}=\{\pi^{-1}(U)\colon U\in\mathcal{U}\}$.
We can find $n$ such that the cover $\{(f_n^\infty)^{-1}(c)\colon c\in C_n\}$ refines $\mathcal{V}$. Consider the cover of $|\mathbb{C}|$ by closed sets: $\mathcal K=\{K_c=\pi\circ (f_n^\infty)^{-1}(c)\colon c\in C_n\}$. Note that $K_c \cap K_d \not = \emptyset$ if and only if $\langle c,d \rangle \in E(C_n)$. Fix a compatible metric on $|\mathbb{C}|$. Since $\mathcal K$ is finite and each set $K_c$ is compact, there is $\varepsilon >0$ such that the distance between any two non-intersecting sets in $\mathcal K$ is greater than $\varepsilon$. Let $U_c$  be the $\frac{\varepsilon}{3}$-neighbourhood of $K_c$ intersected with all open sets from $\mathcal{U}$ that contain $K_c$. Then $U_c \cap U_d \not = \emptyset$ if and only if   $K_c\cap K_d\neq \emptyset$ and thus $\{U_c\colon c\in C_n\}$ is an open circular chain  cover of $|\mathbb C|$ refining~$\mathcal{U}$.

\bigskip

To show Kelley property and that there are no local endpoints, we need two observations.
The first observation is an immediate consequence of the definition of an almost wrapping map. Observation 1 will be the only place where we will use that our homomorphisms are almost wrapping maps. 

{\bf Observation 1.}
Suppose that $f\colon D\to C$ is an almost wrapping map, $C$ and $D$ are cycles, and $A\subseteq C$ is an arc.
Let $x\in D$ be such that $f(x)\in A$ and $f(x)$ is not an end-vertex of $A$. Then there is an arc $B\subseteq D$ such that $x$ is not an end-vertex of $B$ and $f(B)=A$.

\bigskip
We will use in Observation 2 that for every $n$ and $x\in C_n$, every component of $f_n^{-1}(x)$ has at least two vertices.

{\bf Observation 2.}
 Suppose that $(a_n)\in \mathbb{C}$ is a single vertex, that is, it is edge related only to itself. Then for any $m$ and $l$ there is $n>m$  such that both arcs in $(f^n_m)^{-1}(a_m)\setminus\{a_n\}$, adjacent to $a_n$, have at least $l$ vertices.
\begin{proof}[Proof of Observation 2]
Note that it suffices to show the statement for $l=1$, and then apply it $l$ times to get the required conclusion. 

Suppose towards the contradiction that for every $n>m$ there is~$x_n$ adjacent to $a_n$ which is mapped by $f^n_m$ to a vertex adjacent to $a_m$. Since for every $n> m$ the component of $(f_{n-1})^{-1}(a_{n-1})$ containing $a_{n}$ has at least two vertices, 
the other vertex adjacent to $a_n$, which we denote by $a'_n$, satisfies $f_{n-1}(a'_n)=a_{n-1}$. Since $f_{n-1}$ is a homomorphism, $f_{n-1}(x_n)\in\{a'_{n-1}, a_{n-1},x_{n-1}\}$.
However, $f^{n-1}_m(a'_{n-1})=f^{n-1}_m(a_{n-1})=a_m$, and hence
$f_{n-1}(x_n)=x_{n-1}$. Thus $(x_n)\in\mathbb{C}$ and $\langle(a_n), (x_n)\rangle \in E(\mathbb{C})$, which contradicts  that $(a_n)$ is a single vertex.

 \end{proof}

2) {\bf Kelley property:} 
Let $X$ be a proper subcontinuum of 
$|\mathbb{C}|$,  $x\in X$,   ${\bf K}=\pi^{-1}(X)$, and let ${\bf a}=(a_n)\in \mathbb{C}$ be such that $\pi({\bf a})=x$.
Therefore ${\bf K}$ is closed, connected, and additionally: if ${\bf p}\in {\bf K}$, ${\bf q}\in \mathbb{C}$, and $\langle {\bf p},{\bf q}\rangle \in E(\mathbb{C})$, then ${\bf q}\in {\bf K}$.

If we show that $\mathbb C$ is Kelley then, by Proposition~\ref{Kelley}, $|\mathbb C|$ is Kelley. Thus we need  to show that for every sequence ${\bf p}_k\to {\bf a}$ in $\mathbb C$ there are closed connected subsets ${\bf P}_k$ of $\mathbb{C}$ with ${\bf p}_k\in {\bf P}_k$ such that $\lim {\bf P}_k={\bf K}$. To show this, using the definition of the product topology, we have to show that for every $M$  there is an $N$ such that if  ${\bf c}=(c_n)$ is such that $c_N=a_N$, then there is a closed connected set ${\bf L}$ with ${\bf c}\in {\bf L}$ and $f^\infty_M({\bf L})=f^\infty_M({\bf K})$.

Write ${\bf K}=\iLim\{ K_n, f_n\}$, where $K_n=f^\infty_n({\bf K})$ is an arc. Therefore we have $a_n\in K_n$ for each $n$.
We fix $M$.

\underline{Case 1.} The ${\bf a}=(a_n)$ is the only vertex such that $\pi({\bf a})=x$.
Using   Observation 2, take $N> M$ such that if $s\neq t$ are vertices in $C_N$ adjacent to $a_N$ then $f^N_M(s)=f^N_M(t)=a_M$. Let ${\bf c}=(c_n)$ agree with  ${\bf a}$ on the first $N$ coordinates, and let $L_N=K_N\cup\{s,t\}$ (note: at least one of the $s$ or $t$ is already in $K_N$). Then $a_N$ is not an end-vertex of $L_N$.
Use Observation 1 to construct  arcs $L_n$ for $N<n$ such that $f_n(L_{n+1})=L_n$, $c_n \in L_n$,  $c_n$ is not an end-vertex of $L_n$.
For $n<N$ let $L_n = f^N_n(L_N)$, and let ${\bf L}=\iLim\{ L_n, f_n\}$. Therefore, in particular, we have 
$K_M=L_M$ and ${\bf c}\in {\bf L}$, as we wanted.

\underline{Case 2.} There is ${\bf b}=(b_n)\neq {\bf a}$ such that $\pi({\bf b})=x$. Take $N> M$ such that $a_{N-1}\neq b_{N-1}$.
Since ${\bf a}\in {\bf K}$, we also have ${\bf b}\in {\bf K}$, hence $b_N\in K_N$. Let $s\neq b_N$ be adjacent to $a_N$. 
Since $f_{N-1}(b_N)=b_{N-1}\neq a_{N-1}$ and each component of $(f_{N-1})^{-1}(a_{N-1})$ has at least two vertices, we have $f_{N-1}(s)=a_{N-1}$, and hence $f^N_M(s)=a_M$. Let $L_N=K_N\cup\{s\}$. Then $a_N$ is not an end-vertex of $L_N$. Let ${\bf c}=(c_n)$ agree with  ${\bf a}$ on the first $N$ coordinates. Now use Observation~1 and proceed exactly as in Case~1 to construct a sequence of arcs $(L_n)$ and ${\bf L}=\iLim\{ L_n, f_n\}$ such that 
$K_M=L_M$ and ${\bf c}\in {\bf L}$.

\medskip

3) {\bf no local endpoints:} 
Let $U$ be an open subset of $|\mathbb C|$ and $x\in U$. We have to find continua $X\subseteq U$ and $Y\subseteq U$ such that $x\in X\cap Y$ and neither $X\subseteq Y$ nor $Y\subseteq X$.

\underline{Case 1.} There is exactly one ${\bf a}=(a_n)\in \mathbb{C}$ such that $\pi({\bf a})=x$.
Let $m$ be such that $(f_m^{\infty})^{-1}(a_m) \subseteq \pi^{-1}(U)$.
Apply Observation 2 to $m$ and $l=3$ and get $n$.
In $C_n$ consider $p_1<_{C_n}p_2<_{C_n}p_3<_{C_n} a_n<_{C_n}q_1<_{C_n}q_2<_{C_n}q_3$ and take arcs $A_n=\{p_1,p_2,p_3, a_n, q_1\}$ and $B_n=\{p_3,a_n, q_1, q_2\, q_3\}$.  Next apply Observation~1 repeatedly to get for every $i>n$ an arc $A_i$ in $C_i$ such that $a_i\in A_i$, $a_i$ is not an end-vertex of $A_i$, and $f_i(A_{i+1})=A_i$; and let ${\bf A}=\iLim \{A_i, f_i\}$. Analogously construct ${\bf B}$.  Then ${\bf A}$ and ${\bf B}$ are compact connected topological graphs in $(f_m^{\infty})^{-1}(a_m)$ and ${\bf a} \in {\bf A} \cap {\bf B}$. If ${\bf y}\in {\bf A}$ is such that $f^\infty_n({\bf y})=p_1$ and ${\bf z}\in {\bf B}$ is such that $f^\infty_n({\bf z})=q_3$, then since $p_1 $ is not adjacent to $B_n$ and $q_3$ is not adjacent to $A_n$, we have $\pi({\bf y})\in\pi({\bf A})\setminus \pi({\bf B})$ and $\pi({\bf z})\in\pi({\bf B})\setminus \pi({\bf A})$, showing that  $\pi({\bf A}) \not \subseteq \pi({\bf B})$ and $\pi({\bf B}) \not \subseteq \pi({\bf A})$. Thus $X=\pi({\bf A})$ and $Y=\pi({\bf B})$ are the required continua.

\underline{Case 2.} There are ${\bf a}=(a_n)\neq {\bf b}=(b_n)$ such that $\pi({\bf a})=\pi({\bf b})=x$.
Let $n$ and $u_1<_{C_{n}} u_2<_{C_{n}} a_{n} <_{C_{n}} b_{n}  <_{C_{n}}  v_1<_{C_{n}}  v_2$ be such that 
$(f^{\infty}_n)^{-1}(\{u_1,u_2, a_n,b_n,v_1,v_2\})\subseteq \pi^{-1}(U)$.   Take arcs $A_{n}=\{u_1, u_2, a_{n}, b_{n}\}$ and $B_{n}=\{ a_{n}, b_{n}, v_1,v_2\}$. Since $a_{n}$ is not an end-vertex of $A_{n}$ we may apply Observation~1 repeatedly to get for every $i>n$ an arc $A_i$ in $C_i$ such that $a_i\in A_i$ and $f_i(A_{i+1})=A_i$.  Let ${\bf A}=\iLim \{A_i, f_i\}$. Then ${\bf A}$ is compact connected and ${\bf a}\in A$. Analogously, using that  $b_{n}$ is not an end-vertex of $B_{n}$, construct a compact connected set ${\bf B}$ with ${\bf b}\in {\bf B}$.
   Finally, picking ${\bf y}\in {\bf A}$ and ${\bf z}\in {\bf B}$ such that 
  $f^\infty_n({\bf y})=u_1$ and  $f^\infty_n({\bf z})=v_2$, as in Case 1, we conclude that $X=\pi({\bf A})$ and $Y=\pi({\bf B})$ are as required.

\end{proof}

The following lemma is an immediate consequence of the definition of an almost wrapping map.

\begin{lemma}\label{almwrapinterval}
Let $C,D$ be cycles, $n=|C|$, $C=(c_0, c_1, c_2\ldots, c_n=c_0)$, and let $w\colon D\to C$ be an epimorphism.
Then $w$ is an almost wrapping map with the winding number equal to $k$ if and only if we can partition $D$ into arcs $A_i=[a_i, b_i]$, $0\leq i\leq kn$,  with $b_i$ adjacent to $a_{i+1}$, and $b_{kn}$ adjacent to $a_1$, such that 
$w(a_i)=w(b_i)=c_{(i \text{ mod } n)}$ and 
$w(A_i)=\{c_{(i \text{ mod } n)}, c_{(i+1 \text{ mod } n)}\}$.
\end{lemma}

Let  $w\colon D\to C$ be an almost wrapping map and let $X\subseteq C$ be an arc.
We say that a component $T$ of $w^{-1}(X)$ is {\it proper} iff $w(T)=X$.
Let now $b\in C$ and let $M$ and $M'$ be proper components of $C\setminus \{b\}$. Let $a,c\in 
C$ be such that $a<_C b <_C c$. We say that $M'$ is {\it proper component }adjacent to $M$ if 
 there are $$m<_D x_1<_D \ldots <_D x_{k_1}<_D y<_D z_1<_D \ldots <_D z_{k_2} <_D m'$$ with $m\in M$, $m'\in M'$, 
$w(x_1)=w(y)=w(z_{k_2})=b$, $w(m)=a$, $w(m')=c$, 
$w(x_i)\in\{a,b\}$, and $w(z_i)\in\{b,c\}$. 
The following observation follows immediately from Lemma \ref{almwrapinterval}.
\begin{observation}\label{adjac}
Suppose that $C,D$ are cycles and $w\colon D\to C$ is an almost wrapping map with the winding number equal to $K$, and let $b\in C$. Then there are $M_1, M_2,\ldots, M_K$, proper components of $w^{-1}(C\setminus\{b\})$, such that $M_{i+1}$ is proper component adjacent to $M_i$ for $i=1,\ldots, K-1$, and  $M_1$ is proper component adjacent to $M_K$. 

\end{observation}

\begin{remark}\label{zigzak}
Let $w\colon D\to C$ be an almost wrapping map.
It is possible that for $x,y\in D$ we have $x<_D y$ and $w(y)<_C w(x)$. 
However, there are no $x,y_1,\ldots, y_n,z$,  $n\geq 1$, with $x<_D y_1<_D\ldots <_D y_n<_D z$ 
and  $w(z)<_C w(y_1)=\ldots =w(y_n)<_C w(x)$.

\end{remark}

Generalizing Remark \ref{zigzak} we have the following description of almost wrapping maps. It characterizes almost wrapping maps as those surjective homomorphisms, which do not reverse the order of triples.
\begin{lemma}\label{notswaptriples}
Let $w\colon D\to C$ be a surjective homomorphism, where $C,D$ are cycles.
Then $w$ is an almost wrapping map if and only if there is an orientation of $C$ and of $D$ such that there are no non-adjacent $x\neq z\in C$  and no $a, c\in D$  such that  $w(c)=x$, $w(a)=z$, 
and $w([a,c])=[x,z]$.

\end{lemma}

\begin{proof}

Suppose first that $w\colon D\to C$ is an almost wrapping map. Write $C=(c_0, c_1, \ldots, c_n=c_0)$ and let $f\colon D\to C$ be a confluent epimorphism such that for every $y\in D$ and $c_l\in C$,
if $f(y)=c_l$ then $w(y)\in \{c_l, c_{l+1}\}$, i.e. if $w(y)=c_l$, then $f(y)\in \{c_{l-1}, c_l\}$. 
 Suppose there are  non-adjacent $x=c_i\neq z=c_j\in C$  and $a, c\in D$  such that  $w(c)=x$, $w(a)=z$, 
and $w([a,c])=[x,z]$. Then we have $f(c)\in\{c_{i-1}, c_i\}$,
$f(a)\in\{c_{j-1}, c_j\}$,
and $f([a,c])$ is one of  $[c_{i-1}, c_{j-1}]$, $[c_{i-1}, c_j]$,   $[c_{i}, c_{j-1}]$, $[c_{i}, c_j]$. As $f$ is order-preserving, $f(a)\in\{c_{i-1}, c_i\}$
and $f(c)\in\{c_{j-1}, c_j\}$. However, this is impossible since $c_i$ and $c_j$ are non-adjacent.

To prove the converse, suppose that there is an orientation of $C$ and of $D$  so that there are no non-adjacent $x\neq z\in C$  and no $a, c\in D$  such that  $w(c)=x$, $w(a)=z$, and $w([a,c])=[x,z]$,
and write $C=(c_1, \ldots, c_n=c_1)$.
In particular, there are no $x,y,z\in C$ with $x<_C y$ and $y<_C z$, and 
$a=b_0, b_1, \ldots, c=b_m$ with $b_j<_D b_{j+1}$ such that  $w(c)=x$, $w(b_j)=y$ for $j=1,\ldots, m-1$, and $w(a)=z$. This implies that we can partition $D$ (in a possibly non-unique way) into arcs $A_i=[a_i, b_i]$, $1\leq i\leq kn$, for some $k$, with $b_i$ adjacent to $a_{i+1}$, and $b_{kn}$ adjacent to $a_1$, such that 
$w(A_i)\in \{ c_{(i \text{ mod } n)}, c_{(i+1 \text{ mod } n)}  \}   $.
Therefore $f$ such that $f(A_i)=\{c_i\}$ is a confluent witness for $w$.

\end{proof}

We show in the next lemma that almost wrapping maps are closed under composition.
\begin{lemma}\label{compawrap}
Let $C, D$, and $E$ be cycles and let $g_1\colon D\to C$ and $g_2\colon E\to D$ be almost wrapping maps that have proper confluent witnesses.
Then $g=g_1\circ g_2$ is an almost wrapping map, which has a proper confluent witness.
\end{lemma}
\begin{proof}
Write $D=(d_0,\ldots, d_n=d_0)$ and $C=(c_0,\ldots, c_m=c_0)$.
Let $f_1\colon D\to C$ and $f_2\colon E\to D$ be proper confluent witnesses for $g_1$ and $g_2$, and consider the confluent epimorphism $f_0=f_1\circ f_2$. Note first that for $e\in E$ if $f_2(e)=d_i$ and $f_1(d_i)=c_j$, then $g_2(e)\in \{d_i, d_{i+1}\}$, $g_1(d_i)\in \{c_j, c_{j+1}\}$ and  $g_1(d_{i+1})\in \{c_{j+1}, c_{j+2}\}$,  hence $g(e)\in\{c_j, c_{j+1}, c_{j+2}\}$.  Therefore we have to modify $f_0$ to get a proper confluent witness for $g$.

For each $j$ and a component $Y=[y,z]$ of $f_0^{-1}(c_j)$,
let first $X=[w,x]$ denote the component of  $f_0^{-1}(c_{j-1})$ adjacent to $Y$.   If there is $s\in X$ such that $g(s)=c_{j+1}$, let $s_X$ be the first such vertex, that is, $g(s_X)= c_{j+1}$ and there is no $s' \in [w,s_X)$ such that $g(s') = c_{j+1}$. If there is no $s\in X$ such that $g(s)=c_{j+1}$, let $s_X=y$. Analogously, if there is $s \in Y$ such that $g(s)=c_{j+2}$ let $s_Y$ be the first such vertex in $Y$, otherwise let $s_Y$ be the vertex not in $Y$ that is adjacent to $z$.

Define a map $f$ from $E$ to $C$ by declaring that
$f$ assumes value $c_j$ on $[s_X,s_Y)$.   To see that $f$ is a confluent epimorphism, it suffices to see that $s_Y\neq y$, since then $[s_X,s_Y)\neq\emptyset$. We show that in fact $s_Y\neq y, y^+$, where  $y^+$ is such that $y<_E y^+$, which will imply that for every $a \in C$ every component of $f^{-1}(a)$ contains at least $2$ vertices.
We have that $Y=[y,z]$ is a component of $f_0^{-1}(c_j)$.  Let $Y_1=[y_1,z_1]$ be the component of $f_1^{-1}(c_j)$ such that $f_2(Y)=Y_1$ and   let $y_1^+$ be such that $y_1<_Dy_1^+$. Then $f_2(y)=y_1$ and so $g_2(y)\in \{y_1, y_1^+\}$. Since $f_2$ is a proper confluent witness for $g_2$, we have $f_2(y^+)=f_2(y)=y_1$, and therefore $g_2(y^+)\in \{y_1, y_1^+\}$. Since $f_1$ is a proper confluent witness for $g_1$, we have $f_1(y_1^+)=f_1(y_1)=c_j$, and therefore $g_1(y_1), g_1(y_1^+)\in \{c_j, c_{j+1}\}$. This implies that $g(y), g(y^+)\in \{c_j, c_{j+1}\}$, and hence  $s_Y\neq y, y^+$.

Finally, we  check that for any $t\in [s_X, s_Y)$,  
$g(t) \in \{c_j,c_{j+1}\}$.
We have that  $f_0$ assumes value $c_j$ on $Y$ and $c_{j-1}$ on $X$. Therefore $g$ can only assume values from $\{c_j, c_{j+1}, c_{j+2}\}$ on $Y$ and values from $\{c_{j-1}, c_{j}, c_{j+1}\}$ on $X$. Since $[s_X,s_Y)\subseteq X\cup Y$, and
by the choice of $s_Y$ there is no $t\in [y, s_Y)$ such that $g(t)=c_{j+2}$,
we only have to show that there is no $t\in [s_X, y)$ such that $g(t)=c_{j-1}$. If $s_X=y$ there is nothing to show.
In the case $s_X\neq y$, suppose towards the contradiction that there is $t\in [s_X, x]$ such that $g(t)=c_{j-1}$. 
\blu{Since $g_2$ is an almost wrapping map, $g_2([s_X,t])$ is an arc containing $[g_2(s_X),g_2(t)]$, where one of the end-vertices is equal or adjacent to $g_2(s_X)$ and another one is equal or adjacent to $g_2(t)$. Thus $[g_2(s_X), g_2(t)]\subseteq g_2([s_X,t])$ implies
 $g_1([g_2(s_X), g_2(t)])\subseteq g([s_X,t])\subseteq [c_{j-1}, c_{j+1}]$. Since $c_{j-1}, c_{j+1}\in g_1([g_2(s_X), g_2(t)])$ and $g_1$ is a homomorphism, we have in fact $g_1([g_2(s_X),g_2(t)])=[c_{j-1}, c_{j+1}]$.
The equalities $g_1(g_2(s_X))=c_{j+1}$, $g_1(g_2(t))=c_{j-1}$, and $g_1([g_2(s_X),g_2(t)])=[c_{j-1}, c_{j+1}]$, contradict Lemma \ref{notswaptriples} applied to the almost wrapping map $g_1$ and to $c_{j-1}, c_{j+1}$.
}

\end{proof}

\subsection{Almost graph-solenoids as preimages of solenoids}

In this subsection we show that each solenoid in $|\mathbb G|$ comes from an almost graph-solenoid in $\mathbb G$  via the quotient map. 
\begin{theorem}\label{soltograph-main}
Let $S$ be a solenoid  embedded in $|\mathbb{G}|$. 
Then there is an almost graph-solenoid $\mathbb{S}\subseteq\mathbb{G}$ such that $\pi(\mathbb{S})=S$, where $\pi\colon \mathbb{G}\to |\mathbb{G}|$ is the quotient map.
\end{theorem}

Let $\mathbb{F}$ be a topological graph such that for every ${\bf x}\in \mathbb{F}$ there is at most one ${\bf y}\neq {\bf x}$, ${\bf y}\in \mathbb{F}$ with $\langle {\bf x}, {\bf y} \rangle \in E(\mathbb{F})$,  and let $\delta\colon \mathbb{F}\to |\mathbb{F}|$ be the quotient map. Let 
$$O_{\mathbb{F}}=\{{\bf x}\in\mathbb{F}\colon \text{ there is  no } {\bf y}\neq {\bf x}  \text{ with } \langle {\bf x}, {\bf y} \rangle \in E(\mathbb{F})\}$$
be the set of {\it singletons} of $\mathbb{F}$ and let 
 $$T_{\mathbb{F}}=\{{\bf x}\in\mathbb{F}\colon \text{ there is exactly one } {\bf y}\neq {\bf x}  \text{ with }  \langle {\bf x}, {\bf y} \rangle \in E(\mathbb{F})\}$$ 
 be the set of {\it doubletons}  of $\mathbb{F}$.  We have $\mathbb{F}=O_{\mathbb{F}}\cup T_{\mathbb{F}}$.

The set  $O_{\mathbb{F}}$ is $G_\delta$, i.e. it is the intersection of countably many open sets. Indeed, it is the intersection of
$$A_n=\{{\bf x}\in\mathbb{F}\colon d({\bf x},{\bf y})<\frac{1}{n} \text{ for every } {\bf y} \text{ such that } \langle {\bf x}, {\bf y} \rangle \in E(\mathbb{F})\},$$ where $d$ is some fixed metric on $\mathbb{F}$. Each $\mathbb{F}\setminus A_n$ is closed since the edge relation $E(\mathbb{F})$ is closed in $\mathbb{F}^2$ and  $\mathbb{F}$ is compact.

Note that if $\mathbb{F}_1\subseteq \mathbb{F}_2$ are topological graphs  
then $\mathbb{F}_1\cap O_{\mathbb{F}_2}\subseteq O_{\mathbb{F}_1}$, but the equality may not hold.

\begin{lemma}\label{singletons-dense}

The sets $O_{\mathbb{G}}$ and $T_{\mathbb{G}}$ are dense in $\mathbb{G}$. Therefore $O_{\mathbb{G}}$ is a dense $G_\delta$.
\end{lemma}

\begin{proof}
First note that if $G$ is a connected graph then there is a connected graph $H$ and a confluent (even monotone) $f^H_G\colon H\to G$ such that for each $a\in G$ there is $b\in (f^H_G)^{-1}(a)$ such that for any $c\in H$ with $\langle b,c \rangle\in E(H)$, we have $f^H_G(c)=a$. To get $H$ simply  split each edge $\langle x,y \rangle\in E(G)$ into three: $\langle x,x_1 \rangle$, $\langle x_1, y_1 \rangle$, $\langle y_1,y \rangle$, and let $f^H_G(x)=f^H_G(x_1)=x$ and $f^H_G(y)=f^H_G(y_1)=y$. 

    Let ${\bf U}$ be an open set in $\mathbb G$, $\{F_n, \alpha_n\}$ be  a Fra\"\i ss\'e sequence for the family $\mathcal{G}$, and pick $n_0$ and $a_0\in F_{n_0}$ such that $(\alpha^{\infty}_{n_0})^{-1}(a_0)\subseteq {\bf U}$.
    Using the extension property and the observation above, get $n_1$ and  $a_1\in F_{n_1}$ with  $\alpha^{n_1}_{n_0}(a_1)=a_0$ such that for any $b\in F_{n_1}$ with $\langle a_1, b\rangle$ we have $\alpha^{n_1}_{n_0}(b)=a_0$. Continue this and get a sequence $n_0<n_1<n_2\ldots$ and $a_i\in F_{n_i}$ such that 
    $\alpha^{n_{i+1}}_{n_i}(a_{i+1})=a_i$ and
    for any $b\in F_{n_{i+1}}$ with $\langle a_{i+1}, b\rangle$ we have $\alpha^{n_{i+1}}_{n_i}(b)=a_i$. Then the sequence $(a_i)$ induces a vertex in $\mathbb{G}$, which is a singleton in ${\bf U}$. This finishes the proof of density of $O_{\mathbb{G}}$.

Since all graphs in $\mathcal{G}$ are connected, by Condition (3) of Theorem \ref{limit}, it follows that for any nonempty proper clopen ${\bf U}$ in $\mathbb{G}$, there is an edge joining ${\bf U} $ and its complement. In particular, $T_{\mathbb{G}}$ is dense. 

\end{proof}

\begin{lemma}\label{densesingle}
Let ${\bf F}_0$ be a topological graph  such that ${\bf F}_0=O_{{\bf F}_0}\cup T_{{\bf F}_0}$  and let $\delta\colon {\bf F}_0\to |{\bf F}_0|$ be the quotient map.  Then there is a topological graph ${\bf F}\subseteq {\bf F}_0$ such that  $O_{\bf F}$ is dense in ${\bf F}$ and $\delta({\bf F})=|{\bf F}_0|$.
\end{lemma}

 \begin{proof} 
Let ${\bf U}_1, {\bf U}_2, {\bf U}_3, \ldots$ be a countable open base for ${\bf F}_0$ with ${\bf U}_i \not = \emptyset$. We will proceed by the following recursive procedure. Suppose at step $n-1$ we have constructed a closed subset ${\bf F}_{n-1}\subseteq{\bf F}_0$.
 If for every ${\bf x}\in  {\bf F}_{n-1}\cap {\bf U}_n$ there is ${\bf y}\in {\bf F}_{n-1}\setminus {\bf U}_n$ such that  $\langle {\bf x}, {\bf y} \rangle \in E({\bf F}_{n-1})$,  then take ${\bf F}_n={\bf F}_{n-1}\setminus {\bf U}_n$. Otherwise, let ${\bf F}_n={\bf F}_{n-1}$.
Clearly $\delta({\bf F}_n)= |{\bf F}_0|$. 
 
We claim that the set ${\bf F}=\bigcap_n {\bf F}_n$ is as required. Clearly  $\delta({\bf F})= |{\bf F}_0|$. 
Suppose towards a contradiction that $O_{{\bf F}}$ is not dense in ${\bf F}$. 
Let ${\bf U}$ be an open nonempty set in ${\bf F}$ such that
 for every ${\bf x}\in {\bf U}$ there is
${\bf y}\in {\bf F}$, ${\bf y}\neq {\bf x}$, so that $\langle {\bf x}, {\bf y} \rangle \in E({\bf F})$. 
 Fix a metric $d$ on ${\bf F}$ and
 let $$ {\bf D}_n=\{{\bf x}\in {\bf U}\colon \text{ there is  }  {\bf y}\in {\bf F} \text{ with } 
 \langle {\bf x}, {\bf y} \rangle \in E({\bf F}) \text{ and } d({\bf x},{\bf y})\geq \frac{1}{n}\}.$$
 Each ${\bf D}_n$ is closed in ${\bf U}$ and  we have $\bigcup_n {\bf D}_n={\bf U}$. 
By the Baire category theorem for the open set ${\bf U}$ there is $n_0$ such that ${\bf D}_{n_0}$ has a nonempty interior. 
Take ${\bf V}\subseteq {\rm int}({\bf D}_{n_0})$ open nonempty of diameter $<\frac{1}{n_0}$. 
Therefore, for every ${\bf x}\in {\bf V}$, there is ${\bf y}\in {\bf F}\setminus {\bf V}$ with  $\langle {\bf x}, {\bf y} \rangle \in E({\bf F})$. 
Let $n$ be such that  $\emptyset \not = {\bf U}_n\cap {\bf F}\subseteq {\bf V}$, and without loss of generality by shrinking ${\bf V}$, we have ${\bf U}_n\cap {\bf F}= {\bf V}$. 
 This means that ${\bf U}_n$ was subtracted in the construction of ${\bf F}$, which gives a contradiction. Indeed, if ${\bf U}_n $ was not subtracted at stage $n$, then one of the following would have to happen. The first possibility is that there is ${\bf x}\in {\bf U}_n\cap O_{{\bf F}_{n-1}}$.  However, then ${\bf x}\in {\bf F}\setminus {\bf V}$,
 which is impossible. The second possibility is that
there are ${\bf x}\neq {\bf y} \in {\bf U}_n$ with  $\langle {\bf x}, {\bf y} \rangle \in E({\bf F}_{n-1})$. However, then ${\bf x},{\bf y}\notin {\bf V}$, hence also ${\bf x},{\bf y}\notin {\bf F}$, which cannot be the case as $\delta({\bf F})= |{\bf F}_0|$. 
 \end{proof}

\begin{example}
Let ${\bf F}$ be a  topological graph such that ${\bf F}=O_{{\bf F}}\cup T_{{\bf F}}$  and let $|{\bf F}|$ be its topological realization.  Suppose that $O_{{\bf F}}$ is uncountable and dense in ${\bf F}$.
Then there is ${\bf F}_1 \supseteq {\bf F}$ with the same topological realization, ${\bf F}_1=O_{{\bf F}_1}\cup T_{{\bf F}_1}$, but  $O_{{\bf F}_1}$ is not dense in ${\bf F}_1$.

Indeed,  $O_{{\bf F}}$ is an uncountable dense $G_\delta$ in ${\bf F}$ by Lemma \ref{singletons-dense}. Therefore it contains a copy $C$ of the Cantor set. Let $D$ be another copy of the Cantor set and let $h\colon C\to D$ be a homeomorphism. We take a topological graph ${\bf F}_1$ obtained as follows.
Take $V({\bf F}_1)= V({\bf F})\cup V(D)$ and let   $\langle {\bf x},{\bf y} \rangle \in E({\bf F}_1)$ iff ${\bf x}={\bf y}$ or (${\bf x},{\bf y}\in {\bf F}$ and $\langle {\bf x},{\bf y} \rangle \in E({\bf F})$)
or (${\bf x}\in D$, ${\bf y}\in C$, and ${\bf x}=h({\bf y})$) or 
(${\bf y}\in D$, ${\bf x}\in C$, and ${\bf y}=h({\bf x})$). 
 We declare $D$ to be a clopen in ${\bf F}_1$.
Then ${\bf F}_1$ is as required as $D\subseteq T_{{\bf F}_1}$.
\end{example}

Let $S$ be a solenoid embedded in $|\mathbb G|$ and write it as the inverse limit  of $\{X_n, w_n\}$, where $X_n$ is the unit circle $S^1$ and  $w_n\colon X_{n+1}\to X_n$ is given by $w_n(e^{it})=e^{i l_n t}$, for some $l_n\in\mathbb{N}\setminus\{ 0,1\}$. The canonical projection from $S$ to the $n$th factor space $X_n$ we denote by $w^\infty_n$. Further, by Lemma \ref{densesingle}, we can assume that $S$ is a topological realization of a topological graph $\mathbb{S}\subseteq \mathbb{G}$ such that $O_{\mathbb{S}}$  is dense in $\mathbb{S}$.
 Recall that $\pi\colon \mathbb{G}\to |\mathbb G|$ is the quotient map and 
 denote  $\pi_S=\pi|_{\mathbb{S}}$.

For ${\bf A}\subseteq \mathbb{S}$ we let $k({\bf A})$ to be the closure in $\mathbb{S}$ of the set of the singletons of ${\bf A}$, i.e. $k({\bf A})=\overline{O_{\bf A}}$. 
We will apply this operation to sets of the form $(w^\infty_n\pi_S)^{-1}(P)$, where $P$ is an arc $[a=e^{it_1}, b=e^{it_2}]$ (respectively an open arc $(a,b)$) in $S^1=X_n$. Denote $(w^\infty_n\pi_S)^{-1}(P)$ by 
$[a,b]_n$ (respectively $(a,b)_n$). 

\begin{lemma}\label{properties of [a,b]n}
\begin{enumerate}
    \item If ${\bf A}=[a,b]_n$, then 
$k({\bf A})= \overline {(a,b)_n}=\overline{O_{\mathbb{S}}\cap {\bf A}}$. In that case $k({\bf A})$ is a regular closed set.
\item If $[c,d]\subseteq [a,b]$ and $n\in\mathbb{N}$, then $k([c,d]_n)\subseteq k([a,b]_n)$.
\item Suppose that $[a_1, b_1], [a_2, b_2],\ldots, [a_m, b_m]$ is a cover of $S^1$ by arcs such that $b_i=a_{i+1}$, $i=1,\ldots, m-1$ and $b_m=a_1$ and the interiors of the arcs are disjoint. Then for every $n\in\mathbb{N}$,
$k([a_1, b_1]_n)$, $k([a_2, b_2]_n),\ldots, k([a_m, b_m]_n)$ is a cover of $\mathbb{S}$. 
\item Let $[a,b]$ and $[b,c]$ be arcs in $S^1$ such that $[a,b]\cap [b,c]=\{b\}$, $A=(w^\infty_n)^{-1}([a,b])$ and  $C=(w^\infty_n)^{-1}([b,c])$ be sets in the solenoid $S$, and $y\in A\cap C$. If $\pi_S^{-1}(y)=\{{\bf x}\}$, then ${\bf x}\in k([a,b]_n)\cap k([b,c]_n)$.
If  $\pi_S^{-1}(y)=\{{\bf x}_1,{\bf x}_2\}$, ${\bf x}_1\neq {\bf x}_2$, 
then at least one of ${\bf x}_1,{\bf x}_2$ is in $k([a,b]_n)$ and at least one of ${\bf x}_1,{\bf x}_2$ is in $k([b,c]_n)$. 

In particular, there is an edge (possibly a degenerate one) between $k([a,b]_n)$ and $k([b,c]_n)$.

\end{enumerate}
\end{lemma}
\begin{proof} 
(1) Since $O_{[a,b]_n}\cap (a,b)_n$ is dense in $(a,b)_n$, we get $\overline {(a,b)_n}\subseteq k({\bf A})$. To obtain the reverse inclusion we additionally have to show that any vertex in ${\bf x}\in O_{[a,b]_n}\setminus (a,b)_n$ is a limit of vertices in $(a,b)_n$. For that note that $(w^\infty_n)^{-1}((a,b))$ is dense in $(w^\infty_n)^{-1}([a,b])$.
Therefore $\pi_S({\bf x})$ is a limit of vertices $(z_k)$ in $(w^\infty_{n})^{-1}((a,b))$. Let ${\bf y}_k\in (a,b)_n$ be 
such that $\pi_S({\bf y}_k)=z_k$.  After passing to a subsequence we can assume 
that $({\bf y}_k)$ is convergent to some ${\bf y}\in [a,b]_n$. Since $\pi_S({\bf x})=\pi_S({\bf y})$ 
and ${\bf x}\in O_{[a,b]_n}$, we obtain ${\bf x}={\bf y}$. 

Since $(a,b)_n\subseteq \overline{O_{\mathbb{S}}\cap {\bf A}}\subseteq \overline{O_{\bf A}}$, the first part implies $k({\bf A})=\overline{O_{\mathbb{S}}\cap {\bf A}}$.

(2) \blu{
It follows immediately from (1), since $[c,d]_n\subseteq [a,b]_n$ implies 
$k([c,d]_n)= \overline{O_{\mathbb{S}}\cap  [c,d]_n}\subseteq \overline{O_{\mathbb{S}}\cap  [a,b]_n}= k([a,b]_n)$
}

(3) We know that $[a_1, b_1]_n, [a_2, b_2]_n,\ldots, [a_m, b_m]_n$ is a cover of $\mathbb S$ and $O_{\mathbb S}$ is dense in $\mathbb S$. Therefore
$$\mathbb{S}=\overline{\bigcup_{i=1}^m  O_{[a_i, b_i]_n}}=\bigcup_{i=1}^m \overline{O_{[a_i, b_i]_n}}=
\bigcup_{i=1}^m k([a_i, b_i]_n).$$

(4) The first claim is clear.
For the second claim, let $(a_n)$ be a sequence in
$A_0=(w^\infty_n)^{-1}((a,b))$ converging to $ y$, and let $(c_n)$ be a sequence in $C_0=(w^\infty_n)^{-1}((b,c))$ converging to $y$. 
Let ${\bf a}'_n$ be any vertex in $\pi_S^{-1}(A_0)$ such that $\pi_S({\bf a}'_n)=a_n$, and similarly let ${\bf c}'_n$ be any vertex in $\pi_S^{-1}(C_0)$ such that $\pi_S({\bf c'_n})=c_n$. Without loss of generality, by passing to a subsequence, $({\bf a}'_n)$ and $({\bf c}'_n)$ are convergent.
Since $\pi_S^{-1}(A_0)$ and $\pi_S^{-1}(B_0)$ are open there are $({\bf a}''_n)$ in $\pi_S^{-1}(A_0)\cap O_{\mathbb{S}}$ and  $({\bf c}''_n)$ in $\pi_S^{-1}(C_0)\cap O_{\mathbb{S}}$ converging to the same limit as  $({\bf a}'_n)$ and $({\bf c}'_n)$, respectively.
If ${\bf p}=\lim {\bf a}''_n$ and ${\bf q}=\lim {\bf c}''_n$, then we have ${\bf p}\in [a,b]_n$, ${\bf q}\in [b,c]_n$, and $\pi_S({\bf p})=\pi_S({\bf q})=y$.
\end{proof}

\begin{theorem}\label{soltograph}
Let $S$ be a solenoid and let $\mathbb{S}$ be a topological graph whose topological realization is $S$ via a quotient map $\pi_S\colon \mathbb{S}\to S$. Suppose that $O_{\mathbb{S}}$ is dense in $\mathbb{S}$. Then  $\mathbb{S}$ is an almost graph-solenoid. 

Moreover, if $S=\iLim\{X_n=S^1, w_n\}$, where  $w_n\colon X_{n+1}\to X_n$ is given by $w_n(e^{it})=e^{i l_n t}$, for some $l_n\in\mathbb{N}\setminus\{ 0,1\}$, then  
$\mathbb{S}=\iLim\{C_{k_n}, g_n\}$, where $C_{k_n}$'s are cycles, $g_n$'s are almost wrapping maps,  the winding number of $g_n$ is $l_{k_n}\ldots l_{(k_{n+1})-1}$, and $g_n$ has a proper confluent witness.

\end{theorem}

\begin{proof}
Let $\mathcal{P}_1$ be the cover of $X_1$: 
$$\{ [e^{0}, e^{i2\frac{\pi}{5}}], [e^{i2\frac{\pi}{5}}, e^{i 2\frac{2\pi}{5}}],\ldots,
[e^{i2\frac{4\pi}{5}}, e^{i 2\frac{5\pi}{5}}]\},$$ and for every $n>1$ let  $\mathcal{P}_n$ be the cover of $X_n$: 
$$\{[e^{i \frac{m 2\pi}{L_n}}, e^{i \frac{(m+1)2\pi}{L_n}}]\colon m=0,1,\ldots, L_n-1\},$$
where $L_n=5\cdot 2^{n-1}\cdot l_1\cdot l_2\cdot \ldots \cdot l_{n-1}$.

Let $\mathcal{D}_n$ be the cover of $S$  defined as 
$$\mathcal{D}_n=\{(w_n^{\infty})^{-1}(P)\colon P\in \mathcal{P}_n\}.$$

For every $n$ consider the cover $\mathcal{E}_n=\{k({\bf E}) \colon {\bf E}\in \pi_S^{-1}(\mathcal{D}_n)\}$ of $\mathbb{S}$  by closed sets.  
It follows from Lemma~\ref{properties of [a,b]n}(3) that this is indeed a cover. Let $C_n$ be the graph where $V(C_n)=\mathcal{E}_n$ and $\langle {\bf E},{\bf F} \rangle\in E(C_n)$ 
 iff there is an edge (possibly degenerate) between a vertex in ${\bf E}$ and a vertex in ${\bf F}$. Using Lemma  \ref{properties of [a,b]n}(4), we have that $C_n$ is a cycle. 
 Moreover, whenever $n>m$, the cover $\mathcal{E}_n$ refines $\mathcal{E}_m$ and this refinement induces a map $f^n_m$ from $C_n$ onto $C_m$. Since for any $[c,d]\in\mathcal{P}_n$ and $[a,b]\in \mathcal{P}_m$, we have $f^n_m(k([c,d]_n))=k([a,b]_m)$ iff $w^n_m([c,d])\subseteq [a,b]$, the $f^n_m$
 is a confluent epimorphism whose winding number is equal to $l_{m}\cdot \ldots \cdot l_{n-1}$ and any component in the preimage of a vertex contains at least two vertices.

 Let $d_0$ be a metric on $S^1$. If $x,y \in S$ then $x$ and $y$ may be written as $x=(x_1, x_2,\ldots )$ and $y=(y_1,y_2,\ldots)$.  Let $d_S$ denote the metric on $S$ defined as $d_S(x,y)=\sum_{n=1}^\infty d_0(x_n,y_n)/L_n$, 
 and let $d$  be any metric on $\mathbb{S}$. Then  for every $\varepsilon>0$ there is $N$ such that for every $n>N$ and every ${\bf E}\in \mathcal{E}_n$, ${\rm diam}({\bf E})<\varepsilon$. Indeed, 
 note that $\pi_S |_{O_\mathbb{S}} $ is a homeomorphism onto its image. Therefore there is $M$ such that for any ${\bf u},{\bf v}\in O_\mathbb{S}$  we have $d({\bf u},{\bf v})\leq M d_S(\pi_S({\bf u}), \pi_S({\bf v}))$. By the definition of $d_S$ there is $N$ such that for every $n>N$ and every $D\in \mathcal{D}_n$, ${\rm diam}(D)<\frac{\varepsilon}{M}$. Then 
  ${\rm diam}(O_{\pi_S^{-1}(D)})<\varepsilon$. Since a set and its closure have the same diameter, we get 
  ${\rm diam} (k(\pi_S^{-1}(D)))<\varepsilon$.

Now we enlarge the sets in $\mathcal {E}_n$ to obtain a cover of clopen  sets $\mathcal{U}_n=\{{\bf U_E}\colon {\bf E}\in\mathcal{E}_n\}$.
We  require that ${\bf E}\subseteq {\bf U_E}$ and for every $n$ and ${\bf U_E}, {\bf U_F}\in\mathcal{U}_n$, there is an edge between a vertex in ${\bf U_E}$ and a vertex in ${\bf U_F}$ iff $\langle {\bf E}, {\bf F}\rangle\in E(C_n)$. Moreover, let ${\rm diam}({\bf U_E})<2{\rm diam}({\bf E})$, and let whenever $n>m$, the cover $\mathcal{U}_n$ refine $\mathcal{U}_m$.

Enumerate $\mathcal{U}_n$ into ${\bf U}_1,{\bf U}_2,\ldots, {\bf U}_{L_n}$ so that for every $i$, there is an edge between a vertex in ${\bf U}_i$ and a vertex in ${\bf U}_{i+1}$ (we let ${\bf U}_{L_n+1}={\bf U}_1$). We now let ${\bf V}_i={\bf U}_i\setminus {\bf U}_{i+1}$ and $\mathcal{V}_n=\{{\bf V}_i\colon i=1,\ldots, L_n\}$. Then $\mathcal V_n$ is a clopen partition of $\mathbb S$.

There is an edge between a vertex in ${\bf V}_i$ and a vertex in
${\bf V}_{i+1}$, for each~$i$. Indeed, since $\mathbb{S}$ is connected, the only other possibility would be that there is a single $i_0$ such that there is no edge between a vertex in ${\bf V}_{i_0}$ and a vertex in ${\bf V}_{i_0+1}$. However, this is impossible, since the same would have to be true for the cover $\mathcal{V}_{n+1}$. However, the winding number of the confluent epimorphism obtained from the inclusion of $\mathcal{E}_{n+1}$ into $\mathcal{E}_n$ is at least 2.

Since for every $\varepsilon>0$ there is $N$ such that for every $n>N$ and every ${\bf U}\in \mathcal{U}_n$, ${\rm diam}({\bf U})<\varepsilon$, using the Lebesgue covering lemma, we pass to a subsequence of covers $(\mathcal{V}_{k_n})$ so that for each $n$, $\mathcal{V}_{k_{n+1}}$ refines
$\mathcal{V}_{k_n}$. 
Let $g_n\colon C_{k_{n+1}}\to C_{k_n}$ be a surjective homomorphism obtained from the inclusion of $\mathcal{V}_{k_{n+1}}$ into
$\mathcal{V}_{k_n}$. If ${\bf V}^{n+1}_j={\bf U}^{n+1}_j\setminus {\bf U}^{n+1}_{j+1}\in \mathcal{V}_{k_{n+1}}$ and ${\bf U}_i^n\in \mathcal{U}_{k_n}$ is such that ${\bf U}^{n+1}_j\subseteq {\bf U}_i^n$ then, since ${\bf U}_i^n\cap {\bf U}_{i+2}^n=\emptyset$, we have ${\bf V}^{n+1}_j\subseteq {\bf V}_i^n\in \mathcal{V}_{k_{n}}$ or 
${\bf V}^{n+1}_j\subseteq {\bf V}_{i+1}^n\in \mathcal{V}_{k_{n}}$, where 
${\bf V}_i^n={\bf U}_i^n\setminus {\bf U}_{i+1}^n$ and ${\bf V}_{i+1}^n={\bf U}_{i+1}^n\setminus {\bf U}_{i+2}^n$. Therefore if $\hat{f}_n=f^{k_{n+1}}_{k_n}$, we have that if $\hat{f}_n(y)=c_i$, then $g_n(y)\in \{c_i, c_{i+1}\}$, where $C_{k_n}=(c_0,c_1,\ldots, c_{L_{k_n}}=c_0)$. Thus $g_n$ is an almost wrapping map as required.

\end{proof}

\begin{proof}[Proof of Theorem \ref{soltograph-main}]
Apply  Lemma \ref{densesingle} to $\pi^{-1}(S)$ and then Theorem \ref{soltograph}.
\end{proof}

\subsection{Solenoids that do not embed in $|\mathbb{G}|$}

The goal of this subsection is the following result.
\begin{theorem}\label{notuniversal}
Let a solenoid $S$ embed in $|\mathbb{G}|$. Then $S$ is the universal solenoid.
\end{theorem}

We first need a lemma.
\begin{lemma}\label{phomomorph}
Let $\{H_i, h_i\}_{i \ge 0}$ be an inverse sequence of finite graphs where the $h_i$'s are  homomorphisms and let ${\bf{H}}=\iLim\{H_i, h_i\}$. Let $p\colon H_0\to K$ be a  map to a finite graph $K$. If $p\circ h^\infty_0$ is a homomorphism then there is $N$ such that for every $t\geq N$, the map $p\circ h^t_0$ is a homomorphism. 
\end{lemma}

\begin{proof}
It suffices to show that for every $a,b\in K$ such that $\langle a,b \rangle \notin E(K)$ there is $N$ such that for every $t\geq N$ if $x,y\in H_t$, $p\circ h^t_0(x)=a$, and $p\circ h^t_0(y)=b$, then  $\langle x,y \rangle \notin E(H_t)$.  
Suppose towards the contradiction that this is not the case and let $a,b\in K$ be such that $\langle a,b \rangle \notin E(K)$ and for infinitely many $t$ we have $x,y\in H_t$ with $p\circ h^t_0(x)=a$,  $p\circ h^t_0(y)=b$, and  $\langle x,y \rangle \in E(H_t)$. 

For each $t$ let $L_t$ denote the set of pairs $(x,y)\in H_t\times  H_t$ with $\langle x,y\rangle \in E(H_t)$,
$p\circ h^t_0(x)=a$, and $p\circ h^t_0(y)=b$. Let $L=\bigcup_{t\geq 0} L_t$ be a disjoint union and consider the partial order on $L$ given by: for  $( x,y)\in H_t\times  H_t$ and  $( x',y')\in H_{t'}\times  H_{t'}$ we let 
$( x,y)\precsim ( x',y')$ iff $h^{t'}_t(x')=x$ and $ h^{t'}_t(y')=y$. 
Since each $h_t$  is a homomorphism, for every $t \ge 0$,
if $(x,y)\in L_{t+1}$, then $(h_t(x), h_t(y))\in L_t$.
There are: (1) finitely many minimal elements in  $(L, \precsim)$,  (2) for every $(x,y)\in L_t$ there are finitely many $(x',y')\in L_{t+1}$ with $(x,y)\precsim (x',y')$ and, by our assumption, (3) $L$ is infinite.
 Therefore, by the K\"onig's Lemma (see for example \cite[Section 12.3]{HJ}), there is an infinite chain 
$(( x_t,y_t))_{t\geq 0}$ in   $(L, \precsim)$ with $(x_t,y_t)\in L_t$. This means that for every $t$, $h_t(x_{t+1})=x_t$, $h_t(y_{t+1})=y_t$, and $\langle x_t,y_t\rangle \in E(H_t)$, i.e. $\langle (x_t), (y_t) \rangle \in E(\bf{H})$. 
However, since $ p\circ h_0^{\infty}( (x_t) )=a$ and $ p\circ h_0^{\infty} ( (y_t))=b$, this contradicts that $ p\circ h_0^{\infty}$ is a homomorphism. 
\end{proof}

\begin{proof}[Proof of Theorem \ref{notuniversal}]

Let $S$ be a solenoid  embedded in $|\mathbb{G}|$. Write $S$ as $\iLim\{X_n=S^1, u_n\}$, where  $u_n\colon X_{n+1}\to X_n$ is given by $u_n(e^{it})=e^{i l_n t}$, for some $l_n\in\mathbb{N}\setminus\{ 0,1\}$.
Then by Theorems \ref{soltograph-main} and \ref{soltograph} there is an almost graph-solenoid $\mathbb{S}\subseteq\mathbb{G}$ such that $\pi(\mathbb{S})=S$, where $\pi\colon \mathbb{G}\to |\mathbb{G}|$ is the quotient map. Moreover, $\mathbb{S}$ can be  realized as $\iLim\{C_n, w_n\}$, where $C_n$'s are cycles, $w_n$'s are almost wrapping maps,  the winding number of $w_n$ is $l_n$, and $w_n$ has a proper confluent witness. If $S$ is not universal, then 
there is a prime number $P$  and an integer $i_0$ such that for every $i>i_0$ the winding number of $w_{i_0}^i$ is not divisible by $P$. 
Towards the contradiction suppose that  there is an embedding
$\tau\colon \mathbb{S}\to \mathbb{G}
(=\iLim\{F_n,\alpha_n\}$), where $\{F_n,\alpha_n\}$ is a Fra\"{\i}ss\'{e} sequence for $\mathcal G$. Therefore $\tau(\mathbb{S})=\iLim\{S_n,\alpha_n|_{S_{n+1}}\}$, where $$V(S_n)=V(\alpha_n^\infty(\tau(\mathbb{S})))$$ and $$\langle x,y \rangle \in E(S_n) \text{ iff }\langle x,y\rangle\in E(F_n).$$

 We  construct
 increasing sequences $(n_i)$ and $(m_i)$ and surjective homomorphisms $p_i\colon S_{n_i}\to  C_{m_i}$, $q_i\colon C_{m_{i+1}}\to S_{n_i}\subseteq F_{n_i}$ with $p_i\circ q_i= w^{m_{i+1}}_{m_i}$ and 
$q_i\circ p_{i+1}=\alpha^{n_{i+1}}_{n_i} |_{S_{n_{i+1}}}$.

Suppose that we have already constructed $(n_j)$, $(m_j)$, $p_j$, $q_j$, for $j<i$. We now construct $p_i$.
Apply Lemma \ref{laddersur} to $f_1=w^\infty_{m_{i}}$ and get $k$ and a surjection 
$p\colon S_k\to C_{m_i}$ such that
$f_1=p\circ\alpha^\infty_{k}\circ \tau$.
Apply Lemma \ref{phomomorph} to $\{S_i, \alpha_i|_{S_{i+1}}\}_{i\geq k}$ and $p$. We can do this since $p\circ \alpha^\infty_k\circ\tau=w^\infty_k$, $w^\infty_k$ is a homomorphism, $\tau $ is an embedding, and hence $p\circ \alpha^\infty_k|_{\tau(\mathbb{S})}$ is a homomorphism. Let $N$ be as in the conclusion of Lemma  \ref{phomomorph} and let $n_{i}=N$. We let $p_i= p\circ \alpha_k^N|_{S_N}$.

Now apply Lemma \ref{laddersur} to $f_2=\alpha^\infty_{n_{i}}\circ\tau$ and get $m_{{i+1}}$ and 
 a surjection $q_{i}\colon C_{m_{i+1}}\to S_{n_i}$ such that 
$f_2=q_{i}\circ w^\infty_{m_{i+1}}$.
We show that each $q_i$ is a homomorphism. Indeed, let $\langle a,b\rangle \in E(C_{m_{i+1}})$. Then there is $\langle  {\bf a}',{\bf b}'\rangle \in E(\mathbb{S})$ with 
 $w^\infty_i({\bf a}')=a$ and  $w^\infty_i({\bf b}')=b$.
 Then $q_i(a)= \alpha^\infty_{n_i}\tau({\bf a}')$, $q_i(b)= \alpha^\infty_{n_i}\tau({\bf b}')$, and 
$\langle \alpha^\infty_{n_i}\tau({\bf a}'),\alpha^\infty_{n_i}\tau({\bf b}')\rangle \in E(F_{n_i})$, as needed.

Without loss of generality, we have $m_0>i_0$, i.e. for every $i>m_0$ the winding number of $w_{m_0}^i$ is not divisible by $P$.
Denote $C=C_{m_0}$, $D=C_{m_1}$, and $X=S_{n_0}$. Therefore we have homomorphisms $p_0\colon X\to C$ and $q_0\colon D\to X$ such that
$w^{m_1}_{m_0}=p_0\circ q_0$.

We will show first that 
for every finite connected graph $Z$ and for every confluent epimorphism
$f\colon Z\to X$  there exist $j\geq 1$ and a homomorphism $r\colon C_{m_j}\to Z$ with $q_0\circ w^{m_j}_{m_1} = f\circ r$. Then we will construct $Z$ and $f$ that do not have that property, which will contradict the existence of the embedding $\tau$. 

Let $Z$ be a finite connected graph and let $f\colon Z\to X$ be a confluent epimorphism.
By Lemma~\ref{extension} we can find a finite connected graph $Z^*$ such that  $Z\subseteq Z^*$ as well as a confluent epimorphism
$f^*\colon Z^*\to F_{n_0}$ such that $f^*|_Z=f$ and $(f^*)^{-1}(X)=Z$. Since $\{F_n, \alpha_n\}$ is a Fra\"{\i}ss\'{e} sequence we find $j$ and a confluent epimorphism $g^*\colon F_{n_j}\to Z^*$ such that 
$f^*\circ g^*= \alpha^{n_j}_{n_0}$. Let $g=g^*|_{S_{n_j}}$. Note that $f^*\circ g^*= \alpha^{n_j}_{n_0}$ together with $(f^*)^{-1}(X)=Z$ imply that the range of $g$ is contained in $Z$,
and hence $f\circ g= \alpha^{n_j}_{n_0}|_{S_{n_j}}$.
As $g^*$ is a homomorphism, so is $g$, and therefore  $r\colon C_{m_{j+1}}\to Z$ given by $r=g\circ q_j$ is a homomorphism.
This implies $f\circ g\circ q_j= (\alpha^{n_j}_{n_0}|_{S_{n_j}})\circ q_j$, i.e. $f\circ r=q_0\circ w^{m_{j+1}}_{m_1} $.

We now construct  a  finite connected graph $Z$ and a confluent epimorphism
$f\colon Z\to X$ such that for every $j\geq 1$ there is no homomorphism $r\colon C_{m_j}\to Z$ for which $q_0\circ w^{m_j}_{m_1} = f\circ r$ holds.  
Fix some $a,b\in C$ with $a<_C b$, and let $A=p_0^{-1}(a)$ and $B=p_0^{-1}(b)$. 
Consider  $P$ many copies of $X$, and name them as $X_1, X_2, \ldots, X_{P}$.
Let $A_i\subseteq X_i$ be the copy of $A$ and similarly let $B_i\subseteq X_i$ be the copy of $B$. Take the disjoint union of all $X_1, X_2, \ldots, X_{P}$. Remove all edges between a vertex in $A_i$ and a vertex in $B_i$, $i=1,2,\ldots, P$. Add corresponding edges between a vertex in $A_i$ and a vertex in $B_{i+1}$, $i=1,2,\ldots, P$, where $B_{P+1}=B_1$.
Call the obtained structure $Z'$  and let $f'\colon Z' \to X$ be the obvious projection map.
Then $f'$ is confluent. Let $Z$  be any  component of $Z'$ and let $f=f'|_Z$. Then $f\colon Z\to X$ is confluent by Proposition \ref{confluent-restrictions} and is as required. 
Indeed suppose there exist $j\geq 1$ and a homomorphism $r\colon C_{m_j}\to Z$ with $q_0\circ w^{m_j}_{m_1} = f\circ r$.
 If $M$ is a proper component of $(w_{m_0}^{m_j})^{-1}(C\setminus\{b\})$, then there is $i$ such that $r(M)\subseteq X_i$. 
If $M$ and $M'$ are proper components of $(w_{m_0}^{m_j})^{-1}(C\setminus\{b\})$ and $M'$ is proper component adjacent to $M$, there is $i$ such that $r(M)\in X_i$ and $r(M')\in X_{i+1}$.
These remarks together with  Observation \ref{adjac} imply that the winding number of $w_{m_0}^{m_j}$ is divisible by $P$, which gives the desired contradiction.

\end{proof}

\subsection{Points in $|\mathbb{G}|$ that do not belong to a solenoid.  }

In this subsection we accomplish the following result.
\begin{theorem}\label{mainsol}
The set of points in $|\mathbb{G}|$ that do not belong to a solenoid is dense.
\end{theorem}

We first present a construction in which given a finite connected graph $B$ and an increasing sequence $\emptyset\neq X_0 \subsetneq X_1\subsetneq\ldots\subsetneq X_{n-1}\subsetneq X_n=B$ of connected subgraphs of $B$ we obtain a finite connected graph $C$ containing copies of the graphs $X_i$. An important feature of $C$ will be that if the vertices in each of these copies of an $X_i$ in $C$ are identified to a single vertex then the resulting graph will be a tree. Then, using this construction repeatedly, we obtain a Fra\"{\i}ss\'e sequence $\{K_n, \beta_n\}$ for $\mathcal{G}$
together with a vertex ${\bf x}=(x_n) \in \mathbb G$ with $x_n\in K_n$ having the property we call star adjacently disconnecting. We show in Theorem~\ref{mainfund} that the set of star adjacently disconnecting vertices is dense in $\mathbb G$ and in Theorem~\ref{almgrasol} that star adjacently disconnecting vertices do not belong to an almost graph-solenoid. To finish, we only need to transfer these results to the topological realization of $\mathbb G$. 

\bigskip

\noindent {\bf{Construction of $C$.}}

Given a finite connected graph $B$ and $\emptyset\neq X_0 \subsetneq X_1\subsetneq\ldots\subsetneq X_{n-1}\subsetneq X_n=B$, an increasing sequence of connected subgraphs of $B$, we construct a finite connected graph $C$ as the union of disjoint copies of the graphs $X_n$, such that some of the $X_n$'s are connected by edges. We will call the graph $C$ the unfolding of $B$ with respect to the sequence $(X_i)$. If we identify the vertices of each $X_i$ in $C$ the resulting quotient space will be a rooted tree, $T$. The vertex $X_0$ will occur exactly one time and will be the root of $T$. We call $X_0$ the kernel. If $X_i$ and $X_j$ are vertices in $T$ such that $X_j$ is a successor in $T$ of $X_i$ then $j>i$. A vertex $X_i$ is an end-vertex of $T$ if and only if $i=n$.

We will iteratively construct finite connected graphs $Y_0\subseteq Y_1\subseteq\ldots\subseteq~Y_n$ and then we let $C=Y_n$.

\underline{Step 0:} Take $Y_0=X_0$.

\underline{Step 1:}   
We construct $Y_1$. For every $x\in Y_0$
and $t\in B\setminus X_0$ with $\langle x,t\rangle\in E(B)$ 
take  $X^{x,t}$ to be an isomorphic copy of $X_{i+1}$, where $i$ is such that $t\in X_{i+1}\setminus X_i$ and fix an isomorphism between $X_{i+1}$ and $X^{x,t}$.
Denote the copy of $t$ in $X^{x,t}$ by $\tau(X^{x,t})$. So $t$ is in $X_{i+1}$ and $\tau(X^{x,t})$ is the corresponding vertex in $X^{x,t}$ with respect to the fixed isomorphism between the two graphs.
Then $Y_1$ is the graph $Y_0$ together with all of the graphs $X^{x,t}$ for $x\in Y_0$ and corresponding vertices $t$ and with  edges joining $x\in Y_0$ with $\tau(X^{x,t})$.

\underline{Step $1<k<n$:}
We have already constructed $Y_0\subseteq Y_1\subseteq\ldots\subseteq Y_k$.

To construct $Y_{k+1}$ note that $Y_k\setminus Y_{k-1}$ is the disjoint union of isomorphic copies of the graphs $X_j$ for various $j \ge k-1$. 
For a vertex $x\in Y_k\setminus Y_{k-1}$, denote by  $\psi_x$ the isomorphism from $X_j$ to the isomorphic copy of $X_j$ that contains $x$. If $t \in B\setminus X_j$ is such that $\langle \psi^{-1}_x(x),t\rangle \in E(B)$ take $X^{x,t}$ to be an isomorphic copy of the graph $X_{i+1}$, where $i$ is such that $t\in X_{i+1}\setminus X_i$. Denote, in the same manner as in Step~1, the copy of $t$ in $X^{x,t}$ by $\tau(X^{x,t})$.
  Then $Y_{k+1}$ is the graph $Y_k$ together with all of the graphs $X^{x,t}$ and edges joining $x$ with $\tau(X^{x,t})$ for $x\in Y_k\setminus Y_{k-1}$.

Precisely, let $Y_{k+1}$ be the following graph.
Take

$$V(Y_{k+1})=V(Y_k)\cup \bigcup_{t,x} V(X^{x,t})$$
and

\blu{$$E(Y_{k+1})=E(Y_k)\cup \bigcup_{t,x} \bigl (E(X^{x,t})
\cup \{\langle x,\tau(X^{x,t}) \rangle\}\bigr ),$$

where $x\in Y_k\setminus Y_{k-1}$ and $t \in B\setminus X_j$ is such that $\langle \psi^{-1}_x(x),t\rangle \in E(B)$, for some $j\ge k-1$.}

Finally let $C=Y_n$ and let $g\colon C\to B$ be the obvious projection, that is, if $x \in C$ then $x \in \psi(X_j)$ where $\psi$ is an isomorphism from some $X_j$ onto a subgraph of $C$, we let $g(x)=\psi^{-1}(x)$. Then $g$ is a confluent epimorphism by Theorem \ref{confluent-edges}.  
Call the graph $C$ the {\it unfolding} of $B$ with respect to the sequence $(X_i)$. 
The graph  $Y_0$ will be called the {\it kernel} of $C$.

\begin{example}

Let $B$ be the graph having vertices $a,b,c,d,e$ and edges $\langle a,b\rangle$, $\langle a,c\rangle$, $\langle a,d\rangle$, $\langle b,c\rangle$, $\langle c,d\rangle$, $\langle c,e\rangle$, $\langle d,e\rangle$. Let $X_0$ be the subgraph having vertices $a,b$, $X_1$ the subgraph having vertices $a,b,c$, $X_2$ the subgraph having vertices $a,b,c,d$, and $X_3 =B$.  The graphs $Y_0$, $Y_1$ and part of $Y_2$ are shown below.

\bigskip

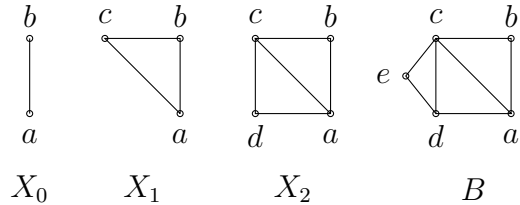
\begin{center}
    \begin{tikzpicture}
    
   
   \draw (0,0) -- (0,1);
   \node at (0,-0.3) {$a$};
   \draw (0,0) circle (0.04);
   \node at (0,1.3) {$b$};
   \draw (0,1) circle (0.04);
   \node at (0,-1) {$X_0$};
   
   
   \draw (2,0) -- (2,1) -- (1,1) -- (2,0);
  \node at (2,-0.3) {$a$};
   \draw (2,0) circle (0.04);
    \node at (2,1.3) {$b$};
   \draw (2,1) circle (0.04);
    \node at (1,1.3) {$c$};
   \draw (1,1) circle (0.04);
   \node at (1.5,-1) {$X_1$};
   
   
   \draw (4,0) -- (4,1) -- (3,1) -- (4,0);
  \node at (4,-0.3) {$a$};
   \draw (4,0) circle (0.04);
    \node at (4,1.3) {$b$};
   \draw (4,1) circle (0.04);
    \node at (3,1.3) {$c$};
   \draw (3,1) circle (0.04);
   \draw (3,1) -- (3,0) -- (4,0);
  \node at (3,-0.3) {$d$};
   \draw (3,0) circle (0.04);

   \node at (3.5,-1) {$X_2$};
  

   \draw (6.4,0) -- (6.4,1) -- (5.4,1) -- (5.4,0) -- (6.4,0) -- (5.4,1);
  \node at (6.4,-0.3) {$a$};
   \draw (6.4,0) circle (0.04);
    \node at (6.4,1.3) {$b$};
   \draw (6.4,1) circle (0.04);
    \node at (5.4,1.3) {$c$};
   \draw (5.4,1) circle (0.04);
  \node at (5.4,-0.3) {$d$};
   \draw (5.4,0) circle (0.04);
 \draw (5.4,0) -- (5,0.5) -- (5.4,1);
   \node at (4.7,0.5) {$e$};
   \draw (5.0,0.5) circle (0.04);

   \node at (5.9,-1) {$B$};

    \end{tikzpicture}
\end{center}

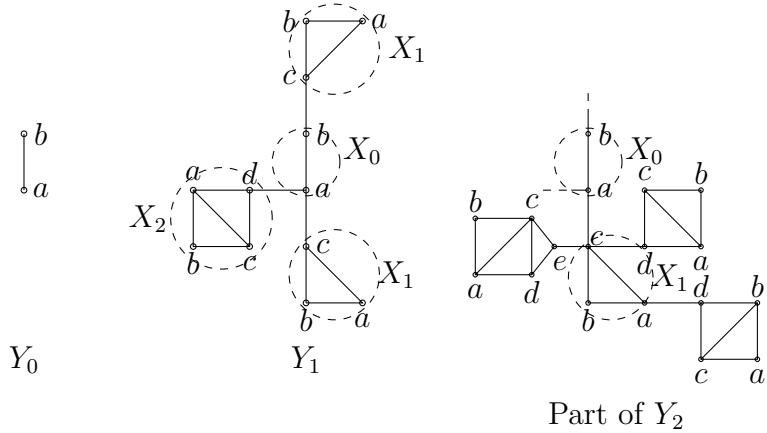
\begin{center}
\begin{tikzpicture}[scale=0.75]

\draw (0,2) -- (0,3);
\node at (0.3,2) {$a$};
\draw (0,2) circle (0.05);
\node at (0.3,3) {$b$};
\draw (0,3) circle (0.05);

\draw (5,0) -- (5,5);
\draw (3,2) -- (5,2);
\draw (3,2) -- (3,1) -- (4,1) -- (4,2);
\draw (3,2) -- (4,1);
\draw (5,0) -- (6,0) -- (5,1);
\draw (5,4) -- (6,5) -- (5,5);
\node at (5,-0.3) {$b$};
\draw (5,0) circle (0.05);
\node at (6,-0.3) {$a$};
\draw (6,0) circle (0.05);
\node at (5.3,1) {$c$};
\draw (5,1) circle (0.05);
\node at (5.3,2) {$a$};
\draw (5,2) circle (0.05);
\node at (5.3,3) {$b$};
\draw (5,3) circle (0.05);
\node at (4.7,4) {$c$};
\draw (5,4) circle (0.05);
\node at (4.7,5) {$b$};
\draw (5,5) circle (0.05);
\node at (6.3,5) {$a$};
\draw (6,5) circle (0.05);
\node at (3,2.3) {$a$};
\draw (3,2) circle (0.05);
\node at (4,2.3) {$d$};
\draw (4,2) circle (0.05);
\node at (3,0.7) {$b$};
\draw (3,1) circle (0.05);
\node at (4,0.7) {$c$};
\draw (4,1) circle (0.05);

\draw [dashed] (5,2.5) circle (0.6);
\node at (6,2.7) {$X_0$};

\draw [dashed] (5.5,4.5) circle (0.8);
\node at (6.8,4.5) {$X_1$};

\draw [dashed] (5.5,0.5) circle (0.8);
\node at (6.6,0.5) {$X_1$};

\draw [dashed] (3.5,1.5) circle (0.9);
\node at (2.2,1.5) {$X_2$};

\node at (0,-1) {$Y_0$};
\node at (5,-1) {$Y_1$};


\draw (10,3.3) -- (10,0) -- (11,0) -- (10,1) -- (12,1) -- (12,2) -- (11,2) -- (11,1);
\draw (12,1) -- (11,2);
\draw (11,0) -- (13,0) -- (13,-1) -- (12,-1) -- (12,0);
\draw (12,-1) -- (13,0);
\draw (10,1) -- (9.4,1) -- (9,0.5) -- (8,0.5) -- (8,1.5) -- (9,1.5) -- (9.4,1);
\draw (9,0.5) -- (9,1.5) -- (8,0.5);
\draw (9.7,2) -- (10,2);
\draw [dashed] (9.2,2) -- (9.7,2);
\draw [dashed] (10,3.3) -- (10,3.8);

\node at (10.3,3) {$b$};
\draw (10,3) circle (0.04);
\node at (8,0.2) {$a$};
\draw (8,0.5) circle (0.04);
\node at (8,1.8) {$b$};
\draw (8,1.5) circle (0.04);
\node at (9,0.2) {$d$};
\draw (9,0.5) circle (0.04);
\node at (9,1.8) {$c$};
\draw (9,1.5) circle (0.04);
\node at (9.5,0.7) {$e$};
\draw (9.4,1) circle (0.04);
\node at (10.15,1.15) {$c$};
\draw (10,1) circle (0.04);
\node at (10.3,2) {$a$};
\draw (10,2) circle (0.04);
\node at (10,-0.3) {$b$};
\draw (10,0) circle (0.04);
\node at (11,-0.3) {$a$};
\draw (11,0) circle (0.04);
\node at (12,0.3) {$d$};
\draw (12,0) circle (0.04);
\node at (13,0.3) {$a$};
\draw (13,0) circle (0.04);
\node at (12,-1.3) {$c$};
\draw (12,-1) circle (0.04);
\node at (13,-1.3) {$b$};
\draw (13,-1) circle (0.04);
\node at (11,0.7) {$d$};
\draw (11,1) circle (0.04);
\node at (12,0.7) {$a$};
\draw (12,1) circle (0.04);
\node at (11,2.3) {$c$};
\draw (11,2) circle (0.04);
\node at (12,2.3) {$b$};
\draw (12,2) circle (0.04);

\draw [dashed] (10,2.5) circle (0.6);
\node at (11,2.7) {$X_0$};

\draw [dashed] (10.4,0.45) circle (.75);
\node at (11.45,0.45) {$X_1$};

\node at (10.5,-2) {Part of $Y_2$ };

\end{tikzpicture}

\end{center}

\end{example}

Let $A$ be a connected finite graph and let $T$ be a connected subgraph of $A$.  We say that $T\subseteq A$ is  {\it adjacently disconnecting  $A$}
if for every $x\neq y\in A\setminus T$, adjacent to $T$, any path joining $x$ and $y$ intersects $T$.
In other words, $T$ is adjacently disconnecting $A$ if any homomorphism from a finite arc into $A$, which maps one end-vertex of the arc to $x$ and the other to $y$, is such that the intersection of the image of the arc with $T$ is nonempty.

\begin{lemma}\label{adjadis}
 Let $B$ be a finite connected graph and let 
 $\emptyset\neq X_0\subsetneq\ldots\subsetneq X_{n-1}\subsetneq X_n=B$ be 
 an increasing sequence of connected subgraphs of $B$, and let $g\colon C\to B$ be the epimorphim obtained in the construction above. For every $k=0,1,\ldots,n$, let  $C_k$ denote the component of $g^{-1}(X_k)$ that contains the kernel of $C$. Then $C_k$ is adjacently disconnecting  $C$.
\end{lemma}

\begin{proof}
The set of vertices $V(C)$ is the disjoint union of multiple copies of $V(X_0), V(X_1),\ldots, V(X_{n})$. 
For each $k= 0,1,\ldots, n$ the $C_k$ is the  
subgraph of $C$ whose set of vertices is the union of all disjoint copies of $ V(X_0), V(X_1),\ldots, V(X_k)$ that were attached at some stage of the construction. 
Finally, note that if we were to identify the vertices in each of the disjoint copies of the $V(X_i)$'s in $C$ to single vertices, then the resulting graph is a tree. 
\end{proof}


Given a finite graph $G$ and a vertex $x\in G$, the {\it star} of $x$, denoted st$(x)$, is
${\rm st}(x)=\{x\}\cup\{y\in G\colon y \text{ is adjacent to } x\}$.

\blu{

Let $\bf H$ be a topological graph. A vertex ${\bf x}\in O_{\bf H}$ is said to be a {\it star adjacently disconnecting vertex} if  there exists 
$\{H_n,\beta_n\}$, where each $\beta_n\colon H_{n+1}\to H_n$ is a confluent epimorphism between finite connected graphs such that ${\bf H} =\iLim\{H_n,\beta_n\}$, and the following holds for ${\bf x}=(x_n)$:
 for every  $m$ and $n$ with $m<n$ the component  $C_{m,n}$ of $(\beta_m^n)^{-1}({\rm st}(x_m))$ that contains $x_n$ is adjacently disconnecting  $H_n$.
In that case we say that $\{H_n,\beta_n\}$ {\it witnesses} that ${\bf x}$ is a star adjacently disconnecting vertex for ${\bf H}$. 
 }
 
\begin{theorem}\label{mainfund}

The set of star adjacently disconnecting vertices \blu{witnessed by a Fra\"{\i}ss\'e sequence} is dense in $\mathbb G$.

\end{theorem}

\begin{proof}

Let  ${\bf U}$ be a nonempty open set in $\mathbb G$.
Let $\{F_n,\alpha_n\}$ be a Fra\"{\i}ss\'e sequence for $\mathcal G$, where $F_1$ is not a singleton and there is $z \in F_1$ such that $(\alpha_1^\infty)^{-1}(z)\in {\bf U}$. We will construct an inverse sequence $\{K_n,\beta_n\}$ which is also a Fra\"{\i}ss\'e sequence for $\mathcal G$, and sequences of confluent epimorphisms $g_n$ and $f_n$. Each  $K_n$ is the unfolding of $F_{k_n}$ with respect to a prescribed increasing sequence of connected graphs $(X_i)$,  $g_n\colon K_n \to F_{k_n}$ is the mapping defined in the unfolding construction, $f_n\colon F_{k_{n+1}} \to K_n$ is obtained using that $\{F_n,\alpha_n\}$ is a Fra\"{\i}ss\'e sequence, and $\beta_n=f_n \circ g_{n+1}$. We also construct a vertex ${\bf x} \in \mathbb G$ where ${\bf x} =(x_n)$ with $x_n \in K_n$ and $\beta_n(x_{n+1})=x_n$.

Let  $K_1$ be the unfolding of $F_1$ with respect to the sequence $(\{z\}, F_1)$, and $g_1 \colon K_1 \to F_1$ be the mapping from the unfolding construction.
Let $x_1\in K_1$ be the unique vertex in the kernel of $K_1$. Create a graph $G$ by replacing each edge $\langle x_1,y\rangle \in E(K_1)$ by the two edges $\langle x_1,y_1\rangle$ and $\langle y_1,y\rangle$. Then $g\colon G \to K_1$ that maps each $y_1$ to $x_1$ and is the identity otherwise is a confluent epimorphism.  Let $k_1 = 1$ and since $\{F_n, \alpha_n\}$ is a Fra\"{\i}ss\'e sequence there is a $k_2$ and a confluent epimorphism $h\colon F_{k_2} \to G$ such that $g_1\circ g \circ h = \alpha^{k_2}_{k_1}$. Let $f_1= g \circ h$. Note that for some $v \in f_1^{-1}(x_1)$  any vertex in $F_{k_2}$ that is adjacent to $v$ is mapped by $f_1$ to $x_1$.

Suppose that $\{K_n,\beta_n\}$ and  $x_n\in K_n$ with  $\beta_{n-1}(x_n) = x_{n-1}$,   $n\leq N$, are constructed. 
We also have  an increasing  sequence $(k_n)$, $n\leq N+1$, and maps
$f_n\colon F_{k_{n+1}}\to K_n$,  $g_n\colon K_{n}\to F_{k_n}$, $n\leq N$, with $\beta_n=f_n\circ g_{n+1}$ for $n<N$,
and $\alpha^{k_{n+1}}_{k_{n}}=g_{n}\circ f_{n}$ for $n\leq N$.
We now construct $K_{N+1}$, $x_{N+1}\in K_{N+1}$, $f_{N+1}$, and  $g_{N+1}$
with the corresponding properties for $N+1$ instead of $N$.
Then we let $\beta_N=f_N\circ g_{N+1}$.

Let $X_{N+1}=F_{k_{N+1}}$ and 
pick a vertex  $v_{N+1}\in f_N^{-1}(x_N)$ such that all vertices adjacent to $v_{N+1}$ are also mapped to $x_N$, and denote 
$X_0=\{v_{N+1}\}$. Let 
$X_1$ be the component of $f_N^{-1}({\rm st}(x_N))$ that contains $X_0$.
Let for $1\leq i<N$, $X_{i+1}$ be the  component of $(\beta_{N-i}^{N} \circ f_N)^{-1}({\rm st}(x_{N-i}))$ that contains $X_0$.
 Let $K_{N+1}$ be the unfolding of $F_{k_{N+1}}$ with respect to the sequence $(X_i)$, let  $g_{N+1}\colon K_{N+1}\to F_{k_{N+1}}$ be the resulting confluent epimorphism, and let $x_{N+1}$ be the unique vertex in the kernel of $K_{N+1}$.
 Take $\beta_N=f_N\circ g_{N+1}$. 
Because $\{F_n,\alpha_n\}$ is a Fra\"{\i}ss\'e sequence there exists $k_{N+2}$ and a confluent epimorphism $f_{N+1}\colon F_{k_{N+2}}\to K_{N+1}$ such that $\alpha^{k_{N+2}}_{k_{N+1}}=g_{N+1}\circ f_{N+1}$.  
Similarly as in the first step
we can and we do assume that some vertex in $F_{k_{N+2}}$ mapped to $x_{N+1}$ by $f_{N+1}$ is such that all vertices adjacent to it are also mapped to $x_{N+1}$.

The sequence $\{K_n,\beta_n\}$ is a Fra\"{\i}ss\'e sequence, which follows from the fact that $\{F_n,\alpha_n\}$ is a Fra\"{\i}ss\'e sequence and diagram chasing.

This finishes the construction. Now we check that $\{K_n,\beta_n\}$ and ${\bf x}=(x_n)$
have the required properties.

By the construction, $\beta^{n+1}_n({\rm st}(x_{n+1}))=x_n$ for all $n$, which implies ${\bf x}\in O_{\mathbb{G}}$.
Furthermore, for every $m<n$, the $X_{n-m}\subseteq F_{k_{n}}$ is the  component of $(\beta_{m}^{n-1} \circ f_{n-1})^{-1}({\rm st}(x_m))$ that contains $X_0=\{v_n\}\subseteq F_{k_{n}} $, $\beta_m^n=\beta_m^{n-1}\circ f_{n-1}\circ g_n$, and $g_{n}(x_{n})=v_n$. 
Therefore $C_{m,n}$ is the component of $g_n^{-1}(X_{n-m})$ that contains $x_n$, which by  Lemma \ref{adjadis}, is adjacently disconnecting ~$K_n$. Since $\{x_n\}=g_n^{-1}(X_0)$, Lemma \ref{adjadis} additionally implies that $\{x_n\}$  
is adjacently disconnecting  $K_n$.

Finally, the automorphism between $\mathbb{G}=\varprojlim \{F_n, \alpha_n\}$ and
$\mathbb{G}=\varprojlim \{K_n, \beta_n\}$ given by $(f_n)$ and $(g_n)$ takes ${\bf x}$ into ${\bf U}$.
    The image of ${\bf x}$ by any automorphism, in particular by this one, is again star adjacently disconnecting.

\end{proof}

\begin{theorem}\label{almgrasol}
A star adjacently disconnecting vertex ${\bf x}\in \mathbb G$ \blu{witnessed by a Fra\"{\i}ss\'e sequence} does not belong to an almost graph-solenoid.

 \end{theorem}

 \begin{proof}
Let $\{K_n,\beta_n\}$ witness that ${\bf x}$ is a  star adjacently disconnecting vertex.
Let $\mathbb{S}$ be an almost graph-solenoid and write
$\mathbb{S}=\iLim\{C_i,w_i\}$, where $C_i$'s are cycles and  $w_i\colon C_{i+1}\to C_i$ are almost wrapping maps, which have proper confluent witnesses.
Towards the contradiction suppose that  there is an embedding
$\tau\colon \mathbb{S}\to \mathbb{G}
(=\iLim\{K_n,\beta_n\}$)  such that ${\bf x}\in \tau(\mathbb{S})$. Therefore $\tau(\mathbb{S})=\iLim\{S_n,\beta_n|_{S_{n+1}}\}$, where $V(S_n)=V(\beta_n^\infty(\tau(\mathbb{S})))$ and $\langle x,y \rangle \in E(S_n)$ iff $\langle x,y\rangle\in E(K_n)$.
As in the proof of Theorem \ref{notuniversal}  
we obtain 
 increasing sequences $(n_i)$ and $(m_i)$ and surjective homomorphisms  $p_i\colon S_{n_i}\to  C_{m_i}$, $q_i\colon C_{m_{i+1}}\to S_{n_i}\subseteq K_{n_i}$ with $p_i\circ q_i= w^{m_{i+1}}_{m_i}$ and 
$q_i\circ p_{i+1}=\beta^{n_{i+1}}_{n_i} |_{S_{n_{i+1}}}$.

Since almost wrapping maps are closed under compositions (see Lemma \ref{compawrap}), $w^{m_{i+1}}_{m_i}$ is an almost wrapping map. To simplify the notation, we assume that $m_i=i$.
Without loss of generality, we may assume that $C_1$ has at least 6 vertices.
Let  $A=(q_1)^{-1}({\rm st}(x_{n_1}))=(q_1)^{-1}({\rm st}(x_{n_1})\cap S_{n_1})$.

\medskip 

\noindent {\bf{Claim 1.}}
 There is an arc in $C_2\setminus A$ that consists of at least two vertices.

\begin{proof}[Proof of  Claim 1]

Suppose towards the contradiction that any arc in  $C_2\setminus A$ is a singleton.
Then since $w_1$ is a homomorphism, any vertex  $w_1(y)$ for a $y\in C_2\setminus A$, is adjacent to $w_1(A)\subseteq p_1({\rm st}(x_{n_1})\cap S_{n_1})\subseteq {\rm st}(p_1(x_{n_1}))$. 
But this is impossible since ${\rm st}(p_1(x_{n_1}))$ is an arc with three vertices and $C_1$ has at least 6 vertices.

\end{proof}

 Denote $\beta=\beta^{n_2}_{n_1}$,
 let $D$ be the component of $\beta^{-1}({\rm st}(x_{n_1}))$ containing $x_{n_2}$, and let $E=(q_2)^{-1}(D)$.  Since ${\bf x}$ is a star adjacently disconnecting vertex, $D$ is adjacently disconnecting $K_{n_2}$.
Let $\{I_i=[k_i,l_i], J_i=[u_i,v_i], i\leq m\}$ be the partition of $C_3$ into  arcs such that
$E=\bigcup_i J_i$, $l_i$ is adjacent to $u_i$,  $v_{i}$ is adjacent to $k_{i+1}$, and $v_m$ is adjacent to $k_1$.
Write also $A$ as the disjoint union of arcs $M_j=[a_j, b_j]$ in a way that $b_j\in [a_j, a_{j+1}]$.

Let us make several observations.
\begin{enumerate}
\item  For every $i$, $w_2(J_i)\subseteq A$, in particular, $w_2(u_i), w_2(v_i)\in A$. Indeed,  from $\beta\circ q_2=q_1\circ w_2$ it follows that $w_2(E)\subseteq A$.

\item  For every $i$, $w_2(k_i), w_2(l_i)\notin A$ and $w_2(k_i)=w_2(l_i)$. Indeed, suppose towards the contradiction that say $w_2(l_i)\in A$. That implies by $\beta\circ q_2=q_1\circ w_2$ that $\beta\circ q_2(l_i)\in {\rm st}(x_{n_1})$, hence $q_2(l_i)\in \beta^{-1}({\rm st}(x_{n_1}))$.
By the choice of the partition $u_i$ is adjacent to $l_i$, therefore since $q_2$
is a homomorphims, $q_2(u_i)$ is adjacent to $q_2(l_i)$. This along with $q_2(l_i)\in \beta^{-1}({\rm st}(x_{n_1}))$ and $q_2(u_i)\in D$ implies that $q_2(l_i)\in D$. A contradiction.

Since $D$ is adjacently disconnecting, it holds $q_2(k_i)=q_2(l_i)$, and hence   $w_2(k_i)=p_2( q_2(k_i))= p_2(q_2(l_i))=w_2(l_i)$. 

\item  For every $i$, $w_2(l_i)$ is adjacent to $w_2(u_i)$, and $w_2(v_i)$ is adjacent to $w_2(k_{i+1})$. This simply follows from the fact that $w_2$ is a homomorphism.

\item  For every $i$ there is $j$ such that $w_2(J_i)=\{a_j\}$ or  $w_2(J_i)=\{b_j\}$ or  $w_2(J_i)=M_j$. Moreover, in the case $w_2(J_i)=M_j$ and $|M_j|\geq 3$, we have $a_j=w_2(u_i)\neq w_2(v_i)=b_j$.

 It follows from (1), (2) and (3) that there is $j$ such that $w_2(\{u_i, v_i\})=\{a_j\}$ or  $w_2(\{u_i, v_i\})=\{b_j\}$ or  $w_2(\{u_i, v_i\})=\{a_j, b_j\}$.
In the scenario $w_2(\{u_i, v_i\})=\{a_j, b_j\}$ we get in fact $w_2(J_i)=M_j$, as $w_2$ is a homomorphism.
 In the case $w_2(\{u_i, v_i\})=\{a_j\}$ and $b_j\neq a_j$, then by (2) and (3), we obtain $w_2(l_i)=w_2(k_{i+1})$,
 and denote $x=w_2(k_{i+1})$. Let $y\in M_j$ be such that $w_2(J_i)=[a_j,y]$, and take $z\in J_i$ with $w_2(z)=y$.
 Since $w_2$ is an almost wrapping map, by  Lemma \ref{notswaptriples} applied to $x,y$ and $z, k_{i+1}$,
we must have $y=a_j$, and hence $w_2(J_i)=\{a_j\}$.  We similarly conclude that $w_2(J_i)=\{b_j\}$ in the case when $w_2(\{u_i, v_i\})=\{b_j\}$. Moreover if $w_2(u_i)\neq w_2(v_i)$, again by Lemma \ref{notswaptriples}, this time applied to $a_j, b_j$ and $v_i, u_i$, we must have $a_j=w_2(u_i)\neq w_2(v_i)=b_j$ and not $b_j=w_2(u_i)\neq w_2(v_i)=a_j$.
\end{enumerate}

\medskip

\noindent {\bf{Claim 2.}}
There are $i_0,j_0$ such that $w_2(J_{i_0})=M_{j_0}$ and $|M_{j_0}|\geq 3$.

\begin{proof}[Proof of Claim 2]
    Take any $i_0$ such that $x_{n_2}\in q_2(J_{i_0})$. 
Take $j_0$ such that $w_2(J_{i_0})\subseteq M_{j_0}$. Then $p_2(x_{n_2})\in M_{j_0}$. Since $M_{j_0}$ is a component of $A=(q_1)^{-1}({\rm st}(x_{n_1}))$, and $q_1(p_2(x_{n_2}))=x_{n_1}$, 
both vertices adjacent to $p_2(x_{n_2})$ are also in $M_{j_0}$, which implies  $|M_{j_0}|\geq 3$. By Observation (4), we moreover have $w_2(J_{i_0})=M_{j_0}$.
\end{proof}

\bigskip

\blu{To obtain a contradiction needed to finish the proof of Theorem~\ref{almgrasol} we have to consider two cases. 

{\bf Case 1.} 
$A$ is connected, and hence $A=M=[a,b]$. Let $x\notin M$ be adjacent to $a$, and let $y\notin M$ be adjacent to $b$. By Claim 1 we have $x\neq y$. For the $J_{i_0}=[u_{i_0},v_{i_0}]$ as in Claim 2 we have (by Observation (4)) that $w_2(u_{i_0})=a$ and $w_2(v_{i_0})=b$. Therefore $w_2(k_{i_0+1})=w_2(l_{i_0+1})=y$. However, since $x\neq y$ and since $|M|\geq 3$, we must have  $w_2(v_{i_0+1})=w_2(u_{i_0+1})=b$ and $w_2(k_{i_0+2})=w_2(l_{i_0+2})=y$, etc., we continue along the cycle $C_3$ until we reach $J_{i_0}$ again. But then we get $w_2(v_{i_0})=w_2(u_{i_0})=b$, which gives a contradiction.

{\bf Case 2.}  $A$ is disconnected, and hence the complement consists of at least two arcs. Then, by Claim 1, at least one of the arcs in $C_2\setminus A$ has at least two vertices.  Hence the set $Z=\{z\in C_2\colon z \text{ is adjacent to } A\}$ has at least three vertices. We define a cycle $K$ taking $V(K)=Z$, and let $z_1<_K z_2$ iff in $C_2$ we have $(z_1,z_2)\cap Z=\emptyset$.

Let $L$ be the cycle on $m$ vertices, which we denote as $I_1,\ldots, I_m$, where $m$ is, as above, the number of arcs $I_i$ in the decomposition of $C_3$ and we let $I_i <_L I_{i+1}$ for $i=1,\ldots, m$, where $I_{m+1}=I_1$.  Let $u\colon L\to K$ be the map that takes $I_i=[k_i,l_i]$ to $w_2(k_i)=w_2(l_i)$. Then $u$ is a homomorphism by Observations (1)-(4).

Let $a$ and $b$, $a\neq b$, be end-vertices of an arc in $C_2\setminus A$ which has at least two vertices.
 
{\bf (i)} We claim  $\langle a,b\rangle$, as an edge in $K$, is not an image of an edge in~$L$.

Indeed, suppose this is not the case and fix $i$ such that $u(I_i)=a$ and $u(I_{i+1})=b$ (or  $u(I_i)=b$ and $u(I_{i+1})=a$).
In $C_3$ the arc  $J_i$  is adjacent to both $I_i$ and $I_{i+1}$. Let $M$ be the component  of  $A$  such that $w_2(J_i)\subseteq M$.  Then  $M$ is adjacent to both $a$ and $b$. But this is impossible since we are in Case 2.

{\bf (ii)} On the other hand we claim $\langle a,b\rangle $ is  an image of an edge in~$L$.

Take $i_0$ and $j_0$ as in Claim 2. So $J_{i_0}=[u_{i_0}, v_{i_0}]$, and denote $s=u(k_{i_0+1})$ (start) and $e=u(l_{i_0})  $ (end). Clearly $s$  and $e$ are adjacent in $K$. Since $w_2$ is an almost wrapping map and $|w_2(J_{i_0})|\geq 3$, there is no $i_1\leq m$ such that $s=u(l_{i_1})$  and $e=u(k_{i_1+1})$. That is, there are no $x<_L y$ such that $u(x)=s$ and $u(y)=e$ . Since $u$ is a homomorphism and both $s$ and $e$ are in the image of $u$, we obtain that 
any edge in $K$, except possibly $\langle e,s\rangle$, is an image of an edge of $L$. Since $s$ and $e$ are in different components of $C_2\setminus A$, we have $\langle s,e\rangle \neq \langle a,b \rangle$.  Hence $\langle a,b\rangle $ is  an image of an edge in $L$.

Clearly  (i) and (ii) together give a contradiction.
}
\end{proof}

\begin{proof}[Proof of Theorem \ref{mainsol}]

The set of star adjacently disconnecting vertices
witnessed by a  Fra\"{\i}ss\'{e} sequence is dense in 
$\mathbb{G}$, and let ${\bf x}\in\mathbb{G}$ be one of them. We claim that $z=\pi({\bf x})$, where $\pi\colon\mathbb{G}\to |\mathbb{G}|$ is the projection map, does not belong to a solenoid. Suppose towards a contradiction that there is a solenoid $S$ such that $z\in S\subseteq |\mathbb{G}|$. By Lemma \ref{densesingle} there is a topological graph $\mathbb{S}\subseteq \pi^{-1}(S)$ such that  $O_\mathbb{S}$ is dense in $\mathbb{S}$ and $\pi(\mathbb{S})=S$.
Since ${\bf x}\in O_{\pi^{-1}(S)}$, we have ${\bf x}\in\mathbb{S}$. 
By Theorem \ref{soltograph}, $\mathbb{S} $ is an almost graph-solenoid.
On the other hand, by Theorem \ref{almgrasol},  ${\bf x}$ does not belong to an almost graph-solenoid. We obtain a contradiction.

\end{proof}

\section{Open problems}

As we proved in Section 9, we can distinguish points in $
|\mathbb{G}|$: some of them belong to the universal solenoid and some do not. Moreover, 
every solenoid that embeds in $|\mathbb{G}|$, is the universal solenoid.

We do not know if we can further distinguish points by the number of solenoids that contain them. For $x\in |\mathbb{G}|$ let $n_S(x)$ denote the cardinality of the set $$ 
\{S\subseteq |\mathbb{G}| \colon S \text{ is homeomorphic to a solenoid and } x\in S\}.$$

\begin{question}
What are the possible values of $n_S(x)$ for $x\in |\mathbb{G}|$? 
\end{question}

The following was suggested to us by the referee. 
\begin{question}
Solenoids may intersect in a point or in an arc. Can both situations take place for solenoids embedded in $|\mathbb{G}|$? 
\end{question}

Furthermore, we do not know whether there are continua, other than solenoids, which we can use to distinguish points in $|\mathbb{G}|$. In particular, can we distinguish points by belonging or not belonging to the pseudo-arc? In Section 6, we showed that the pseudo-arc can be embedded in~$|\mathbb{G}|$.
\begin{question}
Are there $x\in |\mathbb{G}|$ such that for any embedding $\tau \colon P\to  |\mathbb{G}|$, where $P$ is the pseudo-arc, we have $x\notin \tau(P)$?
\end{question}

The homeomorphism group $ H(|\mathbb{G}|)$ acts on $|\mathbb{G}|$ via $(h,x)
\mapsto h(x)$. 
\begin{question}
What is the cardinality of orbits under this action?
\end{question}

\begin{question}
Are all orbits dense?
\end{question}

We know that the pseudo-arc, the universal pseudo-solenoid, the  universal solenoid, and the Cantor fan, embed in 
$ |\mathbb{G}| $. On the other hand, every continuum which embeds in $ |\mathbb{G}| $ is one-dimensional and hereditary unicoherent. Not all continua with those two properties embed in $ |\mathbb{G}| $ since the universal solenoid is the only solenoid that embeds in $ |\mathbb{G}| $.

\begin{problem}
Which one-dimensional hereditary unicoherent continua can be embedded in $ |\mathbb{G}| $? 
In particular, can we embed the pseudo-circle and the solenoid of pseudo-arcs in $ |\mathbb{G}| $?
\end{problem}

We have shown many properties of the continuum $|\mathbb G|$. We 
would like to have a complete list of properties that characterize $ |\mathbb{G}| $ up to homeomorphism.
\begin{problem}

Find a topological characterization of $|\mathbb{G}|$.
\end{problem}

The homeomorphism group $ H(|\mathbb{G}|)$ becomes a separable topological group when equipped with the supremum metric.
\begin{problem}
Investigate properties of $H( |\mathbb{G}|)$, the homeomorphism group of  $|\mathbb{G}|$.
\end{problem}

Finally, in investigating the homogeneity of $|\mathbb{G}|$ we spent time thinking about the following combinatorial question, which we believe is of independent interest. 
\begin{question}
Given a finite connected graph $H$ does there exist a homogeneous finite connected graph $G$ and a confluent epimorphism $f\colon G \to H$?
\end{question}

{\bf Acknowledgements.}
We would like to thank the referee for  carefully reading  our article and many insightful comments.

\end{document}